\documentclass{article}

\usepackage{bbm}

\usepackage[all]{xy}\usepackage[latin1]{inputenc}  
\usepackage[dvips]{graphics,graphicx}
\usepackage{amsfonts,amssymb,amsmath,color,mathrsfs, amstext,bbm}
\usepackage{amsbsy, amsopn, amscd, amsxtra, amsthm,authblk, enumerate,mathbbol}
\usepackage{upref}
\usepackage{geometry}
\usepackage{bm}
\usepackage{dsfont}
\usepackage{todonotes}
\usepackage{mathtools}
\usepackage{appendix}
\usepackage{float}

\usepackage[all]{xy}\usepackage[latin1]{inputenc}        
\usepackage[dvips]{graphics,graphicx}
\usepackage{amsfonts,amssymb,amsmath,xcolor,mathrsfs, amstext}
\usepackage{amsbsy, amsopn, amscd, amsxtra, amsthm,authblk, enumerate}
\usepackage{upref}
\usepackage{geometry}
\usepackage{float}

\usepackage[colorlinks,
            linkcolor=blue,
            anchorcolor=green,
            citecolor=blue
            ]{hyperref}

\numberwithin{equation}{section}

\def\e{\varepsilon}
\def\epsilon{\varepsilon}

\newcommand{\ol}{\overline}
\newcommand{\wt}{\widetilde}
\newcommand{\wh}{\widehat}
\def\alb#1\ale{\begin{align*}#1\end{align*}}
\newcommand{\eqb}{\begin{equation}}
\newcommand{\eqe}{\end{equation}}

\DeclareMathOperator{\supp}{supp}
\DeclareMathOperator{\dist}{dist}

\newcommand{\bbC}{\mathbb{C}}
\newcommand{\bbD}{\mathbb{D}}
\newcommand{\bbE}{\mathbb{E}}
\newcommand{\bbH}{\mathbb{H}}

\newcommand{\bbN}{\mathbb{N}}

\newcommand{\bbR}{\mathbb{R}}
\newcommand{\bbP}{\mathbb{P}}
\newcommand{\bbQ}{\mathbb{Q}}
\newcommand{\bbZ}{\mathbb{Z}}

\newcommand{\cD}{\mathcal{D}}
\newcommand{\cE}{\mathcal{E}}

\newcommand{\cP}{\mathcal{P}}
\newcommand{\cG}{\mathcal{G}}

\newcommand{\cT}{\mathcal{T}}

\newcommand{\cM}{\mathcal{M}}

\newcommand{\cH}{\mathcal{H}}

\newcommand{\cV}{\mathcal{V}}

\newcommand{\fc}{\mathfrak{c}}
\newcommand{\fR}{\mathfrak{R}}
\newcommand{\fD}{\mathfrak{D}}

\newcommand{\rmL}{\mathrm{L}}

\newcommand{\rmE}{\mathrm{E}}
\newcommand{\rmQ}{\mathrm{Q}}

\newcommand{\frt}{\mathfrak{t}}
\newcommand{\fre}{\mathfrak{e}}

\newcommand{\diam}{\mathrm{diam}}
\newcommand{\area}{\mathrm{area}}
\newcommand{\disk}{\mathrm{disk}}
\newcommand{\sph}{\mathrm{sp}}

\newcommand{\SLE}{\mathrm{SLE}}
\newcommand{\carrier}{\mathrm{carrier}}

\newcommand{\IG}{\mathrm{IG}}

\newcommand{\nina}[1]{{\color{cyan}{#1}}}

\newtheorem{theorem}{Theorem}[section]
\newtheorem{lemma}[theorem]{Lemma}
\newtheorem{proposition}[theorem]{Proposition}
\newtheorem*{proposition*}{Proposition}

\newtheorem*{corollary*}{Corollary}
\newtheorem{remark}[theorem]{Remark}

\theoremstyle{definition}
\newtheorem{definition}[theorem]{Definition}
\newtheorem*{definitions*}{Definitions}

\newtheorem*{example*}{\bf Example}

\numberwithin{equation}{section}
\title{The circle packing and Riemann uniformization embedding of the mated-CRT map converges to Liouville quantum gravity}
\author{Nina Holden\thanks{Courant Institute, New York University, nina.holden@nyu.edu} \qquad   Pu Yu\thanks{Courant Institute, New York University, py628@nyu.edu}}
\date{\today}

\begin{document}

\maketitle

\begin{abstract} 
We prove that the mated-CRT map converges to Liouville quantum gravity under circle packing and the Riemann uniformization embedding. The proof is based on our recent work~\cite{HY26} on circle packing and the Riemann uniformization embedding for random planar maps in ergodic scale-free environments, comparison of circle packings in different domains, and several existing and new estimates for the mated-CRT map. 
As an intermediate step, we prove extensions of He and Schramm's result (1996) on convergence of circle packing maps to conformal maps in global uniform topology. The latter results are of independent interest and will be applied in our subsequent work~\cite{HY26a} on the convergence of spanning tree weighted planar maps to $\sqrt{2}$-LQG under circle packing and the Riemann uniformization embedding.
\end{abstract}

\section{Introduction}

\emph{Liouville quantum gravity} (LQG) was invented in Polyakov's seminal work~\cite{polyakov1981quantum}. LQG has a single parameter $\gamma\in(0,2]$, and has been conjectured to describe the scaling limit of random planar maps. The \emph{Schramm-Loewner evolution} (SLE)~\cite{Sch00} is a family of random fractal curves $\kappa>0$, and are the scaling limit of interfaces from many  natural  statistical physics models.

There are several natural ways to phrase the convergence of (decorated) random planar maps to (SLE decorated) LQG surfaces. In each of those, there are certain random planar maps that have been proved to converge to LQG surfaces:
\begin{enumerate}
 \item Convergence as mated pair of trees. As constructed in~\cite{DMS14}, given a pair of Brownian motions, one can construct a pair of continuum random trees and further a space-filling $\SLE$-decorated LQG surface. On the other hand, many planar maps are encoded by a pair of random walks, and this convergence of random walk to Brownian motion is also referred to as  \emph{peanosphere convergence}. Examples include~\cite{She16b,LSW17,BHS18,KMSW19}.
   \item Convergence in the sense of metric space (i.e., in Gromov-Hausdorff topology), where we view random planar maps with metric spaces endowed with graph distance. In~\cite{LeGall13,Mie13}, it is shown that uniform  random quadrangulations converge to the Brownian map, which is identified with $\sqrt{8/3}$-LQG by~\cite{MSBrownian1,MSBrownian2,MSBrownian3}. It is further shown in~\cite{GM16avoid} the self-avoiding walk on random quadrangulations converges to $\SLE_{8/3}$.
    \item Convergence under \emph{discrete conformal embeddings}. That is, one draws the planar map on the plane using conformal embeddings (e.g., circle packing, Riemann uniformization, etc.), and counts the number of vertices in each domain $D\subseteq\bbC$. For the mated-CRT map, this is done in~\cite{GMS21tutte} for the Tutte embedding  and~\cite{bertacco2025scaling} for the Smith embedding.  In~\cite{HS19}, it is shown that under the Cardy embedding, the uniform triangulations converges to $\sqrt{8/3}$-LQG.    
\end{enumerate}

In this paper, we prove the convergence of mated-CRT map to LQG under two natural conformal embeddings, the \textit{circle packing} and the \textit{Riemann uniformization embedding}. The mated-CRT map can be defined in terms of a pair of correlated Brownian motions, and naturally embedded onto the SLE decorated LQG surfaces generated by the Brownian motions; see Section~\ref{subsubsec:intro-mated-CRT} for details. Our main result is, for the mated-CRT map, its  circle packing and Riemann uniformization embedding  are close to the a priori embedding described in terms of SLE and LQG. The proof is based on our recent work~\cite{HY26} on convergence of   circle packing and Riemann uniformization embedding for triangulations from ergodic scale-free environments defined in~\cite{GMSinvariance}, along with comparison over different circle packings discussed in Section~\ref{subsubsec:intro-conf}.  {The convergence of circle packing maps to conformal maps was first conjectured by Thurston and solved by He and Schramm~\cite{HeSchramm}. The original convergence in~\cite{HeSchramm} was in local uniform topology. We further prove that under certain conditions, the above convergence can be extended to the global uniform topology in various settings; see Theorem~\ref{thm:He-Schramm-extension-1}, Theorem~\ref{thm:He-Schramm-sphere} and Theorem~\ref{thm:He-Schramm-disk}. These results are of independent interest, and will be applied in our subsequent work~\cite{HY26a} on convergence of tree-weighted planar maps under circle packing.}

The rest of this section as follows. In Section~\ref{subsec:intro-background}, we provide some backgrounds on circle packing, the Riemann uniformization embedding, and the mated-CRT map. In Section~\ref{subsec:intro-mainresult}, we provide precise statements of our main results. In Section~\ref{subsec:intro-conf}, we recap our results from~\cite{HY26} and discuss transforms between different circle packings, and explain how these results are applied towards the proof of our main results. 

\subsection{Background and context}\label{subsec:intro-background}

\subsubsection{Triangulations and discrete conformal embeddings}\label{subsubsec:conf-embed}

A planar map $\cM$ is an embedding of a planar graph $\cG$ modulo orientation-preserving homeomorphisms of $\wh\bbC:=\bbC\cup\{\infty\}$. 
A face is a connected component of the complement of the planar map. A triangular face is a face with three edges on its boundary.  
A (planar) triangulation is the union of the vertices and edges of a planar map $\cM$, as well as (a subset of) the triangular faces of $\cM$. We will work with \emph{infinite plane triangulations}, where the underlying planar map is one-ended and every face has degree three  {and belongs to the triangulation},  \emph{disk triangulations}, where the triangulation is finite and topologically homeomorphic to the unit disk, and \emph{sphere triangulations},  where the triangulation is finite and every face is a triangular face.  {For triangulations $T_1$ and $T_2$, we say $T_1$ is a submap of $T_2$ if all the set of vertices, edges and faces of $T_1$ are corresponding subsets of $T_2$.}

A planar map $\cM$ can be embedded in $\wh{\bbC}$ (as described in the previous paragraph) in a number of different ways. In this paper, we focus on two natural  embeddings: the \emph{circle packing} and the \emph{Riemann uniformization embedding}. These two embeddings are canonical examples of \emph{discrete conformal embeddings}.  {This means in particular that} embeddings of the same planar map in different domains induce (or is believed to induce) a map that approximates the conformal map between the two domains. For circle packing, this was conjectured by Thurston and solved in~\cite{RodinSullivan87} in the case of the hexagonal lattice, 
in~\cite{kS90,kS96} for circle packings with uniformly bounded radii ratio for neighboring disks via probabilistic methods,
and by~\cite{HeSchramm} in full generality. For Riemann uniformization embedding, this follows from the definition.

\begin{figure}[t]
    \centering
    \includegraphics[scale=0.78]{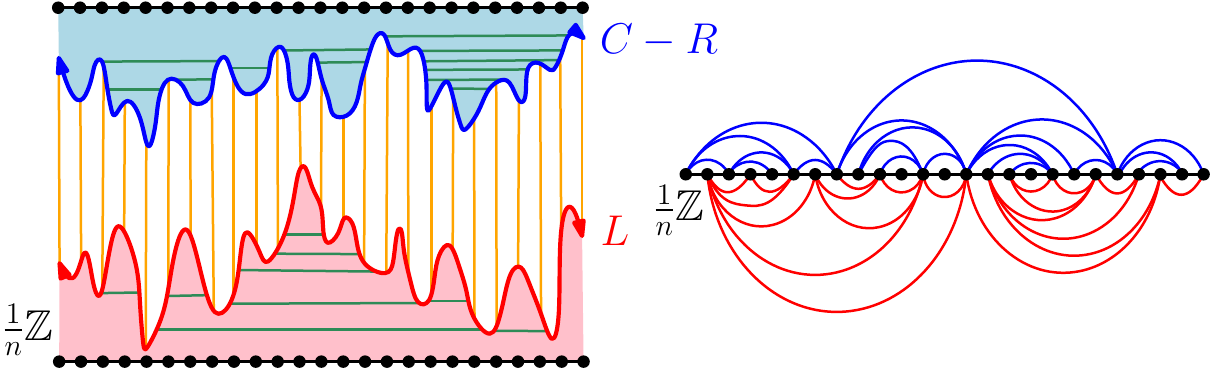}
    \caption{\textbf{Left:} To construct the mated-CRT map $\cG^{1/n}$ geometrically, one can draw the graph of $L$ (red) and the graph of $C-R$ (blue) for some large constant $C > 0$ chosen so that the parts of the graphs over some time interval of interest do not intersect. One then divides the region between the graphs into vertical strips (boundaries shown in orange) and identifies each strip with the horizontal coordinate $x\in \frac{1}{n}\bbZ$ of its rightmost point. Vertices $x_1,x_2\in \frac{1}{n}\bbZ$ are connected by an edge if and only if the corresponding strips either share a boundary segment (i.e., $|x_1-x_2| = \frac{1}{n}$) or are connected by a horizontal line segment which lies under the graph of $L$ or above the graph of $C-R$. One such segment is shown in green in the figure for each pair of vertices for which this latter condition holds.
\textbf{Right:} One can draw the graph $\cG^{1/n} $ in the plane by connecting two vertices $x_1,x_2 \in  \frac{1}{n}\bbZ$ by an edge by an arc above (resp.\ below) the blue line if the corresponding strips are connected by a horizontal segment above (resp.\ below) the graph of $L$ (resp.\ $C-R$); and connecting each pair of consecutive vertices of $ \frac{1}{n}\bbZ$ by an edge. This gives $\cG $ a planar map structure under which it is a triangulation. Note that the two pictures do not correspond to the same mated-CRT map realization.
}
    \label{fig:mated-crt-map}
\end{figure}

Let $\cM$ be a simple planar map with vertex set $\cV\cM$. We say that a collection $\cP = (D_v:v\in\cV\cM)$ of closed disks with disjoint interiors is a \textit{circle packing} for $\cM$, if two disks $D_v$ and $D_{v'}$ are tangent if and only if the vertices $v$ and $v'$ are connected by an edge.  This gives an embedding of the planar map, with the vertices embedded at the centers of the disks and the edges are straight line segments connecting the centers of tangent disks.  {We also require that the embedding preserves the orientation of the planar map.} By the celebrated Koebe-Andreev-Thurston circle packing theorem, if $\cM$ is a disk (resp.\ sphere) triangulation and is simple, then up to M\"{o}bius transforms, there exists a unique circle packing of $\cM$ in the unit disk $\bbD$ (resp.\ Riemann sphere $\wh\bbC$). In the case where $\cM$ is a disk triangulation, we say $\cP$ is a \emph{circle packing of $\cM$ in $\bbD$}, if all the disks in $\cP$ are in $\ol\bbD$, and the disks for boundary vertices are internally tangent to $\partial\bbD$. We further say $\cP$ is \emph{centered at a vertex} $v$, if the disk $D_v$ for the vertex $v$ is centered at origin.  For infinite plane triangulations, the existence of a circle packing is a straightforward consequence of the tightness from the Ring Lemma~\cite{RodinSullivan87}, and  the uniqueness up to M\"{o}bius transforms was proved by Schramm~\cite{Schramm91}. In this case, there is a dichotomy: an infinite plane triangulation $T$ is \emph{circle packing parabolic} (resp.\ \emph{circle packing hyperbolic}), if there exists a circle packing $\cP$ for $T$, such that the union of all the faces and edges  is equal to the whole plane (resp.\ unit disk). In~\cite{HeSchramm-hy-pa}, it was further shown that, if the simple random walk on $T$ is recurrent, then $T$ is circle packing parabolic, while if $T$ has uniformly bounded degree and the simple random walk on $T$ is transient, then $T$ is circle packing hyperbolic. See also~\cite{circlepackingbook,Nachmiaslec} for a comprehensive introduction on circle packing and related topics.

A second natural way to embed random planar maps is through the Riemann uniformization theorem. We follow the presentation in~\cite[Section 2.2]{GRRiemann}. Roughly speaking, for a triangulation $T$ without loops, one associate each face with an equilateral triangle and glue them isometrically along their edges to get a Riemann surface $M(T)$. The metric space $M(T)$ can be equipped with a compatible Riemann surface structure by defining coordinate charts as follows. If $x$ is an interior point of a triangle or interior to an edge on the boundary of the support, then we use this triangle, viewed as a subset of the plane,  as the chart, and use the identity map as the projection. Similarly, if $x$ is interior to an edge and not on the boundary of the support, then we place two adjacent equilateral triangles in the plane and use the identity as the projection. Now let $x$ be an interior vertex of degree $n$, and let $\Delta_1,...,\Delta_n$ be the equilaterals which meet at $x$  in cyclic order. We place them in the plane such that for $j=1,...,n$, $\Delta_j$ is the  triangle with vertices $0,e^{\pi(j-1)i/3}, e^{\pi ji/3}$, and define the coordinate map for $z=re^{i\theta}\in\Delta_j$ by $\phi(z) = r^{6/n}e^{6i\theta/n}$, using $\pi(j-1)/3<\theta<\pi j/3$ on $\Delta_j$. By further setting $\phi(0)=0$, this defines a homeomorphism between a neighborhood of $x$ in $M(T)$ and a neighborhood of 0 in $\bbC$. Moreover, for the atlas chosen as above, the transition map between two overlapping charts are analytic. 
The case where $x$ be an boundary vertex of degree $n$ can be treated similarly using the map $z\mapsto z^{3/(n-1)}$. This defines the Riemann surface structure for $M(T)$. Following the Riemann uniformization theorem, if $T$ is a disk (resp.\ sphere) triangulation, then $M(T)$ is conformally equivalent to the unit disk (resp. the Riemann sphere).  If $T$ is an infinite plane triangulation, there are again two possible cases, where the Riemann surface $M(T)$ is \emph{parabolic} (resp.\ \emph{hyperbolic}) if it is conformally equivalent to the whole plane (resp.\ the unit disk). As commented in~\cite{GRRiemann}, similar to the circle packing case, this is equivalent to the recurrence/transience of the Brownian motion on the surface $M(T)$. See e.g.~\cite{GRRiemann,Curienglimpse} for some works in this direction.

There exist a number of discrete conformal embeddings beyond the two discussed above, including the Tutte embedding, the Smith embedding \cite{bsst-smith}, the Cardy-Smirnov embedding \cite{HS19}, and the t-embedding/Coulomb gauge/s-embedding \cite{chelkak-icm,clr-t-emb,klrr-circle-patterns}.

\subsubsection{The mated-CRT map and its a priori embedding}\label{subsubsec:intro-mated-CRT}

We start with the definition of the mated-CRT map. Fix $\gamma\in(0,2)$ and  consider a two dimensional Brownian motion $Z=(L,R)$ with covariance matrix chosen so that
\begin{equation}\label{eq:variance}
    \mathrm{Var}(L_t)=\mathrm{Var}(R_t) = \mathbbm{a}^2|t|,  \ \ \ \mathrm{Cov}(L_t,R_t) = -\cos(\frac{\pi\gamma^2}{4})\mathbbm{a}^2|t|, \ \  \text{where } \mathbbm{a}^2 = \frac{2}{\sin(\frac{\pi\gamma^2}{4})}.
\end{equation}
For $n\in\bbN$, the \emph{mated-CRT map $\cG^{1/n}$} is defined as follows. The vertex set $\cV\cG^{1/n}$ is equal to $\frac{1}{n}\bbZ$, with two vertices $x_1,x_2\in \frac{1}{n}\bbZ$ connected by an edge if and only if either
\begin{equation}\label{eq:def-mated-CRT}
    \max\Big\{\inf_{t\in[x_1-\frac{1}{n},x_1]}L_t,  \inf_{t\in[x_2-\frac{1}{n},x_2]}L_t\Big\}\leq \inf_{t\in[x_1,x_2-\frac{1}{n}]}L_t \ \ \text{or} \ \ \max\Big\{\inf_{t\in[x_1-\frac{1}{n},x_1]}R_t,  \inf_{t\in[x_2-\frac{1}{n},x_2]}R_t\Big\}\leq \inf_{t\in[x_1,x_2-\frac{1}{n}]}R_t.
\end{equation}
If $x_2-x_1\neq\frac{1}{n}$ and both conditions in~\eqref{eq:def-mated-CRT} holds, then there are two edges between $x_1$ and $x_2$. It turns out that the map $\cG^{1/n}$ is an infinite plane triangulation. See Figure~\ref{fig:mated-crt-map} for an illustration.

One can also define mated-CRT map with disk and sphere topology by working on variants of the Brownian motion $Z$. Let $(L^\disk,R^\disk):[0,1]\to\bbR_+^2$ be the Brownian motion in~\eqref{eq:variance} starting from $(0,1)$ and conditioned to stay in the  positive quadrant and hit $(0,0)$ at time 1. The construction of this process was discussed in~\cite[Section 4]{AG19} based on results in~\cite{Shi85}. The \emph{mated-CRT map $\cG^{\disk,1/n}$ with disk topology} is defined is defined in the same way as $\cG^{1/n}$ using $(L^\disk,R^\disk)$ in place of $(L,R)$, where the vertex set is $\frac{1}{n}\bbZ\cap (0,1]$, and two vertices $x_1,x_2\in \frac{1}{n}\bbZ\cap (0,1]$ is connected by an edge if and only if~\eqref{eq:def-mated-CRT} holds for $(L^\disk,R^\disk)$ in place of $(L,R)$, while $x_1,x_2$ are connected by two edges if $x_2-x_1\neq\frac{1}{n}$ and both conditions in~\eqref{eq:def-mated-CRT} for $(L^\disk,R^\disk)$ hold. Similarly, let $(L^\sph,R^\sph)$ be the Brownian motion $(L,R)$ in~\eqref{eq:variance} starting from $(0,0)$, and conditioned  to stay in the  positive quadrant and hit $(0,0)$ at time 1. The process  $(L^\sph,R^\sph)$ has been constructed in~\cite[Section 3]{MS-mot-finite}. The  \emph{mated-CRT map $\cG^{\sph,1/n}$ with sphere topology} and vertices $\frac1n\bbZ\cap(0,1]$ is defined in an identical way as $\cG^{\disk,1/n}$, where one uses $(L^\sph,R^\sph)$ instead of $(L^\disk,R^\disk)$ and we declare that the vertices 1 and $\frac1n$ are connected only by a single edge. 

\begin{figure}[t]
    \centering
    \includegraphics[scale=0.7]{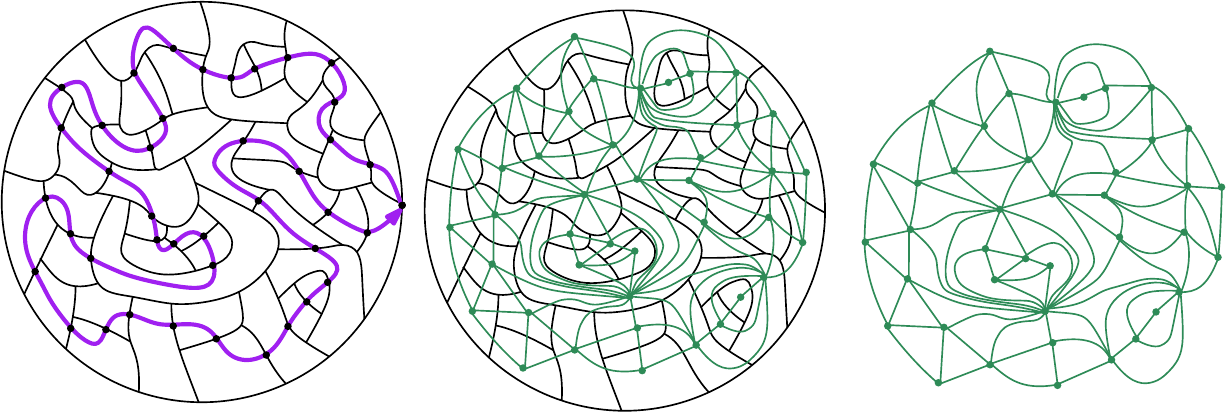}
    \caption{\textbf{Left:} A segment of a space-filling curve $\eta^\disk:[0,1]\to\bbD$, divided into cells $\eta^\disk([x-\frac{1}{n},x])$ for $x\in\frac{1}{n}\bbZ\cap (0,1]$. The order in which the cells are hit is shown by the purple path. This figure looks like what we would expect to see for $\gamma\leq\sqrt{2}$  since the cells are simply connected. \textbf{Middle:} We draw a green point in each cell and connect the points whose cells share a corresponding boundary arc. The green graph is then an a priori embedding of the mated-CRT map $\cG^{\disk,1/n}$ into $\ol\bbD$. \textbf{Right:} Same as the middle picture but without the original cells, so that only the a priori embedding of the mated-CRT map $\cG^{\disk,1/n}$ is visible.}
    \label{fig:mated-CRT1}
\end{figure}


Thanks to the main results in~\cite{DMS14,MS-mot-finite,AG19}, the mated-CRT maps comes a priori with a certain kind of embedding into $\bbC$ described by SLE-decorated LQG surfaces. Following~\cite[Theorem 1.9 and Theorem 1.11]{DMS14},   the Brownian motion $(L,R)$ in~\eqref{eq:variance}  (where the value of $\mathbbm{a}$ was computed in~\cite{ARS21}) a.s.\ determines a pair $(\eta,h)$, where $(\bbC,h,0,\infty)$ is an embedding of a $\gamma$-LQG surface called $\gamma$-quantum cone, and $\eta$ is a space-filling $\SLE_{16/\gamma^2}$ loop in $\wh\bbC$ from $\infty$ to $\infty$ independent from $h$ and parameterized by the $\gamma$-LQG mass induced by $h$, such that $\eta(0)=0$. See Section~\ref{sec:pre} for a brief overview of these objects. This defines the a priori embedding of $\cG^{1/n}$ as follows. For each $x\in\frac{1}{n}\bbZ$, we choose a point $v_x$ in the interior of $\eta([x-\frac{1}{n},x])$. We call the clockwise (resp.\ counterclockwise) boundary arc of $\partial\eta([x-\frac{1}{n},x])$ from $\eta(x-\frac{1}{n})$ to $\eta(x)$ the left (resp.\ right) boundary of $\eta([x-\frac{1}{n},x])$. We connect each $v_x$ and $v_{x+\frac{1}{n}}$ by a curve. For $x_1,x_2\in \frac{1}{n}\bbZ$ such that $|x_1-x_2|\neq\frac{1}{n}$, using the description of $(L_t,R_t)$ in terms of the boundary length process for $\eta((-\infty,t])$ as in~\cite[Section 8.2]{DMS14} (see also Figure~\ref{fig:quantumcell3}), it follows that almost surely,  the equation~\eqref{eq:def-mated-CRT} for $L$ (resp.\ $R$) holds if and only if  the left (resp.\ right) boundary of $\eta([x_1-\frac1n,x_1])$ shares a connected boundary arc with the (resp.\ right) boundary  of $\eta([x_2-\frac1n,x_2])$, and we can draw the edges connecting $x_1$ and $x_2$ accordingly. The a priori embedding for the mated-CRT map $\cG^{\disk,1/n}$ and $\cG^{\sph,1/n}$ with disk topology and sphere topology can also be defined analogously. In the disk setting, by~\cite{AG19},  the Brownian excursion $(L^\disk,R^\disk)$ defined above generates a pair $(\eta^\disk,h^\disk)$, where $(\bbD,h^\disk,0,1)$ is an embedding of a $\gamma$-LQG surface called quantum disk, and $\eta$ is a counterclockwise space-filling $\SLE_{16/\gamma^2}$ loop in $\ol\bbD$ independent from $h^\disk$ from 1 to 1 and parameterized by the $\gamma$-LQG mass induced by $h^\disk$, and the a priori embedding of $\cG^{\disk,1/n}$ is defined in the same way using the cells $\{\eta^\disk([x-\frac{1}{n},x]):x\in \frac{1}{n}\bbZ\cap(0,1]\}$. In the sphere case, by~\cite{MS-mot-finite}, the Brownian excursion $(L^\sph,R^\sph)$ defined above generates a pair $(\eta^\sph,h^\sph)$, where $(\wh\bbC,h^\disk,0,\infty)$ is an embedding of a $\gamma$-LQG surface called quantum sphere, and $\eta$ is a  space-filling $\SLE_{16/\gamma^2}$ loop in $\wh\bbC$ independent from $h^\sph$ from $\infty$ to $\infty$ and parameterized by the $\gamma$-LQG mass induced by $h^\sph$, and the a priori embedding for $\cG^{\sph,1/n}$  is defined in the same way using the cells $\{\eta^\sph([x-\frac{1}{n},x]):x\in \frac{1}{n}\bbZ\cap(0,1]\}$. See also Figure~\ref{fig:mated-CRT1} for an illustration of the a priori embedding in the disk topology case.  Note that the cells $\eta([x-\frac{1}{n},x])$ are simply connected when $\gamma\in(0,\sqrt2]$ and not simply connected when $\gamma\in(\sqrt{2},2)$; see Figure~\ref{fig:mated-CRT3} for an illustration. In~\cite[Proposition 2.2]{GMS21tutte}, it is proved that the mated-CRT map defined above via LQG and SLE is the same as the same as the one defined using~\eqref{eq:def-mated-CRT}.

\begin{figure}[t]
    \centering
    \includegraphics[scale=0.7]{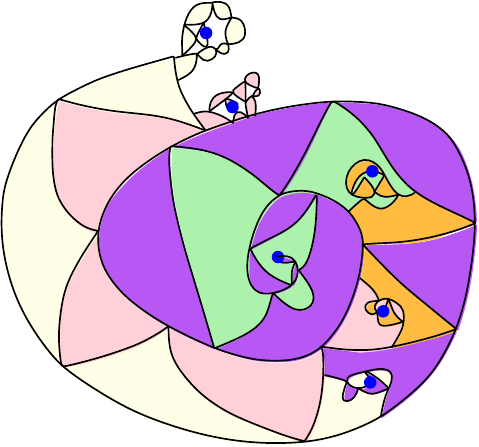}
    \caption{An illustration of five cells $\eta([x-\frac{1}{n},x])$ consecutively hit by $\eta$ in the regime $\gamma\in(\sqrt2,2)$, where the blue points represents the first and last points on the cell visited by $\eta$. This figure is slightly misleading since the set of pinch points where the left boundary and right boundary of each cell meet is in fact uncountable. Two cells are adjacent if and only if they share a connected boundary arc. In particular, the green cell and the pink cell are not considered to be adjacent. }
    \label{fig:mated-CRT3}
\end{figure}

\subsection{Main results}\label{subsec:intro-mainresult}

The main result of this paper, roughly speaking, is that if one considers the circle packing and Riemann uniformization embedding of the maps $\cG^{1/n},\cG^{\disk,1/n}$ and $\cG^{\sph,1/n}$, then for large $n$, the circle packing and Riemann uniformization embedding are close to the a priori embedding. Note that circle packing is defined for simple planar maps whereas the mated-CRT map has double edges. To resolve this, one has to turn   a triangulation $T$ with multiple edges into a triangulation that is simple. There are multiple  ways to do so, and   we are going to use the following natural operator $\Phi_{\rm simple}$ that collapses multiple edges of $T$ as below.   Let $T$ be a disk triangulation or an infinite plane triangulation. For each pair of vertices $(v_1,v_2)$ of $T$ such that $v_1,v_2$ are connected by edges $e_1,...,e_k$, we remove the vertices and edges enclosed by $e_1,...,e_k$, and further remove the edges $e_2,...,e_k$ from $T$. We let $\Phi_{\rm simple}(T)$ be the remaining map. Note that if $T$ is a disk triangulation  {with at least three vertices on the boundary}, then $\Phi_{\rm simple}(T)$  remains to be a disk triangulation, while if $T$ is an infinite plane triangulation, then $\Phi_{\rm simple}(T)$ is either an infinite plane triangulation or empty. If $T$ is a sphere triangulation with $n$ vertices, suppose for each pair of vertices $(v_1,v_2)$ of $T$ such that $v_1,v_2$ are connected by edges $e_1,...,e_k$, one of the connected components of $T\backslash\{e_1,...,e_k\}$ has at least $3n/4$ vertices; in this case, we further remove all the vertices and edges in the other connected components as well as $e_2,...,e_k$ and let $\Phi_{\rm simple}$ be the remaining map. If this is not the case, then we set $\Phi_{\rm simple}(T)=\emptyset$. If $T$ is a sphere triangulation, then  $\Phi_{\rm simple}(T)$ is either a sphere triangulation or empty.  Observe that for a disk, sphere or infinite plane triangulation $T$ such that $\Phi_{\rm simple}(T)\neq\emptyset$, if  $v\notin \cV\Phi_{\rm simple}(T)$, then there exists a unique pair of vertices $(v_1,v_2)\in \cV\Phi_{\rm simple}(T) $ such that the vertex $v$ is removed when we collapse multiple edges between $v_1$ and $v_2$ as above. See also Figure~\ref{fig:mated-CRT2} for an example for the map $\cG^{\disk,1/n}$ in Figure~\ref{fig:mated-CRT1}.  {Further note that the operation $\Phi_{\rm simple}$ does not create new edges, and two vertices of $\Phi_{\rm simple}(T)$ are connected by an edge in $\Phi_{\rm simple}(T)$ if and only if they are connected by an edge in $T$.}

\begin{figure}[t]
    \centering
    \includegraphics[scale=0.7]{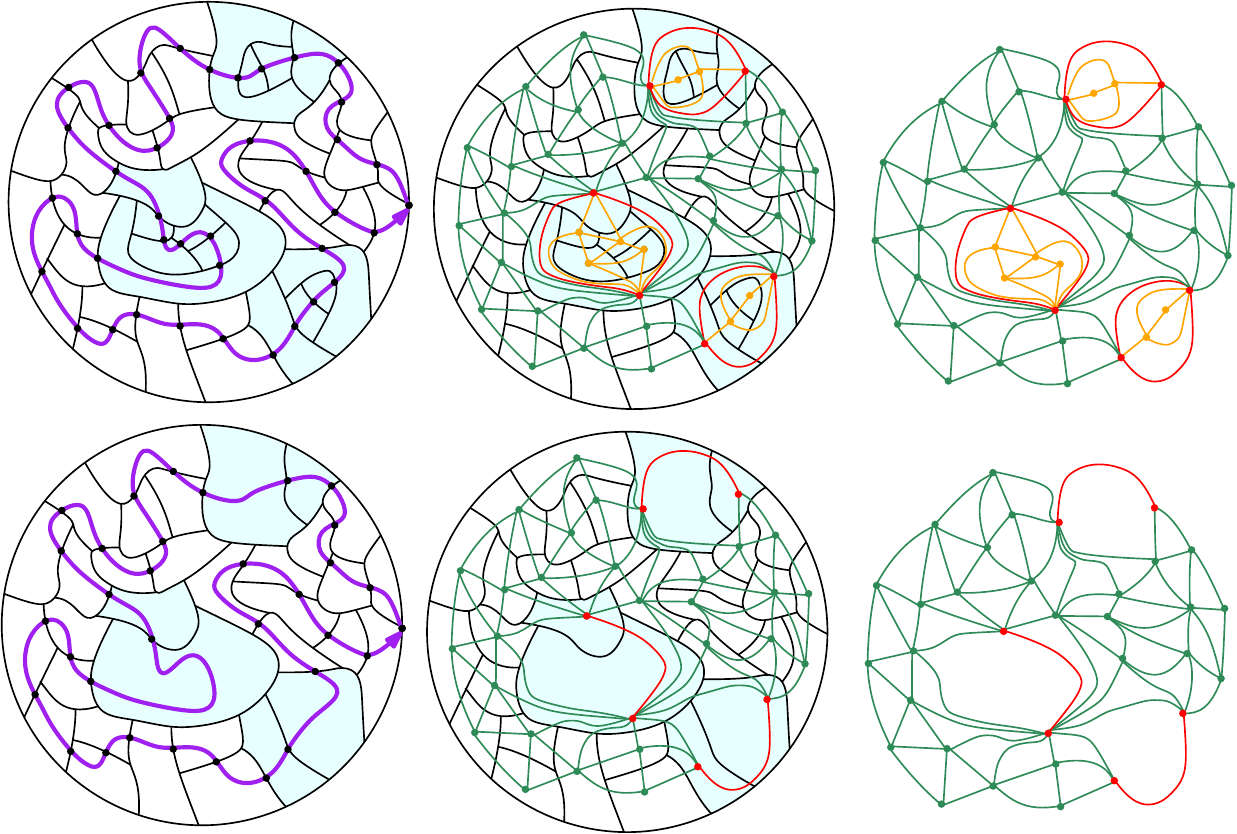}
    \caption{An illustration of the operator $\Phi_{\rm simple}$ for the map in Figure~\ref{fig:mated-CRT1}. \textbf{Top:} The cells that form double edges or are contained in double edges are colored in light blue in the left panel. In the middle and right panel, the outermost double edges and their endpoints are marked in red, and the vertices and edges enclosed be the outermost double edges are marked in orange. \textbf{Bottom:} In the left panel, for two cells such that there are two edges connecting them,  we merge the content between them with the cell that is visited later by the space-filling curve. In the middle and right panel, for each pair of the red double edge, we remove all the orange edges and vertices between them and also one of the red edges. This turns the map into a simple disk triangulation.}
    \label{fig:mated-CRT2}
\end{figure}

\begin{definition}\label{def:pack-non-simple}
   Let $T$ is a disk, sphere or infinite plane triangulation with  multiple edges. If $\Phi_{\rm simple}(T)$ is empty, then we set the circle packing for $T$ to be empty. If $\Phi_{\rm simple}(T)$ is not empty, we define circle packings for $T$ as below. Let $\wt\cP = \{\wt D_v:v\in \cV\Phi_{\rm simple}(T) \}$ be a circle packing for $\Phi_{\rm simple}(T)$. For a vertex $v\in\cV T$, if $v\in \cV\Phi_{\rm simple}(T) $, then we let $D_v = \wt D_v$. If $v\notin \cV\Phi_{\rm simple}(T)$, let $(v_1,v_2)\in \cV\Phi_{\rm simple}(T) $  be the unique pair of vertices such that the vertex $v$ is removed when we collapse multiple edges between $v_1$ and $v_2$. Then we let $D_v$ be the tangent point between the disks $D_{v_1}$ and $D_{v_2}$, and we view $D_v$ as a degenerate disk. We call $\cP:=\{D_v:v\in\cV T\}$ a circle packing for $T$. If $T$ is an infinite triangulation, we say $T$ is circle packing parabolic (resp.\ hyperbolic) if $\Phi_{\rm simple}(T)$ is circle packing parabolic (resp.\ hyperbolic). 
\end{definition}

We are ready to state our main results. In the rest of the paper, we will always assume $\gamma\in (0,2)$ and $\kappa=16/\gamma^2$.

\begin{theorem}\label{thm:pack-cone}
  Let   $Z=(L,R)$ be a Brownian motion with variance in~\eqref{eq:variance}. Let $(\bbC,h,0,\infty)$ be the $\gamma$-quantum cone under circle average embedding (see Definition~\ref{def:alpha-cone}) and $\eta$ be the space-filling $\SLE_\kappa$ in $\wh\bbC$ from $\infty$ to $\infty$ independent from $h$ and parameterized by the $\gamma$-LQG mass $\mu_h$, such that $(h,\eta)$ is  determined by $Z$ in the sense of~\cite[Theorem 1.9 and Theorem 1.11]{DMS14}. Let $\cG^{1/n}$ be the mated-CRT map defined by $Z$. Then $\cG^{1/n}$ is a.s.\ circle packing parabolic, and there exists a circle packing $\cP^{n}$ for $\cG^{1/n}$ such that the following is true:
  \begin{enumerate}
      \item For each $r>0$, the maximal diameter of disks in $\cP^{n}$ intersecting $B(0;r)$ tends to 0 in probability;
      \item For each $x\in\frac{1}{n}\bbZ$, we write $o_n(x)$ for the center of the disk for $x$ in $\cP^{n}$. Then 
      \begin{equation}
          \sup\big\{|\eta(x)-o_n(x)|: x\in\frac{1}{n}\bbZ, \eta(x)\in B(0;r)  \big\} \to 0 \text{ in probability as } n\to\infty;
      \end{equation}
      \item Let $\mu^n$ be the measure on $\bbC$ that assigns mass $\frac{1}{n}$ to each of the point $o_n(x)$ for  $x\in\frac{1}{n}\bbZ$, where duplication is counted (i.e., if   $\#\{x\in \frac{1}{n}\bbZ:o_n(x)=z\}=k$, then $\mu^n(\{z\})=k/n$). Then for each $r>0$, the Prokhorov distance between $\mu^n|_{B(0;r)}$ and $\mu_h|_{B(0;r)}$ tends to 0 in probability;
      \item For every $a>0$, the space-filling path on the embedded mated-CRT map which comes from the left-right ordering of the vertices restricted to time $[-a,a]$ converges uniformly to $\eta|_{[-a,a]}$ in probability.
  \end{enumerate}
\end{theorem}

\begin{theorem}\label{thm:pack-disk}
     Let $Z^\disk=(L^\disk,R^\disk):[0,1]\to\bbR_+^2$ be the Brownian motion in~\eqref{eq:variance} starting from $(0,1)$ and conditioned to stay in the  positive quadrant and hit $(0,0)$ at time 1. Let $(\bbD,h^\disk,\eta^\disk)$ be the counterclockwise space-filling $\SLE_\kappa$ decorated quantum disk determined by $Z^\disk$ in the sense of~\cite[Theorem 1.1]{AG19}, where the curve $\eta^\disk$ is parameterized by the $\gamma$-LQG measure $\mu_{h^\disk}$ induced by $h^\disk$ and embedded such that $\eta^\disk(0)=\eta^\disk(1)=1$ and $\eta^\disk(\frac12)=0$.  Let $\cG^{\disk,1/n}$ be the mated-CRT map defined by $Z^\disk$. Then there exists a circle packing $\cP^{\disk,n}$ for $\cG^{\disk,1/n}$  such that the following is true:
     \begin{enumerate}
         \item The maximal diameter of disks in $\cP^{\disk,n}$ tends to 0 in probability;
         \item For each $x\in\frac{1}{n}\bbZ\cap (0,1]$, we write $o_n(x)$ for the center of the disk for $x$ in $\cP^{\disk,n}$. Then 
      \begin{equation}\label{eq:thm:pack-disk-0}
          \sup\big\{|\eta^\disk(x)-o_n(x)|: x\in\frac{1}{n}\bbZ\cap (0,1]   \big\} \to 0 \text{ in probability as } n\to\infty;
      \end{equation}
      \item Let $\mu^{\disk,n}$ be the measure on $\ol\bbD$ that assigns mass $\frac{1}{n}$ to each of the point $o_n(x)$ for  $x\in\frac{1}{n}\bbZ\cap(0,1]$, where duplication is counted. Then   the Prokhorov distance between $\mu^{\disk,n} $ and $\mu_{h^\disk}$ tends to 0 in probability;
      \item The space-filling path on the embedded mated-CRT map which comes from the left-right ordering of the vertices converges uniformly to $\eta^\disk$ in probability.
     \end{enumerate}
\end{theorem}

For $z,w\in\wh\bbC = \bbC\cup\{\infty\}$, we write $d_\sph(z,w)$ for the spherical distance between $z$ and $w$.
\begin{theorem}\label{thm:pack-sph}
    Let $Z^\sph=(L^\sph,R^\sph):[0,1]\to\bbR_+^2$ be the Brownian motion in~\eqref{eq:variance} starting from $(0,0)$ and conditioned to stay in the  positive quadrant and hit $(0,0)$ at time 1. Let $(\wh\bbC,h^\sph,\eta^\sph)$ be the counterclockwise space-filling $\SLE_\kappa$ loop decorated quantum sphere determined by $Z^\sph$ in the sense of~\cite[Theorem 1.1]{MS-mot-finite}, where the curve $\eta^\sph$ is parameterized by the $\gamma$-LQG measure $\mu_{h^\sph}$ induced by $h^\sph$ and embedded such that $\eta^\sph(0)=\infty$, $\eta^\sph(\frac12)=0$ and $\eta^\sph(\frac14)=1$.  Let $\cG^{\sph,1/n}$ be the mated-CRT map defined by $Z^\sph$. Then $\Phi_{\rm simple}(\cG^{\sph,1/n})$ is non-empty with probability tending to 1 as $n\to\infty$, and there exists a circle packing $\cP^{\sph,n}$ for $\cG^{\sph,1/n}$  such that the following is true:
     \begin{enumerate}
         \item The maximal spherical diameter of disks in $\cP^{\sph,n}$ tends to 0 in probability;
         \item For each $x\in\frac{1}{n}\bbZ\cap (0,1]$, let $o_n(x)$ be the Euclidean center of the disk for $x$ in $\cP^{\sph,n}$ if it does not contain $\infty$ and set $o_n(x)=\infty$ otherwise. Then 
      \begin{equation}
          \sup\big\{d_{\sph}(\eta^\sph(x),o_n(x)): x\in\frac{1}{n}\bbZ\cap (0,1]   \big\} \to 0 \text{ in probability as } n\to\infty;
      \end{equation}
      \item Let $\mu^{\sph,n}$ be the measure on $\wh\bbC$ that assigns mass $\frac{1}{n}$ to each of the point $o_n(x)$ for  $x\in\frac{1}{n}\bbZ\cap(0,1]$, where duplication is counted. Then   the Prokhorov distance between $\mu^{\sph,n} $ and $\mu_{h^\sph}$ tends to 0 in probability;
      \item The space-filling path on the embedded mated-CRT map which comes from the left-right ordering of the vertices converges uniformly to $\eta^\sph$ in probability under the spherical metric.
     \end{enumerate}
\end{theorem}

Next we state our main result for the Riemann uniformization embedding. For a disk, sphere or infinite plane triangulation $T$, we write $M(T)$ for the Riemann surface generated by $T$ as in Section~\ref{subsubsec:conf-embed}. With slight abuse of notation, we will also write $x\in\frac{1}{n}\bbZ$ for the vertex $x$ on the surface $M(T)$ when $T\in\{\cG^{1/n},\cG^{\disk,1/n},\cG^{\sph,1/n}\}.$

\begin{theorem}\label{thm:unif-cone}
    Let $(Z,h,\eta)$ and $\cG^{1/n}$ be as in Theorem~\ref{thm:pack-cone}. Then $M(\cG^{1/n})$ is a.s.\ parabolic, and there exists a conformal map $\varphi^n:M(\cG^{1/n})\to\bbC$ such that the following is true:
     \begin{enumerate}
      \item For each $R>0$, the maximal diameter of edges  $\varphi^n(\fre)$  intersecting $B(0;R)$ tends to 0 in probability;
      \item We have 
      \begin{equation}
          \sup\big\{|\eta(x)-\varphi^n(x)|: x\in\frac{1}{n}\bbZ, \eta(x)\in B(0;R)  \big\} \to 0 \text{ in probability as } n\to\infty;
      \end{equation}
      \item Let $\mu^n$ be the measure on $\bbC$ that assigns mass $\frac{1}{n}$ to each of the point $\varphi^n(x)$ for  $x\in\frac{1}{n}\bbZ$. Then for each $R>0$, the Prokhorov distance between $\mu^n|_{B(0;R)}$ and $\mu_h|_{B(0;R)}$ tends to 0 in probability;
      \item For every $a>0$, the space-filling path on the embedded mated-CRT map which comes from the left-right ordering of the vertices restricted to time $[-a,a]$ converges uniformly to $\eta|_{[-a,a]}$ in probability.
  \end{enumerate}
\end{theorem}

\begin{theorem}\label{thm:unif-disk}
     Let $(Z^\disk,h^\disk,\eta^\disk)$ and $\cG^{\disk,1/n}$ be as in Theorem~\ref{thm:pack-disk}. Then   there exists a conformal map $\varphi^{\disk,n}:M(\cG^{\disk,1/n})\to\ol\bbD$ such that the following is true:
     \begin{enumerate}
      \item The maximal diameter of edges  $\varphi^{\disk,n}(\fre)$   tends to 0 in probability;
      \item We have
      \begin{equation}
          \sup\big\{|\eta^\disk(x)-\varphi^{\disk,n}(x)|: x\in\frac{1}{n}\bbZ\cap(0,1]  \big\} \to 0 \text{ in probability as } n\to\infty;
      \end{equation}
      \item Let $\mu^{\disk,n}$ be the measure on $\ol\bbD$ that assigns mass $\frac{1}{n}$ to each of the point $\varphi^{\disk,n}(x)$ for  $x\in\frac{1}{n}\bbZ\cap (0,1]$. Then  the Prokhorov distance between $\mu^{\disk,n}$ and $\mu_{h^\disk}$ tends to 0 in probability;
      \item The space-filling path on the embedded mated-CRT map which comes from the left-right ordering of the vertices   converges uniformly to $\eta^\disk$ in probability.
  \end{enumerate}
\end{theorem}

\begin{theorem}\label{thm:unif-sph}
     Let $(Z^\sph,h^\sph,\eta^\sph)$ and $\cG^{\sph,1/n}$ be as in Theorem~\ref{thm:pack-sph}. Then   there exists a conformal map $\varphi^{\sph,n}:M(\cG^{\sph,1/n})\to\wh\bbC$ such that the following is true:
     \begin{enumerate}
      \item The maximal spherical diameter of edges  $\varphi^{\sph,n}(\fre)$   tends to 0 in probability;
      \item We have
      \begin{equation}
          \sup\big\{d_{\sph}(\eta^\sph(x),\varphi^{\sph,n}(x)): x\in\frac{1}{n}\bbZ\cap(0,1]  \big\} \to 0 \text{ in probability as } n\to\infty;
      \end{equation}
      \item Let $\mu^{\sph,n}$ be the measure on $\wh\bbC$ that assigns mass $\frac{1}{n}$ to each of the point $\varphi^{\sph,n}(x)$ for  $x\in\frac{1}{n}\bbZ\cap (0,1]$. Then  the Prokhorov distance between $\mu^{\sph,n}$ and $\mu_{h^\sph}$ tends to 0 in probability;
      \item The space-filling path on the embedded mated-CRT map which comes from the left-right ordering of the vertices   converges uniformly to  {$\eta^\sph$} in probability under the spherical metric.
  \end{enumerate}
\end{theorem}

\subsection{Circle packing and Riemann uniformization for random planar maps}\label{subsec:intro-conf}
 The proof of Theorems~\ref{thm:pack-cone}-\ref{thm:unif-sph} is consisted of two steps. The first step is to prove a variant of Theorem~\ref{thm:pack-cone} and Theorem~\ref{thm:unif-cone} in the setting of a $\gamma$-LQG surface called the 0-quantum cone. This is based on our recent work~\cite{HY26}, which we briefly recap in Section~\ref{subsubsec:intro-cell}, applied to the cell configuration considered in~\cite{GMS21tutte} in the Tutte embedding case.  The second step is to transfer from the setting of 0-quantum cone to the various surfaces considered in Theorems~\ref{thm:pack-cone}-\ref{thm:unif-sph}. In the Tutte embedding setting, this mainly follows from local absolute continuity between different LQG surfaces. For circle packing and Riemann uniformization, one major challenge is that these embeddings  are non-local and extra input is needed here. In Section~\ref{subsubsec:intro-conf}, we introduce several extensions of He and Schramm's work~\cite{HeSchramm} on convergence of circle packing maps to conformal maps. The proofs of these extensions will be carried out in Section~\ref{sec:CP-conformal}.  In Section~\ref{subsubsec:intro-overview}, we briefly explain how these results  are applied to the proof of our main results in Section~\ref{sec:proof}.  

\subsubsection{Circle packing and Riemann uniformization for random planar maps in an ergodic scale-free environment}\label{subsubsec:intro-cell}


In our recent work~\cite{HY26}, we proved that for embedded random triangulations in an ergodic scale-free environment, under certain moment and degree bounds, they are close to their circle packing and Riemann uniformization embedding. The  ergodic scale-free environment was first considered by Gwynne, Miller and Sheffield in their work~\cite{GMSinvariance}, where they proved that random walks on such random environments converge to Brownian motion in quenched sense, and they later applied this result to prove the convergence of the Tutte embedding for the mated-CRT map and the Poisson-Voronoi tessellation of the Brownian disk~\cite{GMS21tutte,GMSBrowniandisk}.

We start with the following definition of cell configurations.

\begin{definition}\label{def:cell-config}
    A \textit{cell configuration} on $\bbC$ consists of the following objects.
    \begin{enumerate}
        \item A locally finite collection $\cH$ of  compact connected subsets of $\bbC$ (``cells") with non-empty interiors whose union is all of $\bbC$ and such that the intersection of any two elements of $\cH$ has zero Lebesgue measure.
        \item A planar map $\cM$, which we call the \emph{associated map} of $\cH$, with vertices $\{v_H:H\in\cH\}$. We require that if $v_H$ and $v_{H'}$ are connected by an edge of $\cM$ then $H\neq H'$ and $H\cap H'\neq\emptyset$. 
    \item A function $\fc = \fc_\cH$ (``conductance") on the edges  $\cM$ with values in $(0,\infty)$. 
    \end{enumerate}
\end{definition}
We say $H\sim H'$ if $H$ and $H'$ are connected by  {at least one} edge in $\cM$. Note that we use a different definition of a cell configuration than in~\cite[Definition 1.1]{GMSinvariance}, where a cell configuration is defined to be a tuple $(\cH,\sim,\fc)$. We need to also include $\cM$ in the definition since our goal is to prove closeness between $\cH$ and conformal embeddings of the map $\cM$ and, furthermore, will work with settings where $\cM$ has multiple edges.  
We will typically slightly abuse notation by making  the associated map $\cM$ and the function $\fc$ implicit, so we write $\cH$ instead of $(\cH,\cM,\fc )$. 

For $A\subseteq\bbC$, let $\cH(A) = \{H\in\cH:H\cap A\neq\emptyset\}$, and we view $\cH(A)$ as an edge-weighted planar map with edge set $\cE\cH(A)$ consisting of all of the edges in $\cE\cH$ joining elements in $\cH(A)$ and conductances given by the restriction of $\fc$. We also define a metric on the space of cell configurations by
 \begin{equation}\label{eq:metric-cc}
     \begin{split}
         \mathbb{d}^{\rm CC}(\cH,\cH') := \int_0^\infty e^{-r}\wedge \inf\Big\{\sup_{z\in\bbC} |z-f_r(z)|+\sup_{e\in \cE\cH(B(0;r))}\big|\fc(e)-\fc'(f_r(e)) \big|   \Big\}\,dr,
     \end{split}
 \end{equation}
 where the infimum is over all homeomorphisms $f_r:\bbC\to\bbC$ such that $f_r$ takes each cell in $\cH(B(0;r))$ to $\cH'(B(0;r))$   and  preserves the associated map when restricted to cells intersecting $B(0;r)$, and $f_r^{-1}$ does the same with $\cH$ and $\cH'$ reversed. Here $f_r(\fre)$ is understood as the image of the edge $e$ under map isomorphisms. We will typically be working with a random cell configuration, i.e., a random
variable taking values in the space of cell configurations equipped with the Borel $\sigma$-algebra generated
by the above metric.


\begin{definition}\label{def:CellConfigConstraint}
    We define  the following conditions  on the cell configurations.
    \begin{enumerate}
        \item \textbf{Translation invariance modulo scaling.} There is a (possibly random and $\cH$-dependent) increasing sequence of open sets $U_j\subseteq\bbC$,  each of which is either a square or a disk, whose union is all of $\bbC$ such that the following is true. Conditional on $\cH$ and $U_j$, let $z_j$ for $j\in\bbN$ be sampled uniformly from Lebesgue measure on $U_j$. Then there exist random numbers $C_j>0$ (possibly depending on $\cH$ and $z_j$) such that $C_j(\cH-z_j)\to\cH$ in law with respect to the metric~\eqref{eq:metric-cc}.
        \item  \textbf{Ergodicity modulo scaling.} Every real-valued measurable functions $F=F(\cH)$ which is invariant under translation and scaling, i.e., $F(C(\cH-z)) = F(\cH)$ for each $z\in\bbC$ and $C>0$, is a.s.\ equal to a deterministic constant.
        \item  \textbf{Connectedness along  {macroscopic} lines.} Almost surely,  {for each given $\e>0$, there exists $R>0$ such that for $r>R$ and each horizontal or vertical line segment $\ell\subseteq B(0;r)$ with length at least $\e r$}, the subgraph of $\cH$ induced by the set of cells which intersect $\ell$ is connected. 
        \item \textbf{Rotation invariance in law.} For some $\theta_0\in(0,2\pi)\backslash\{\pi\}$, $e^{i\theta_0}\cH$ has the same law as $\cH$.

    \end{enumerate}
\end{definition}

Now we state the main results of~\cite{HY26}.

\begin{theorem}[Theorem 1.6 of~\cite{HY26}]\label{thm:CellSystemPack}
    Let $\cH$ be a random cell configuration such that its associated map is a simple infinite plane triangulation. Suppose $\cH$   satisfies all the conditions in Definition~\ref{def:CellConfigConstraint}. Further assume $\cH$ satisfies the moment bound
    \begin{equation}\label{eq:thm-inv-circle}
        \bbE\Big[\frac{\diam(H_0)^2}{\area(H_0)}\deg(H_0)^4\Big]<\infty.
    \end{equation}
    Then $\cH$ is a.s.\  circle packing parabolic, and there exists a circle packing $\cP$ for $\cH$ such that the following is true.  {Almost surely, for any $\e>0$, there exists $R_0>0$ such that for every $r>R_0$, the maximal diameter of disks in $\cP$ intersecting $B(0;r)$ is less than $\e r$.} For $H\in\cH$, let $c(H) = \frac{1}{\area(H)}\int_H z\,dz$ be the Euclidean center of $H$, and $o(H)$ be the center of the disk for $H$ in $\cP$. Then almost surely,  {up to a reflection of $\bbC$},
    \begin{equation}\label{eq:thm:CellsystemPack}
        \lim_{r\to\infty}\frac{1}{r}\max_{H\in \cH(B(0;r))}|c(H)-o(H)| = 0.
    \end{equation}
\end{theorem}

\begin{theorem}[Theorem 1.7 of~\cite{HY26}]\label{thm:invariance-uniformization}
     Let $\cH$ be a random cell configuration such that its associated map is an infinite plane triangulation. Suppose  $\cH$ satisfies the conditions in  Definition~\ref{def:CellConfigConstraint}.  Further assume $\cH$ satisfies the moment bound~\eqref{eq:thm-inv-circle}.
    Then the Riemann surface $M(\cH)$ formed by the associated map of $\cH$ is a.s.\   parabolic, and there exists a  conformal map $\varphi_0:M(\cH)\to\bbC$ for $\cH$ such that the following is true. For $H\in\cH$, let $c(H) = \frac{1}{\area(H)}\int_H z\,dz$ be the Euclidean center of $H$, and $v_H$ be the   vertex for $H$ in $M(\cH)$. Then {almost surely},  {up to a reflection of $\bbC$,}
    \begin{equation}\label{eq:thm:invariance-uniform}
        \lim_{r\to\infty}\frac{1}{r}\max_{H\in \cH(B(0;r))}| c(H)-\varphi_0(v_H)| = 0.
    \end{equation}
    Furthermore,  for each $e\in \cE\cH$, we write $\mathfrak{e}_{e}$ for the edge for $e$ on  $M(\cH)$,  
     \begin{equation}\label{eq:thm:invariance-uniform-1}
        \lim_{r\to\infty}\frac{1}{r}\max_{e\in \cE\cH(B(0;r))}\diam(\varphi_0(\mathfrak{e}_{e})) = 0.
    \end{equation}
\end{theorem}

\subsubsection{Convergence of circle packing maps to conformal maps}\label{subsubsec:intro-conf}

To transfer the convergence of circle packing map in the whole-plane setting as in Theorem~\ref{thm:CellSystemPack} to the finite volume setting, one need to compare circle packing of the same triangulation in different domains. To complete the transfer, we introduce several extensions of  He and Schramm's~\cite{HeSchramm}  convergence of circle packing maps to conformal maps; see Theorems~\ref{thm:He-Schramm-extension-0},~\ref{thm:He-Schramm-extension-1},  ~\ref{thm:He-Schramm-sphere} and~\ref{thm:He-Schramm-disk} below. These results are of independent interest, and will also be applied in our subsequent work~\cite{HY26a} on the convergence of spanning tree weighted maps to $\sqrt{2}$-LQG under circle packing and Riemann uniformization embedding. 

We briefly recap He and Schramm's work~\cite{HeSchramm} on the convergence of circle packing maps to conformal maps.
An \emph{interstice} of a circle packing $\cP$ for a triangulation $T$ is a connected component of the complement of $\cup_{D\in \cP} D$  {where we assumes that the disks in $\cP$ are closed}. The \emph{support}, $\supp(\cP)$, of $\cP$ is defined to be the union of the disks in $\cP$ and all bounded interstices of $\cP$.
Let $\{T^n\}_{n\geq1}$ be simple  disk triangulations, and $\cP^n, \wt\cP^n$ be circle packings of $T^n$ in domains $D$ and $\wt D$, respectively. This induces a natural  {(non-unique)} homeomorphism $f^n$ between the supports of $\cP^n$ and $\wt\cP^n$, sending each disk in $\cP^n$ to its corresponding disk in $\wt\cP^n$. It was conjectured by Thurston that if the maximal diameter of disks in $\cP^n$ and $\wt\cP^n$ converge to 0, then the circle packing map $\{f^n\}$ approximates a conformal map $f$ from $D$  to $\wt D$ as $n\to\infty$. This was proved by Rodin and Sullivan~\cite{RodinSullivan87} for the hexagonal lattice case and by Stephenson~\cite{kS90,kS96} in bounded degree setting and finally  
 in full generality by He and Schramm~\cite{HeSchramm}. In~\cite[Theorem 1.2]{HeSchramm}, it has further been proven that if one set the map $f^n$ to be M\"{o}bius on each of the interstice, then $(f^n)'$ and $(f^n)''$ converges to $f'$ and $f''$ as well; we will not need this for this paper. 

 Let $T$ be a simple disk triangulation, and $\cP = \{D_v:v\in\cV T\}$ be a collection of closed disks with disjoint interiors. Suppose for each $v_1,v_2\in\cV\cT$ where $\{v_1,v_2\}\in\cE T$, $D_{v_1}$ and $D_{v_2}$ are tangent, while if $D_{v_1},D_{v_2}$ are tangent and $v_1,v_2$ are not both on the boundary of $T$, then $\{v_1,v_2\}\in\cE T$.  We further connect centers of $D_{v_1}$ and $D_{v_2}$ for every $\{v_1,v_2\}\in\cE T$. We call $\cP$ a \emph{circle embedding} of $T$, if this drawing defines an orientation-preserving embedding of $T$. We define the \emph{carrier} of $\cP$ to be the union of $\Delta_f$ over all faces $f$ of $T$ where $\Delta_f$ is  the triangle formed by connecting the Euclidean centers $D_v$ of the disks for the vertices on $f$.  In other words, a circle embedding of $T$ is a circle packing of $T$ except that we allow two disks for a pair of boundary vertices of $T$ not connected by an edge to be tangent.   The reason that we introduce circle embedding here is, in our later proof, we will compare circle packings for a map $T$ and its submap $T_1$, while the restriction of a circle packing of $T$ to disks for vertices within $T_1$ is only a circle embedding of $T_1$ but not necessarily a circle packing of $T_1$.   Different from~\cite{HeSchramm}, below we will work with homeomorphisms between carriers of circle embeddings.


For a circle embedding $\cP$ for $T$, we define $\partial\cP$ for the union of the disks for boundary vertices. The following is an extension of~\cite[Theorem 1.1]{HeSchramm}, where we allow some tiny portion of disks in $\cP^n$ and $\wt\cP^n$ to  lie outside $D$ and $\wt D${; see assertion (ii) below}.
\begin{theorem}\label{thm:He-Schramm-extension-0}
   Let $D,\wt D\subsetneq \bbC$ be    {bounded Jordan} domains. Let $\{T^n\}_{n\geq1}$ be  disk triangulations, $\cP^n, \wt\cP^n$ be   {circle embeddings} of $T^n$, and $f^n$ be homeomorphisms from the   {carrier} of $\cP^n$ to that of $\wt\cP^n$, sending  each $D_v\cap \mathrm{carrier}(\cP^n)$ where  $v\in\cV T^n$ to $\wt D_v\cap \mathrm{carrier}(\wt \cP^n)$. Assume the following hold:
   \begin{enumerate}[(i)]
       \item The maximal spherical diameter of the disks in $\cP^n$ converges to 0 as $n\to\infty$;
       \item The distance between  the union of boundary disks of $\cP^n$ and $D$  converges to 0 as $n\to\infty$, in the sense that
       \eqb\label{thm:He-Schramm-extension-0--} \lim_{n\to\infty}\max_{{D^n\in\partial\cP^n}}\inf_{z\in D^n}d_{\sph}(z,\partial D) = 0,\eqe
       and the same holds for $\wt\cP^n$ and $\wt D$ instead of $\cP^n$ and $D$. Moreover, the maximal spherical diameter of the disks in $\wt\cP^n$ intersecting $\bbC\backslash \ol{\wt D}$ converges to 0 as $n\to\infty$; 
       \item For some $p_0\in D$, for some $N_0>0$, the sequence $\{f^n(p_0):n\geq N_0\}$ is compact in $\wt D$.
   \end{enumerate}
   Then the conditions (i) and (iii) hold for $(\wt\cP^n,\wt D)$ in place of $(\cP^n,D)$, and there exists a subsequence of $\{f^n\}$ converging uniformly to a conformal map $f:D\to\wt D$ on compact subsets of $D$. 
\end{theorem}

The convergence of the circle embedding maps $f^n$ above is  in local uniform topology. We will further prove that, if the curves formed by the boundary disks of $\cP^n$ and $\wt\cP^n$ converges to $\partial D$ and $\partial \wt D$ modulo time parameterization, then the convergence above can be extended to global uniform topology. To work with uniform convergence of functions $\{f^n\}$ defined on different domains,   we   use the following definition. 
\begin{definition}\label{def:unif-convergence}
    Let  $D,D^n\subseteq\bbC$ be bounded Jordan domains such that the Hausdorff distance between $\ol {D^n}$ and $\ol D$ tends to 0 as $n\to\infty$.   Let $f^n:\ol {D^n} {\to\bbC}$ and $f:\ol D\to\bbC$.  We say $\{f^n\}$ is equicontinuous, if for any $\e>0$, there exists $N>0$ and $\delta>0$, such that for any $n>N$ and $x,y\in \ol {D^n}$ with $|x-y|<\delta$, $|f^n(x)-f^n(y)|<\e$. We say $f^n$ converges to $f$ uniformly, if for each $\e>0$, there exists $\delta>0$ and $N>0$, such that for each $n>N$, $z'\in \ol {D^n}$ and $z\in \ol D$ with $|z'-z|<\delta$, we have $|f^n(z')-f(z)|<\e$.
\end{definition}

We also recall the topology on curves viewed modulo time parameterization. 
Let $\beta_1:[0,T_{\beta_1}]\to\bbC$ and $\beta_2:[0,T_{\beta_2}]\to\bbC$ be continuous curves defined on possibly different time intervals. We define 
\begin{equation}\label{eq:curve-metric-0}
    \mathbb{d}^{\rm CMP}(\beta_1,\beta_2):=\inf_\phi \sup_{t\in[0,T_{\beta_1}]}|\beta_1(t)-\beta_2(\phi(t))|,
\end{equation}
where the infimum is over all increasing homeomorphisms $\phi:[0,T_{\beta_1}]\to [0,T_{\beta_2}]$ and ``CMP'' stands for ``curves modulo parameterization".

\begin{theorem}\label{thm:He-Schramm-extension-1}
     Let $D,\wt D,\{\cP^n\},\{\wt\cP^n\}$ be as in Theorem~\ref{thm:He-Schramm-extension-0}. Further assume: 
   \begin{enumerate}[(iv)]
       \item Let $\gamma_n$ (resp.\ $\wt \gamma_n$) be the loop formed by connecting the centers of adjacent boundary disks of $\cP^n$ (resp.\ $\wt\cP^n$) using straight line segments. Then the distance~\eqref{eq:curve-metric-0} between $\gamma_n$ and $\partial D$ (resp.\ $\wt\gamma_n$ and $\wt\partial D$) tends to 0 as $n\to\infty$.
   \end{enumerate}
   Then $\{f^n\}$ is equicontinuous and the  convergence of a subsequence of  $\{f^n\}$ in Theorem~\ref{thm:He-Schramm-extension-0} to $f$ is uniform  {in the sense of Definition~\ref{def:unif-convergence}.} 
\end{theorem}

\begin{remark}
    
    
    As commented after~\cite[Theorem 1.1]{HeSchramm}, once we apply some appropriate normalization to $\{f^n\}$, the sequence $\{f^n\}$ (rather than a subsequence) would converge. 
\end{remark}

In Theorem~\ref{thm:He-Schramm-extension-0}, we required that $D,\wt D\neq\bbC$. The following is a version of Theorem~\ref{thm:He-Schramm-extension-0} when $D=\wt D=\bbC$.

\begin{theorem}\label{thm:He-Schramm-sphere}
     Let $\{T^n\}_{n\geq1}$ be  disk triangulations, $\cP^n, \wt\cP^n$ be circle embeddings of $T^n$ in $\bbC$, and $f^n$ be homeomorphisms from the  {carrier} of $\cP^n$ to that of $\wt\cP^n$, sending each  sending  each $D_v\cap \mathrm{carrier}(\cP^n)$ where  $v\in\cV T^n$ to $\wt D_v\cap \mathrm{carrier}(\wt \cP^n)$.  Assume the following hold:
     \begin{enumerate}[(i)]
         \item The maximal spherical diameter of disks in $\cP^n$ as well as the maximal spherical diameter of  the bounded interstices of $\cP^n$  tends to 0 as $n\to\infty$;
         \item For every $R>0$, for large enough $n$, $B(0;R)\subseteq \carrier(\cP^n)$; 
         \item For each $n$, $f^n(0)=0$ and $f^n(1)=1$.  
     \end{enumerate}
     Then (i) and (ii) above hold for $\wt\cP^n$ in place of $\cP^n$, and $f^n$ converges uniformly to the identity in the spherical metric, in the sense that for any $\e>0$, there exists some $N>0$, such that for every $n>N$ and $z\in\mathrm{carrier}(\cP^n)$, $d_{\sph}(f^n(z),z)<\e$.
\end{theorem}

Since circle packing is non-local, i.e., circle packing of maps could change very differently if we add vertices and edges, it is not straightforward to compare circle packings for maps and their submaps. Our next result gives some condition where the circle packing of a submap $T^n$ of disk triangulation $\wt T^n$ is close to the circle packing for $\wt T^n$. Roughly speaking, let $v^n$ be an interior vertex of $T^n$. Suppose in the circle packing $\cP^n$ for $T^n$ in $\bbD$  {centered at $v^n$},  the maximal diameter of disks converges to 0, and the diameter of the union of the vertices on the interface between $T^n$ and $\wt T^n\backslash T^n$ is small, then in the circle packing $\wt\cP^n$ for $\wt T^n$ centered at $v^n$, the diameter of the union of  the disks for $\wt T^n\backslash T^n$ and the circle packing map between   $\cP^n$ and $\wt\cP^n$ is approximately identity up to rotations.  See Figure~\ref{fig:pack-disk-1} for an illustration.

\begin{figure}
    \centering
    \includegraphics[scale=0.78]{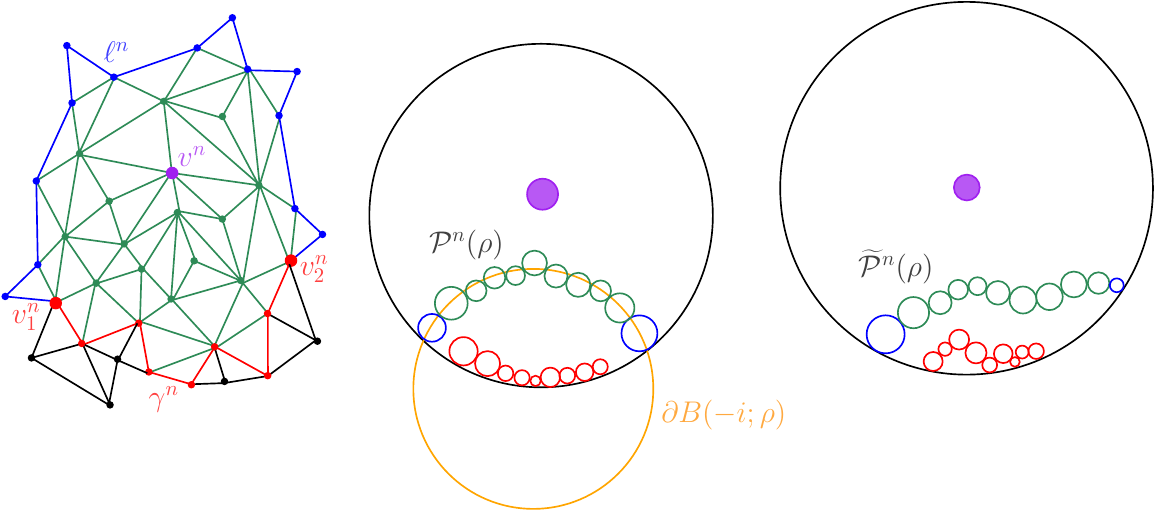}
    \caption{An illustration of Theorem~\ref{thm:He-Schramm-disk}. \textbf{Left:} The triangulation $\wt T^n$. The map $T^n$ is formed by the red path $\gamma^n$, the blue path $\ell^n$ and the green vertices and edges between $\gamma^n$ and $\ell^n$. \textbf{Middle:} (Part of) the circle packing $\cP^n$. The set $\cP^n(\rho)$ is formed by the disks intersecting $\partial B(-i;\rho)$. \textbf{Right:} By Lemma~\ref{lem:He-Schramm}, the disks for vertices on $\gamma^n$ in $\wt\cP^n$ are separated from 0 by $\wt\cP^n(\rho)$ which has small path, which implies that in $\wt\cP^n$, the  union of disks for the vertices on $\gamma^n\cup (T^n\backslash\wt T^n)$ has small diameter. Theorem~\ref{thm:He-Schramm-disk} further follows from Theorem~\ref{thm:He-Schramm-extension-0} and Theorem~\ref{thm:He-Schramm-extension-1}.}
    \label{fig:pack-disk-1}
\end{figure}

\begin{theorem}\label{thm:He-Schramm-disk}
    Let $\wt T^n$ be disk triangulations, $v_1^n$ and $v_2^n$ be boundary vertices of $\wt T^n$, and $\ell^n$ be the clockwise boundary arc of $\partial\wt T^n$ from $v_1^n$ to $v_2^n$. Let $\gamma^n$ be a simple path on $\wt T^n$ that intersects $\ell^n$ only at $v_1^n$ and $v_2^n$. Let $ T^n\subseteq \wt T^n$ be the disk triangulation formed by $\ell^n$, $\gamma^n$ and vertices between $\ell^n$ and $\gamma^n$.  Let $v^n$ be an interior vertex of $ T^n$, and $\wt\cP^n$ and $\cP^n$ be circle packing of $\wt T^n$ and $T^n$ in $\bbD$ centered at $v^n$. Further assume  the following:
    \begin{enumerate}[(i)]
        \item The maximal diameter of disks in $\cP^n$ tends to 0 as $n\to\infty$;
        \item The diameter of the union of the disks in $\cP^n$ for vertices in $\gamma^n$ tends to 0 as $n\to\infty$.
    \end{enumerate}
    Let $\cP_1^n$ be the restriction of $\wt\cP^n$ to the disks for  vertices in $ T^n$, and $f^n$ be a homeomorphism between the  {carriers} of $\cP^n$ and $\cP_1^n$   sending  each $D_v\cap \mathrm{carrier}(\cP^n)$ where  $v\in\cV T^n$ to $\wt D_v\cap \mathrm{carrier}( \cP_1^n)$.  Then the maximal diameter of disks in $\wt\cP^n$ as well as the diameter of the union of disks in $\wt\cP^n$ for vertices in $\gamma^n\cup(\wt T^n\backslash T^n)$ tends to 0 as $n\to\infty$, and there exists $\theta_n\in [0,2\pi)$ such that $e^{i\theta_n} f^n$ converges uniformly to the identity in the sense of Definition~\ref{def:unif-convergence}.
\end{theorem}

We will prove Theorems~\ref{thm:He-Schramm-extension-0},~\ref{thm:He-Schramm-extension-1},  ~\ref{thm:He-Schramm-sphere} and~\ref{thm:He-Schramm-disk} in Section~\ref{subsec:HSextension-pf}. Our proof is mostly based on a length/area argument for circle packing, which originated in~\cite[Equation (2.2)]{HeSchramm}.   The precise statement is in Lemma~\ref{lem:He-Schramm}. It can be viewed as a discrete version of  Wolff's lemma~\cite{wolff1935demonstration} (Lemma~\ref{lem:He-Schramm-unif}), which has numerous applications in complex analysis, e.g., proving that conformal maps between Jordan domains can be continuously extended to the boundary. This type of discrete length/area argument  {first appeared in~\cite{RodinSullivan87}, and} was also used in~\cite{GDN19VEL}, which establishes a relation between maximal diameter of disks in circle packing with Cannon's vertex extremal length~\cite{CannonVEL}, and also in the proof of Theorem~\ref{thm:invariance-uniformization} in our work~\cite{HY26} to establish regularity estimates for the Riemann surface $M(\cH)$.  {We comment that Theorems~\ref{thm:He-Schramm-extension-0},~\ref{thm:He-Schramm-extension-1},  ~\ref{thm:He-Schramm-sphere} and~\ref{thm:He-Schramm-disk} also apply for random triangulations; if the conditions in these theorems hold for  the convergence in probability, and we normalize on $f^n$ such that the convergence holds for the whole sequence $\{f^n\}$ rather than a subsequence, then the corresponding conclusions also holds for convergence in probability as well; one can easily see this by taking subsequences and applying these results in the setting of almost sure convergence.} 

\subsubsection{Overview of proof}\label{subsubsec:intro-overview}

In the setting of Theorem~\ref{thm:pack-cone}, the collection $\{\eta([x-\frac{1}{n},x]):x\in\frac{1}{n}\bbZ\}$ defines a cell configuration with the associated map $\cG^{1/n}$ (where we set the conductance for all edges to be equal to 1), which can be viewed as the a priori embedding of the mated-CRT map $\cG^{1/n}$. However, it does not fit into the first two conditions in Definition~\ref{def:CellConfigConstraint}. To this end, we will first work on a variant $\cH^\rmE$   in the setting of 0-quantum cone in Section~\ref{subsec:mated-CRT-0-cone}, and verify that it satisfies the conditions in Theorem~\ref{thm:CellSystemPack} and Theorem~\ref{thm:invariance-uniformization}. In~\cite[Theorem 4.1]{GMS21tutte}, it has been shown that in the cell configuration $\cH^\rmE$, all the positive moments for $\diam(H_0)^2/\area(H_0)$ are finite, and $\bbE[\deg(H_0)^p]<\infty$ for all $p<4/\gamma^2$.  In the setting of circle packing, due to the presence of multiple edges in the mated-CRT map, we will prove a version of the finiteness of positive moments of  $\diam(H_0)^2/\area(H_0)$ in Proposition~\ref{prop:cell-diam2/area} for the version of cell configurations where we collapse all the multiple edges.  {We will also prove that  $\bbE[\deg(H_0)^p]<\infty$ for all $p>0$ (Proposition~\ref{prop:0-cone-deg-1}) using a comparison between 0-quantum cone and $\gamma$-quantum cone in the Tutte draft.} This establishes the moment bound in~\eqref{eq:thm-inv-circle}. The other conditions in Definition~\ref{def:CellConfigConstraint} can be verified similarly as in~\cite{GMS21tutte}, which allows us to prove a version of Theorem~\ref{thm:pack-cone} and Theorem~\ref{thm:unif-cone} in the context of 0-quantum cone (Theorem~\ref{thm:circle-matedCRT-0-cone} and Theorem~\ref{thm:unif-matedCRT-0-cone}). 

To transfer from the setting of 0-quantum cone to the surfaces in Theorem~\ref{thm:pack-cone}-\ref{thm:unif-sph}, we cut a piece $\eta([a,b])$ from the 0-quantum cone. From the convergence result on 0-quantum cone, the restriction of the circle packing of the whole map to the times in $[a,b]$ is close to $\eta|_{[a,b]}$, and we can apply Theorem~\ref{thm:He-Schramm-extension-1} to transfer to the setting where $\eta|_{[a,b]}$ is embedded in $\bbD$. This gives a version of convergence of circle packing for the mated-CRT map encoded by $\eta|_{[a,b]}$ that is measurable with respect to $\eta|_{[a,b]}$, and we further use standard methods in~\cite[Proof of Proposition 2.10]{GMS19har} to transfer to the setting where $\eta|_{[a,b]}$ is a piece cut off from the $\gamma$-quantum cone. This is done in Proposition~\ref{thm:circle-quantum-cell-0}. We comment that the operator $\Phi_{\rm simple}$ is non-local, and one has to be careful when defining localized versions of convergence of circle packing for quantum cells; see Figure~\ref{fig:quantum-cell}. Now for the quantum disk and quantum sphere in Theorem~\ref{thm:pack-disk} and Theorem~\ref{thm:pack-sph}, following the local absolute continuity between the Brownian motions,  the law of the surfaces encoded by the space-filling curve restricted to time $[\e,1-\e]$ are mutually absolutely continuous. This allows us to combine Proposition~\ref{thm:circle-quantum-cell-0} along with the comparison between different circle packings as in Theorem~\ref{thm:He-Schramm-sphere} and Theorem~\ref{thm:He-Schramm-disk} to extend the interval $[\e,1-\e]$ to $[0,1]$ and complete the proof of Theorem~\ref{thm:pack-disk} and Theorem~\ref{thm:pack-sph}. Theorem~\ref{thm:pack-cone} is proven by starting with the convergence in Proposition~\ref{thm:circle-quantum-cell-0} for $\eta|_{[-1,1]}$ and zooming in near 0, where we further apply the scale invariance of the quantum cone (Lemma~\ref{lem:Rb}) and Theorem~\ref{thm:He-Schramm-sphere}. This finishes the proof of the convergence of the mated-CRT map under circle packing in various (simply connected) topologies. 

The proof of for the convergence in the Riemann uniformization embedding is mostly parallel to the circle packing case, and is actually easier since we no longer need to collapse the double edges as in the circle packing case. In Section~\ref{subsec:conformal-map}, we gather all the ingredients we need on conformal maps, which are the analogs of  Theorems~\ref{thm:He-Schramm-extension-1}, Theorem~\ref{thm:He-Schramm-sphere} and~\ref{thm:He-Schramm-disk}. In Section~\ref{subsec:unif-matedCRT-gamma-cone}, we follow the same outline as in the proof for circle packing setting to complete the proof of Theorems~\ref{thm:unif-cone}-\ref{thm:unif-sph}.

\textbf{Notation.} We write $\bbN_0=\{0,1,2,3,....\}$, and write $\bbR_+ = [0,\infty)$.   For nonnegative functions $f,g:\bbN_0\to \bbR_+$, we write $f(n)\lesssim g(n)$, if there are some constants $M,N>0$ such that $f(n)\leq Mg(n)$ for every $n>N$.

\medskip

\noindent\textbf{Acknowledgements.}   
N.H.\ was supported by grant DMS-2246820 and DMS-2553117 of the National Science Foundation, along with the Simons Collaboration Grant \emph{Probabilistic Paths to Quantum Field Theory}.

\medskip

\textbf{Statement:} We have worked on this paper for more than one year and are currently in the process of revising it. We were prompted to post the paper today since several closely related results were released by OpenAI on October 6. To our knowledge, none of the OpenAI results imply the results of our paper, or vice versa.

\medskip

\textbf{AI usage:} We used ChatGPT Astra for proofreading  and revising the first draft of the paper. AI tools were not used to develop mathematical ideas or in the writing of the manuscript. All suggested changes were reviewed by the authors, who take full responsibility for the content of the paper.

\section{Preliminaries on Liouville quantum gravity and mating of trees}\label{sec:pre}

In this section, we recap the definitions of Gaussian free field, Liouville quantum gravity and space-filling SLE that are needed for this paper. Since there have been numerous works on these topics, we will be relatively brief here; see the surveys and lecture notes~\cite{She07,BN23SLEnotes,BP24GFFnotes} for more details on GFF, SLE and LQG. 

\subsection{Gaussian free field}

Let $X \in\{\bbC,\bbH\}$. Let $m_\bbC$ (resp.\ $m_\bbH$) be the uniform measure on the unit circle {$\partial\bbD$} (resp.\ unit semicircle $\{|z|=1,\mathrm{Im}\, z>0\}$). Define the Dirichlet inner product $\langle f,g\rangle_\nabla = (2\pi)^{-1}\int_X \nabla f\cdot\nabla g $ on the space $\{f\in C^\infty(X):\int_X|\nabla f|^2<\infty; \  \int f(z)m_X(dz)=0\},$ and let $H(X)$ be the closure of this space {with respect to} the inner product $\langle f,g\rangle_\nabla$. Let $(f_n)_{n\ge1}$ be an orthonormal basis of $H(X)$, and $(\alpha_n)_{n\ge1}$ be a collection of independent standard Gaussian variables.  Then the summation
$$h_X = \sum_{n=1}^\infty \alpha_nf_n$$
a.s.\ converges in the space of distributions on $\bbC$, and $h_X$ is the \emph{Gaussian Free Field} on $X$ normalized such that $\int h(z)m_X(dz) = 0$. Let $P_X$ be the law of $h_X$. See~\cite[Section 4.1.4]{DMS14} for more details.  As argued in \cite[Section 4.1]{DMS14},  we have the decomposition $H({X}) = H_1({X})\oplus H_2({X})$, where $H_1(X)$ is the subspace of radially symmetric functions, and $H_2(X)$ is the subspace of functions having mean 0 about all  circles $\{|z|=r\}$ if $X=\bbC$ or all  semicircles $\{|z|=r, \ \mathrm{Im}\, z>0\}$ if $X=\bbH$. As a consequence, for the GFF $h_X$ sampled from $P_X$, we can decompose \eqb\label{eq:GFF-decompose}h_X=h_{X,1}+h_{X,2}\eqe where $h_{X,1}$ and $h_{X,2}$ are independent distributions given by the projection of $h$ onto $H_1(X)$ and $ H_2(X)$, respectively. 

\subsection{Liouville quantum gravity}

Let $\gamma\in(0,2)$ and $Q=\frac{2}{\gamma}+\frac{\gamma}{2}$. Consider the space of pairs $(D,h)$, where $D\subseteq \mathbb C$ is a planar domain and $h$ is a distribution on $D$ (often some variant of the GFF). For a conformal map $g:D\to\tilde D$ and a generalized function $h$ on $D$, define the generalized function $g\bullet_\gamma h$ on $\tilde{D}$ by setting 
\begin{equation}\label{eq:lqg-changecoord}
    g\bullet_\gamma h:=h\circ g^{-1}+Q\log|(g^{-1})'|.
\end{equation}
 Define the equivalence relation $\sim_\gamma$ {as follows. We say that} $(D, h)\sim_\gamma(\wt{D}, \wt{h})$ if there is a {conformal map $g:D\to\wt{D}$ such that $\tilde h = g\bullet_\gamma h$}.
A \textit{quantum surface} $S$ is an equivalence class of pairs $(D,h)$ under the {equivalence} relation $\sim_\gamma$, and we say {that} $(D,h)$ is an \emph{embedding} of $S$ if $S = (D,h)/\mathord\sim_\gamma$.   {Abusing notation, we will sometimes call $(D,h)$ as a quantum surface, and we are then referring to the equivalence class $(D,h)/\sim_\gamma$ that it defines.} Likewise, a \emph{quantum surface with} $k$ \emph{marked points} is an equivalence class of {tuples} 
of the form
	$(D, h, x_1,\dots,x_k)$, where $(D,h)$ is a quantum surface, the points  $x_i\in {\overline{D}}$, and with the further
	requirement that marked points (and their ordering) are preserved by the conformal map $\varphi$ in \eqref{eq:lqg-changecoord}, and a \emph{curve decorated quantum surface} can be defined analogously. 

     For a $\gamma$-quantum surface $(D, h, z_1, ..., z_m)$, its \textit{quantum area measure} $\mu_h$ is defined by taking the weak limit $\epsilon\to 0$ of $\mu_{h_\epsilon}:=\epsilon^{\frac{\gamma^2}{2}}e^{\gamma h_\epsilon(z)}d^2z$, where $d^2z$ is the Lebesgue area {measure} and $h_\epsilon(z)$ is the circle average of $h$ over $\partial B(z, \epsilon)$. When $D=\mathbb{H}$, we can also define the  \textit{quantum boundary length measure} $\nu_h:=\lim_{\epsilon\to 0}\epsilon^{\frac{\gamma^2}{4}}e^{\frac{\gamma}{2} h_\epsilon(x)}dx$ where $h_\epsilon (x)$ is the average of $h$ over the semicircle $\{x+\epsilon e^{i\theta}:\theta\in(0,\pi)\}$. It has been shown in \cite{DS11, SW16} that all these weak limits are well-defined  for the GFF and its variants we are considering in this paper, {and that} $\mu_{h}$ and $\nu_{h}$ {can} be conformally extended to other domains using the relation $\bullet_\gamma$.

Now we introduce the quantum surfaces we will  frequently consider in this paper. 
\begin{definition}[$\alpha$-quantum cone]\label{def:alpha-cone}
  Fix $\gamma\in(0,2)$, $\alpha<Q$ and let $(B_s)_{s\ge0}$ and $(\wt{B}_s)_{s\ge0}$ be independent standard one-dimensional Brownian motions. 
  Let 	
			\begin{equation*}
				Y_t=\left\{ \begin{array}{rcl} 
					B_{t}+\alpha t  & \mbox{for} & t\ge 0,\\
					\wt{B}_{-t} +\alpha t & \mbox{for} & t<0,
				\end{array} 
				\right.
			\end{equation*}
			conditioned on   $\wt{B}_{t} + (Q-\alpha)t>0$ for all $t>0$.  Recall the decomposition in~\eqref{eq:GFF-decompose}. The $\alpha$-quantum cone is the doubly marked quantum surface $(\bbC,h,0,\infty)/\sim_\gamma$, where $h=\psi_1+\psi_2$, and $(\psi_1,\psi_2)$ are independent random generalized functions in $H_1(\bbC)\times H_2(\bbC)$ whose law can be described as follows.  The common value of $\psi_1$ on $\{|z|=e^{-t}\}$ is equal to $Y_t$ for all $t\in\bbR$, and $\psi_2$ has the same law as $h_{\bbC,2}$ as in~\eqref{eq:GFF-decompose}.
\end{definition}
The embedding $(\bbC,h,0,\infty)$ above is also called the circle average embedding of the $\alpha$-quantum cone.

A \emph{quantum disk} is a  finite-mass quantum surface $(\bbD,h^{\mathrm{disk}})/\sim_\gamma$ defined in~\cite[Definition 4.21]{DMS14} (the precise definition is not needed here).  One can consider quantum disks with specified boundary length $\nu_{h^{\mathrm{disk}}}(\partial\bbD)$ and area $\mu_{h^{\mathrm{disk}}}(\bbD)$, as well as  quantum disks with any number of marked boundary points and interior points sampled uniformly from $\nu_h$ and $\mu_h$, respectively. We will typically work with quantum disks with unit area and boundary length and a single marked boundary point. A \emph{quantum sphere} is a  finite-mass quantum surface $(\wh\bbC,h_{\mathrm{sp}})/\sim_\gamma$ defined in~\cite[Definition 4.21]{DMS14}, where $\wh\bbC = \bbC\cup\{\infty\}$, and one can consider quantum spheres with specified area $\mu_{h_{\mathrm{sp}}}(\wh\bbC)$ as well. 

\subsection{Space-filling SLE and mating of trees}\label{subsubsec:MatedCRT}
The Schramm-Loewner evolution (SLE) is a family of fractal curves with a single parameter $\kappa>0$~\cite{Sch00}. $\SLE_\kappa$ is a simple curve for $\kappa\in(0,4)$, self-touching but not space-filling or self-crossing for $\kappa\in(4,8)$, and space-filling for $\kappa\geq8$. As studied in the theory of Imaginary Geometry~\cite{ig1,ig4}, $\SLE_\kappa$ curves can be viewed as flow lines of GFF, where in the whole plane case, we work with whole plane GFF modulo a global additive multiple of $2\pi(\sqrt{\kappa}/{2}-2/\sqrt{\kappa})$. For $\kappa>4$, space-filling $\SLE_\kappa$ is a variant of $\SLE_\kappa$ introduced in~\cite{ig4}, which is also generated via GFF flow lines and agrees with ordinary $\SLE_\kappa$ when $\kappa\geq8$. In this paper, we will work with two kinds of space-filling $\SLE_\kappa$ curves:  the space-filling $\SLE_\kappa$ loop $\eta$ in $\wh\bbC$ from $\infty$ to $\infty$, and the clockwise space-filling $\SLE_\kappa$ loop $\eta^\disk$ in $\ol\bbD$ from 1 to 1. In both setting, one can define an ordering on $\bbQ^2$ or $\ol\bbD\cap\bbQ^2$ by considering flow lines of GFF started at these points, and further extend this ordering to a continuous space-filling curve. See~\cite[Section 4.3]{ig4} and~\cite[Section 2.2]{AG19} for the precise definition of the above two space-filling $\SLE_\kappa$ curves.

\begin{figure}
    \centering
    \includegraphics[scale=0.68]{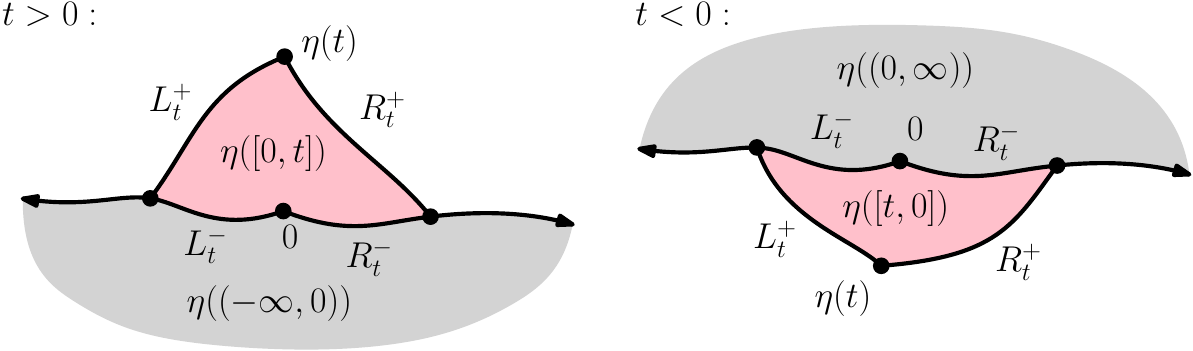}
    \caption{An illustration of the boundary length process in mating of trees when $\gamma\in(0,\sqrt2]$ and $\kappa=16/\gamma^2\geq 8$. The processes $(L_t,R_t)_{t\in\bbR}$ records the change of the left and right boundary $\gamma$-LQG length of $\eta((-\infty,t])$, so $L_t=L_t^+-L_t^-$ and $R_t=R_t^+-R_t^-$.}
    \label{fig:quantumcell3}
\end{figure}

Now we briefly recap the theory of mating of trees as in~\cite{DMS14,MS-mot-finite,AG19}. Let $\gamma\in(0,2)$, $(\bbC,h,0,\infty)$ be a $\gamma$-quantum cone as in Definition~\ref{def:alpha-cone} for $\alpha=\gamma$ with circle average embedding, $\eta$ be a space-filling $\SLE_{16/\gamma^2}$ curve in $\wh\bbC$ from $\infty$ to $\infty$ and  parameterized by the $\gamma$-LQG mass with respect to $h$ such that $\eta(0)=0$, as in the discussion of the a priori embedding for $\cG^{1/n}$ discussed in Section~\ref{subsubsec:intro-mated-CRT}. For $t\in\bbR$, let $L_t$ (resp.\ $R_t$) be the change of the $\gamma$-LQG length of the left (resp.\ right) boundary of $\eta((-\infty,t])$ with respect to time 0. See Figure~\ref{fig:quantumcell3} for an illustration when $\gamma\in(0,\sqrt2]$. By~\cite[Theorem 1.9]{DMS14}, $(L_t,R_t)_{t\in\bbR}$ is a 2D Brownian motion with variance~\eqref{eq:variance} (the correlation for $\gamma\in(0,\sqrt{2})$ was derived in~\cite{GHMS15} and the variance $\mathbbm{a}^2$ was computed in~\cite{ARS21}), and by~\cite[Theorem 1.11]{DMS14}, the process  $(L_t,R_t)_{t\in\bbR}$ a.s.\ determines $(h,\eta)$ up to a rotation.  Now let $(\bbD,h^\disk,1)$ be an embedding of quantum disk of unit area and unit boundary length with a single marked point located at point 1, and $\eta^\disk:[0,1]\to\ol\bbD$ be a counterclockwise space-filling $\SLE_\kappa$ loop in $\ol \bbD$ rooted at 1 and parameterized by the $\gamma$-LQG mass with respect to $h^\disk$. We let $L_t^\disk$ and $R_t^\disk$ be the left and right boundary length of $\eta^\disk([t,1])$.    By~\cite[Theorem 1.1]{AG19}, $(L_t^\disk,R_t^\disk)$ evolves as a Brownian motion of variance~\eqref{eq:variance} that starts from $(0,1)$ and is conditioned on staying within $\bbR_+^2$ and exits $\bbR_+^2$ at $(0,0)$ at time 1. Similarly, let $(\wh\bbC,h^\sph,0)$ be an embedding of quantum sphere of unit area a  with a single marked point located at point $\infty$, and $\eta^\sph:[0,1]\to\wh\bbC$ be a counterclockwise space-filling $\SLE_\kappa$ loop in $\wh\bbC$ rooted at $\infty$ and parameterized by the $\gamma$-LQG mass with respect to $h^\sph$. We let  $L_t^\sph$ and $R_t^\sph$ be the left and right boundary length of $\eta^\sph([0,t])$. Then by~\cite[Theorem 1.1]{MS-mot-finite}, $(L_t^\sph,R_t^\sph)$ evolves as a Brownian motion of variance~\eqref{eq:variance} that starts from $(0,0)$ and is conditioned on staying within $\bbR_+^2$ and exits $\bbR_+^2$ at $(0,0)$ at time 1.

\subsection{Quantitative bounds on SLE and LQG}

In this section we gather estimates on SLE and LQG that will  be used in the remainder of the paper. We will use the following set of notations. For an interval $[a,b)$, we call the intervals $[a,\frac{a+b}{2})$ and $[\frac{a+b}{2},b)$ the dyadic children of $[a,b)$, and we say $[a,b)$ is an dyadic parent of $[a,\frac{a+b}{2})$ or $[\frac{a+b}{2},b)$. We say $I\subseteq I'$ is a dyadic descendant of $I'$, if there is $I= I_1\subseteq I_2\subseteq\cdots\subseteq I_n = I'$ such that for each $j=1,...,n-1$, $I_j$ is a dyadic children of $I_{j+1}$.
\begin{definition}\label{def:1ddyadic}
    Let $\theta$ be uniformly randomly sampled in $[0,1)$. Let $I(0) = [\theta-1,\theta)$. For $k\in\bbN$, inductively let $I(k)$ be sampled uniformly from the two dyadic parents of $I(k-1)$.  The \emph{1D uniform dyadic system} $\fD$ is defined by the set of all dyadic descendants of $I(k)$ with length at least 1 for each $k\in\bbN$. We write  $\Omega_\fD$ for the space of all possible realizations of $\fD$. For $s\in\bbR$, we write $\fD-s$ for the set $\{[a-s,b-s):[a,b)\in \fD\}$.
\end{definition}

We also introduce the measures $\bbP_\rmE$ and $\bbP_\rmQ$ describing the laws of space-filling $\SLE_\kappa$ decorated 0-quantum cone and $\gamma$-quantum cone, respectively. Here the letters ``$\rmE$" (resp.\ ``$\rmQ$") is chosen since it is the first letter of the word ``Euclidean" (resp.\ ``Quantum").
\begin{definition}\label{def:P_EP_Q}
    Let $\bbP_\rmE$ be the law of $(h, \eta,\fD)$ where $(\bbC,h^\IG,h,0,\infty)$ is the  {circle average embedding} of the 0-quantum cone as in Definition~\ref{def:alpha-cone}, $h^\IG$ is a whole plane GFF modulo an integer multiple of $2\pi(2/\gamma-\gamma/2)$, $\eta$ is a whole-plane space-filling $\SLE_\kappa$ generated from $h^\IG$, $\fD$ is a uniform dyadic system as in Definition~\ref{def:1ddyadic},  and the three random variables  $h,\eta,\fD$ are all independent. Let $\theta\sim\mathrm{Unif}([0,1])$ be the uniform random variable generating $\fD$ as in Definition~\ref{def:1ddyadic}. We parametrize $\eta$ by the $\gamma$-LQG area induced by $h$ such that $\eta(0)=0$. Let $\bbP_\rmQ$ be defined in the exact same way, except that $(\bbC,h,0,\infty)$ is the circle average embedding of the $\gamma$-quantum cone instead of the 0-quantum cone. Under either of these probability measures we define the boundary length process $Z$ of  $(h,\eta)$ as in Section~\ref{subsubsec:MatedCRT}.
\end{definition}

Now we are ready to state the quantitative bounds on SLE and LQG. The first property is the scale invariance property for quantum cones, which is from~\cite[Lemma 2.1]{GMS21tutte}.
\begin{lemma}\label{lem:Rb}
    Let $\gamma\in(0,2)$, $\alpha<Q$ and $(\bbC,h,0,\infty)$ be the circle average embedding of an $\alpha$-quantum cone as in Definition~\ref{def:alpha-cone}. Let $\eta$ be an independent space-filling SLE loop from $\infty$ to $\infty$ parameterized by $\mu_h$.  For $b>0$, let 
    \begin{align}
        &\fR_b = \sup\{r>0:h_r(0)+Q\log r = \frac{1}{\gamma}\log b\}; \label{eq:Rb}
        \\  &h^b(\cdot)=h(\fR_b\cdot)+Q\log \fR_b-\frac{1}{\gamma}\log b, \qquad \eta^b(\cdot) = \fR_b^{-1}\eta(b\cdot). \label{eq:hb-etab}
    \end{align}
   (Note that by definition $\fR_1=1$.)  Then $(h^b,\eta^b)$ has the same law as $(h,\eta)$. Moreover, there is a constant $a = a(\alpha,\gamma)>0$ such that for any $b_2>b_1>0$ and $C>1$, 
     \begin{equation}\label{eq:Rb-bound}
      \bbP\Big[C^{-1}(\frac{b_2}{b_1})^{\frac{1}{\gamma(Q-\alpha)}}\leq \frac{\fR_{b_2}}{\fR_{b_1}}\leq C(\frac{b_2}{b_1})^{\frac{1}{\gamma(Q-\alpha)}}\Big]\geq 1-3\exp\big(-\frac{a (\log C)^2}{\log(b_2/b_1)+\log C}\big).
  \end{equation}
\end{lemma}

The second result is the following quantitative bound on cells in mated CRT maps, which is from~\cite[Lemma 2.7]{GMS19har} and ~\cite[Lemma A.9]{Tutte}.
\begin{lemma}\label{lem:diam-cell-max}
 For each $q\in(0,2(2+\gamma)^{-2})$ and $\e\in(0,1)$, 
    $$\bbP_\rmE\Big[\diam(\eta(J))\leq \e^q \ \text{for any interval } J \text{ with } |J|\leq\e/2 \ \text{and }\eta(J)\cap B(0;3/4)\neq\emptyset\Big]\geq 1-\e^{\alpha(q,\gamma)+o_\e(1)},$$
    where the rate of the $o_\e(1)$ depends only on $q$ and $\gamma$, and 
    $$\alpha(q,\gamma):=\frac{q}{2\gamma^2}(\frac{1}{q}-2-\frac{\gamma^2}{2})^2-2q>0.$$
    The same holds for $\bbP_\rmQ$ instead of $\bbP_\rmE$.
\end{lemma}

 {The following results follows from~\cite[Lemma A.5, Lemma A.6]{Tutte}.}
\begin{lemma}\label{lem:SLE-contain}
    Let $\delta>0$. There exists $c>0$ depending only on $\delta$ and $\gamma$ such that  
        \begin{align}
            & \bbP_\rmE [B(0;{2^{-\delta k}}) \nsubseteq \eta([-1,1])] \lesssim 2^{-ck}; \label{eq:SLE-contain-ball}  \\
            & \bbP_\rmE [\eta([-2^{-\delta k},2^{-\delta k}])\nsubseteq B(0;1/2) ]\lesssim 2^{-ck}. \label{eq:SLE-contained-in-ball}
        \end{align}
        {The same holds for $\bbP_\rmQ$ in place of $\bbP_\rmE$.}
\end{lemma}

Recall $\fR_b$ from~\eqref{eq:Rb}, and $h^b,\eta^b$ from ~\eqref{eq:hb-etab}. By~\eqref{eq:Rb-bound}, for any $\delta>0$, for some constant $c$ depending only on $\delta$, we have
\begin{equation}\label{eq:Rb-gammaQ}
  1-  \bbP_\rmE[2^{(\frac{1}{\gamma Q}-\delta)k}\leq \fR_{2^k}\leq  2^{(\frac{1}{\gamma Q}+\delta)k}]\lesssim 2^{-ck}. 
\end{equation}
Since $(h^b,\eta^b)$ has the same distribution as $(h,\eta)$, combined with Lemma~\ref{lem:SLE-contain} we have
\begin{lemma}\label{lem:SLE-contain-1}
    Let $\delta>0$. There exists $c>0$ depending only on $\delta$ and $\gamma$ such that  
        \begin{align}
            & \bbP_\rmE [B(0;{2^{(\frac{1}{\gamma Q}-\delta) k}}) \nsubseteq \eta([-2^k,2^k])] \lesssim 2^{-ck}; \label{eq:SLE-contain-ball-1}  \\
            & \bbP_\rmE [\eta([-2^{ k},2^{k}])\nsubseteq B(0;{2^{(\frac{1}{\gamma Q}+\delta)k}}) ]\lesssim 2^{-ck}. \label{eq:SLE-contained-in-ball-1}
        \end{align}
\end{lemma}

 \section{Circle packing maps and conformal maps}\label{sec:CP-conformal}


\subsection{Boundary perturbation for conformal maps}\label{subsec:conformal-map}


In this subsection, we gather several results on boundary perturbation for conformal maps that will be needed for our paper.

\begin{proposition}\label{prop:boundary-perturb}
    Let $(D^n)_{n\geq1}$,  $(\wt D^n)_{n\geq1}$, $D$ and $\wt D$ be bounded Jordan domains containing 0, and $f^n:D^n\to\wt D^n$ be conformal maps sending 0 to 0 such that $(f^n)'(0)>0$. Suppose $\partial D^n$ converges to $\partial D$ and $\partial \wt D^n$ converges to $\partial \wt D$ in the metric~\eqref{eq:curve-metric-0}. Then $\{f^n\}_{n\geq1}$ is equicontinuous, and $f^n$ converges uniformly to a conformal map $f:D\to\wt D$ with $f'(0)>0$ in the sense of Definition~\ref{def:unif-convergence}.
\end{proposition}

\begin{lemma}\label{lem:unifconvergence-inv}
    Suppose $D^n$ and $D$ are bounded Jordan domains, $f^n:\ol {D^n}\to\bbC$ and $f:\ol D\to\bbC$ are bounded, continuous and injective. If $f^n$ converges to $f$ uniformly in the sense of Definition~\ref{def:unif-convergence}, then $(f^n)^{-1}$ converge to $f^{-1}$ uniformly in the sense of Definition~\ref{def:unif-convergence}.
\end{lemma}

\begin{proof}
Indeed, $f^n(\ol {D^n})$ and $f(\ol D)$ are compact. We   verify that the Hausdorff distance between $f^n(\ol {D^n})$ and $f(\ol D)$ tends to 0 as $n\to\infty$. Fix $\e>0$, using the uniform convergence of $f^n$ to $f$, there exists $\delta,N_1>0$, such that for any $n>N_1$, $x\in \ol {D^n}$ and $x'\in \ol D$ with $|x-x'|<\delta$, $|f^n(x)-f(x')|<\e$. Using that the Hausdorff distance between $\ol {D^n}$ and $\ol D$ tends to 0 as $n\to\infty$, there exists $N_2>0$ such that the Hausdorff distance between $\ol {D^n}$ and $\ol D$ is less than $\delta/2$ for $n>N_2$. Therefore for $n>\max\{N_1,N_2\}$, the Hausdorff distance between $f(\ol D)$ and $f^n(\ol {D^n})$ is less than $\e$.

Now assume on the contrary that $(f^n)^{-1}$ does not converge to $f^{-1}$ uniformly. Then there exists some $\delta>0$, such that for some subsequence $\{k_n\}$, there exists $x_{k_n}\in f^{k_n}(\ol D^{k_n})$ and $\wt x_{k_n}\in f(\ol D)$, we have $|x_{k_n}-\wt x_{k_n}|\to 0$ as $n\to\infty$ and $|y_{k_n}-\wt y_{k_n}|>\e_0$, where $y_{k_n} = (f^{k_n})^{-1}(x_{k_n})$ and $\wt y_{k_n} = f^{-1}(\wt x_{k_n})$. Using compactness, we can pick a further subsequence (still denoted by $\{k_n\}$), such that $y_{k_n}\to y$ and $\wt y_{k_n}\to\wt y$ as $n\to\infty$. Since the Hausdorff distance between $\ol D_n$ and $\ol D$ tends to 0 as $n\to\infty$, we have $y\in \ol D$, and $|y-\wt y|\geq\e_0$ by our assumption. By the continuity of $f$, we have $f(\wt y) = \lim_{n\to\infty} f(\wt y_{k_n}) = \lim_{n\to\infty} \wt x_{k_n}$. Using the uniform convergence of $f^n$ to $f$, we have $|x_{k_n}-f(y)| = |f^n(y_{k_n})-f(y)|\to0$ as $n\to\infty$. Since $|x_{k_n}-\wt x_{k_n}|\to 0$ as $n\to\infty$, we have $f(y) = f(\wt y)$, which contradicts with the injectivity of $f$. 
\end{proof}

\begin{proof}[Proof of Proposition~\ref{prop:boundary-perturb}]
    We start with the case where $D^n = \bbD$. We verify that $\partial \wt D^n$ is uniformly locally connected, in the sense that for any $\e>0$, there exists a $\delta>0$ such that for any $n$ and $a,b\in\partial\wt D^n$ with $|a-b|<\delta$, one of the segment of $\partial\wt D^n$ joining $a$ and $b$ has diameter less than $\e$. We pick $\delta_1>0$, such that for any $a,b\in \partial \wt D$ with $|a-b|<\delta_1$, $a$ and $b$ can be joined by a segment of $\partial \wt D$ with diameter less than $\e/2$; the existence of $\delta_1$ follows from the fact that $\partial \wt D$ is a Jordan curve. If the distance~\eqref{eq:curve-metric-0} between $\partial\wt D^n$ and $\partial \wt D$ is less than $\min\{\delta_1,\e\}/10$, then for any $a,b\in \partial\wt D^n$ with $|a-b|<\delta_2:=\delta_1/4$, $a$ and $b$ can be joined by a segment of $\partial \wt D^n$ with diameter less than $\e$. This verifies that for some $N>0$, the uniform locally connected property for $(\e,\delta_2)$ holds for every $n>N$, and we may select smaller $\delta_2$ such that the property holds for every $n\geq1$. Therefore $\partial \wt D^n$ is uniformly locally connected, and by the proof of (iii)$\Rightarrow$(iv) in~\cite[Theorem 2.1]{BdryBehaviour}, $\bbC\backslash \wt D^n$ is  uniformly locally connected, and by~\cite[Corollary 2.4]{BdryBehaviour} and the Carath\'{e}odory's kernel theorem, $f^n$ converges to $f$ uniformly on $\ol \bbD$. This verifies the case where $D^n=\bbD$. 

    By Lemma~\ref{lem:unifconvergence-inv}, Proposition~\ref{prop:boundary-perturb} holds if  $\wt D^n = \bbD$. For the general setting, consider the conformal maps $\varphi^n:D^n\to\bbD$ and $\wt\varphi^n:\wt D^n\to\bbD$ fixing 0 such that $(\varphi^n)'(0)>0,(\wt\varphi^n)'(0)>0$, and the conformal maps $\varphi:D\to\bbD$ and $\wt\varphi:\wt D\to\bbD$ fixing 0 such that $\varphi'(0)>0,\wt\varphi'(0)>0$. Then $\varphi^n$ converges to $\varphi$ uniformly and $(\wt\varphi^n)^{-1}$ converges to $\wt\varphi^{-1}$ both in the sense of Definition~\ref{def:unif-convergence}. For any $\e>0$, we pick $\delta_3,N_3>0$ such that  $|(\wt\varphi^n)^{-1}(y)-\wt\varphi^{-1}(y')|<\e$ for every $n>N_3$ and $y,y'\in\bbD$ with $|y-y'|<\delta_3$, and pick $\delta_4,N_4>0$ such that  $|\varphi^n(x)-\varphi(x')|<\delta_3$ for every $n>N_4$ and $x\in D^n,x'\in D$ with $|x-x'|<\delta_4$. Then for any $n>\max\{N_3,N_4\}$, $x\in D^n,x'\in D$ with $|x-x'|<\delta_4$, we have $|(\wt\varphi^n)^{-1}(\varphi^n(x))-\wt\varphi^{-1}(\varphi(x'))|<\e$. This finishes the proof.
\end{proof}

The following fundamental estimate is from Wolff~\cite{wolff1935demonstration} and we complete the proof here for completeness.
\begin{lemma}\label{lem:He-Schramm-unif}
  Let $D$ and $\wt D$ be Jordan domains.  Let $\varphi:D\to\wt D$ be a conformal map. Fix $z_0\in\bbC$,  let $\gamma_\rho = \partial B(z_0;\rho)\cap D$, and $L(\rho) =\mathrm{len}(\varphi(\gamma_\rho))$ be the length of the curve $\varphi(\gamma_\rho)$. Then 
  \begin{equation}\label{eq:lem:He-Schramm-unif}
      \int_0^\infty \frac{L(\rho)^2}{\rho}\, d\rho\leq 2\pi\, \area(\wt D).
  \end{equation}
 In particular, for any $0<r_0<r_1$, there exists some $\rho\in[r_0,r_1]$ such that $L(\rho)\leq \Big(2\pi\, \area(\wt D)\big/ \log(r_1/r_0)\Big)^{1/2}.$
\end{lemma}

\begin{proof}
    Since  $D$ and $\wt D$ be Jordan domains, $\varphi$ extends continuously to the boundary. Now for each $\rho$, we have
    $$\frac{L(\rho)^2}{\rho} = \frac{1}{\rho}\Big(\int_{\gamma_\rho} |\varphi'(z)||dz|\Big)^2\leq \frac{1}{\rho}\Big(\int_{\gamma_\rho} |dz|\Big)\Big(\int_{\gamma_\rho}|\varphi'(z)|^2|dz|\Big) \leq 2\pi \Big(\int_{\gamma_\rho}|\varphi'(z)|^2|dz|\Big).$$
     Integrating over $\rho$ and using $\int_D |\varphi'(z)|^2dz = \area(\wt D)$ gives~\eqref{eq:lem:He-Schramm-unif}.
\end{proof}

Now we prove the following corollary of Lemma~\ref{lem:He-Schramm-unif}.

\begin{proposition}\label{prop:unif-disk}
    Let $\theta_n>0$ be a sequence tending to 0, $D^n\subseteq \bbD$ be a Jordan domain,  $\varphi^n:\ol \bbD\to \ol {D^n}$ be a conformal map fixing 0 such that $(\varphi^n)'(0)>0$, and $\varphi^n(e^{i\theta})\in\partial\bbD$ for $\theta\in [\theta_n,2\pi]$. Then the diameter of $\bbD\backslash D^n$ tends to 0, and $\varphi^n$ converges uniformly on $\ol \bbD$ to the identity function. 
\end{proposition}

\begin{proof}
    By Proposition~\ref{prop:boundary-perturb}, it suffices to prove that $\diam(\varphi^n(\{e^{i\theta}:\theta\in[0,\theta_n]\}))\to 0$ as $n\to\infty$. Let $\e>0$ and $\delta = 0.01e^{-\frac{100\pi^2}{\e^2}}$. Suppose $n$ is large enough such that $\theta_n<\delta/5$. We apply Lemma~\ref{lem:He-Schramm-unif} with the conformal map $\varphi^n$ and $r_0=\delta$,$r_1=0.01$. Then there is some $\rho\in[r_0,r_1]$ such that $\mathrm{len}(\varphi^n(\partial B(1;\rho)\cap\bbD))<\e/2$. Since $\rho\in[\delta,0.01]$,   by our assumption, the endpoints of the curve $\varphi^n(\partial B(1;\rho)\cap\bbD)$ are on $\partial\bbD$. It follows that $\{e^{i\theta}:\theta\in[0,\theta_n]\}$ is separated from 0 by $\partial B(1;\rho)\cap\bbD$, and therefore $\varphi^n(\{e^{i\theta}:\theta\in[0,\theta_n]\})$ is separated from 0 in $\bbD$ by a curve of length less than $\e/2$. This completes the proof.
\end{proof}

We will also use the following.

\begin{proposition}\label{prop:unif-sph}
    Let $D^n,\wt D^n$ be simply connected domains containing 0 and 1, $\varphi^n:D^n\to\wt D^n$ be conformal maps fixing 0 and 1. Further assume that for any $R>0$,  $B(0;R)\subseteq D^n$ for large enough $n$. Then for each $R>0$, for large enough $n$, $B(0;R)\subseteq \wt D^n$. Moreover,  $\varphi^n$ converges uniformly to the identity function, in the sense that for any $\e>0$, there exists $N>0$ such that for every $n>N$ and $z\in D^n$, $d_\sph(\varphi^n(z),z)<\e$.
\end{proposition}

\begin{proof}
    By the Koebe distortion theorem, there exists some absolute constants $c_1,c_2>0$ such that $c_2^{-1}<|\frac{\varphi^n(1)}{(\varphi^n)'(0)}|<c_1^{-1}$ and therefore $c_1<|(\varphi^n)'(0)|<c_2.$ Let $R>0$ and suppose $n$ is large enough such that $B(0;R)\subseteq D^n$. By another application of the Koebe distortion theorem to the conformal map $R^{-1}\varphi^n(Rz)$ defined on $\bbD$, we see that for some absolute constants $c_3,c_4>0$, \eqb\label{eq:pf-prop-unif-sph} B(0;c_3R)\subseteq \varphi^n(B(0;R/2))\subseteq B(0;c_4R) \text{ for every $n$ and $R>3$ such that $B(0;R)\subseteq D^n$. }\eqe  By classical results on conformal maps, for any sequence of $\bbN$, there is a further subsequence $\{n_k\}$ such that $\varphi^{n_k}$ converges uniformly on compacts to a function $\varphi$ that is conformal (i.e., holomorphic and one-to-one) on $\bbC$. Since $\varphi$ fixes 0 and 1,  $\varphi$ must be the identity function. This proves the uniform convergence of $\varphi^n$ on compacts.  Furthermore, by the claim~\eqref{eq:pf-prop-unif-sph}, for any $\e\in(0,0.01)$, there exists $\delta\in(0,\e)$ such that if $d_\sph(z,\infty)<\delta$ then $d_\sph(\varphi^n(z),\infty)<\e$ for large enough $n$, which further implies that the convergence is uniform on $D^{n}$ as in the statement of the proposition.
\end{proof}

\subsection{Convergence of circle packing map to conformal map}\label{subsec:HSextension-pf}

In this subsection, we prove our extensions on the convergence of circle packing map to conformal map. In Section~\ref{subsubsec:HS-boundary}, we prove Theorem~\ref{thm:He-Schramm-extension-0}, Theorem~\ref{thm:He-Schramm-extension-1}. In Section~\ref{subsubsec:HS-cone}, we prove Theorem~\ref{thm:He-Schramm-sphere}. In Section~\ref{subsubsec:HS-disk}, we prove Theorem~\ref{thm:He-Schramm-disk}.

\subsubsection{An extension of the He-Schramm theorem to the boundary}\label{subsubsec:HS-boundary}

Our starting point is the  following lemma on the continuity of circle packing maps, which can be viewed as an analog of Lemma~\ref{lem:He-Schramm-unif} for circle packings. This is proved in~\cite[Equation (2.2)]{HeSchramm}. Since this lemma will also be used a number of times later, we include the proof here for completeness. 
For $z_0\in\bbC$ and $\rho<\pi$, let $C_{\rm sp}(z_0,\rho)$ be the circle $\{z:d_{\rm sp}(z,z_0)=\rho\}$. 
See Figure~\ref{fig:pf-He-Schramm-extension-3} for an illustration.
\begin{lemma}\label{lem:He-Schramm}
    Suppose $\cP$ and $\wt\cP$ are two circle embeddings in $\bbC$ for a disk triangulation $T$. Let $\delta$ be the maximal spherical diameter of the closed disks in $\cP$, and assume $\delta\leq 0.1$. For each vertex $v$, let $D_v$ be its disk in $\cP$ and let $\wt D_v$ be its disk in $\wt\cP$. Then there exists some universal constant $c_0$ such that the following is true. For each $z_0\in\wh\bbC$, $r_0\in [4\delta,0.2]$  and $r_1\in [2r_0,1]$, there exists $\rho\in [r_0,r_1]$, such that 
    \begin{equation}\label{eq:lem-He-Schramm}
      L(\rho):=\sum_{v:\ol D_v\cap C_{\mathrm{sp}}(z_0,\rho)\neq\emptyset} \mathrm{diam}_{\rm sp}(\wt D_{v})\leq c_0\log(r_1/r_0)^{-1/2},
    \end{equation}
    where $\mathrm{diam}_{\rm sp}$ stands for the spherical diameter of a disk.
\end{lemma}

\begin{figure}
    \centering
  \includegraphics[scale=0.8]{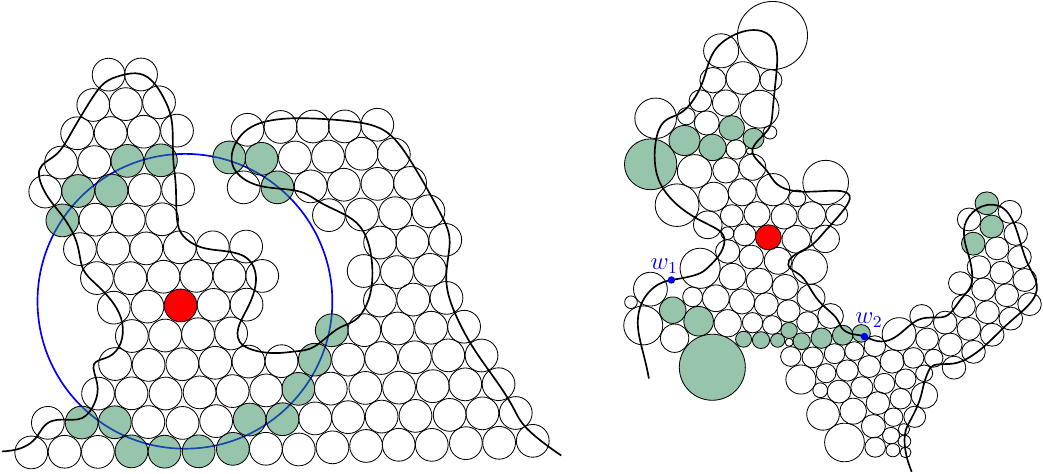}
    \caption{An illustration of the disks in Lemma~\ref{lem:He-Schramm}. The blue circle in the left panel is $C_{\sph}(z_0,\rho)$.  The disk closest to $z_0$ is colored in red, and the circles $D_v^n$ and $\wt D_v^n$ where $D_v^n$ intersects $C_{\rm sp}(z_0,\rho)$ are colored in dark green.}
    \label{fig:pf-He-Schramm-extension-3}
\end{figure}

\begin{proof}
    Define the annulus $A = \{z:r_0\leq d_{\rm sp}(z,z_0)\leq r_1\}$, and let $V^*$ be the set of vertices $v$ in $\cV\cG$ such that $\ol D_v$ intersects $A$. For each $v\in V^*$, let {$[a_v,b_v]\subset [r_0,r_1]$ be the set of $\rho\in[r_0,r_1]$ such that $\ol D_v$ intersects $C_{\rm sp}(z_0,\rho)$.} 
    Then 
    \begin{equation}
    \begin{split}
        &\inf\{L(\rho):r_0\leq \rho\leq r_1\}\log(r_1/r_0)\leq \int_{r_0}^{r_1}\frac{L(\rho)}{\rho}d\rho = \sum_{v\in V^*}\diam_{\rm sp}(\wt D_v)\int_{a_v}^{b_v}\frac{1}{\rho}d\rho\\
        &\qquad\qquad\qquad\leq  \sum_{v\in V^*}\diam_{\rm sp}(\wt D_v)\frac{b_v-a_v}{a_v}
        \leq \big( \sum_{v\in V^*}\diam_{\rm sp}(\wt D_v)^2\big)^{1/2}\big( \sum_{v\in V^*}\frac{(b_v-a_v)^2}{a_v^2}\big)^{1/2}.
        \end{split}
    \end{equation}
    Note that there exists some constant $c_1$ such that for any disk $D_0\subset \bbC\cup\{\infty\}$, $\diam_{\rm sp}(D_0)^2\leq c_1\area_{\rm sp}(D_0)$. Since $\area_{\rm sp}(\bbC\cup\{\infty\})=4\pi$ and $\wt\cP$ is a circle embedding, 
    \begin{equation}
        \sum_{v\in V^*}\diam_{\rm sp}(\wt D_v)^2\leq 4\pi c_1.
    \end{equation}
   On the other hand, for each $v\in V^*$, there exists a disk $Q_v$ inside the disk $D_v$ with spherical diameter $b_v-a_v$ and spherical distance $a_v$ from $z_0$. Moreover, $Q_v$ is contained within $A$. Consequently, 
   \begin{equation}
   \begin{split}
       &\int_{z\in A}d_{\rm sp}(z,z_0)^{-2}d\,\area_{\rm sp}\geq \sum_{v\in V^*}\int_{z\in Q_v}d_{\rm sp}(z,z_0)^{-2}d\,\area_{\rm sp}\\
       &\qquad\geq \sum_{v\in V^*}\area_{\rm sp}(Q_v)/b_v^2\geq \sum_{v\in V^*}c_1^{-1}\diam_{\rm sp}(Q_v)^{2}/b_v^2 = c_1^{-1}\sum_{v\in V^*}\frac{(b_v-a_v)^2}{b_v^2}\geq \frac{1}{4c_1}\sum_{v\in V^*}\frac{(b_v-a_v)^2}{a_v^2},
       \end{split}
   \end{equation}
  {where we used in the last step that $b_v\leq 2a_v$ since $b_v-a_v=\diam_{\rm sp}(Q_v)\leq \delta$ and $a_v\geq r_0\geq \delta$.}
   Finally, as evaluated in~\cite[Equation (2.8)]{HeSchramm},  {the integral on the left side of the last display is bounded above by $2\pi\log(r_1/r_0)$. Combining this with the indented equations above we get} 
   $$ \inf\{L(\rho):r_0\leq \rho\leq r_1\}\log(r_1/r_0)\leq 4\sqrt{2}\pi c_1\log(r_1/r_0)^{1/2}$$
   which further implies~\eqref{eq:lem-He-Schramm} is true for some $\rho\in[r_0,r_1]$.
\end{proof}

\begin{remark}\label{remark:He-Schramm}
    The original statement in~\cite[Equation (2.2)]{HeSchramm} is in the form of spherical distance as above. It is straightforward to check that following the same proof there, if a priori we know that $\cP \subseteq \bbD$ (resp.\ $\wt\cP\subseteq \bbD$), then the spherical distance in the term $C_{\mathrm{sp}}(z_0,\rho)$ (resp.\ $\mathrm{diam}_{\rm sp}(\wt D_{v})$) can be replaced by the Euclidean distance (with potentially different absolute constant $c_0$).
\end{remark}


\begin{lemma}\label{lem:pf-He-Schramm-1}
    In the setting of Theorem~\ref{thm:He-Schramm-extension-0}, for each compact set $\wt K\subseteq \wt D$, there exists a compact set $K\subseteq D$ such that for large enough $n$, $\wt K\subseteq \carrier (\wt\cP^n)$ and $(f^n)^{-1}(\wt K)\subseteq K$. In particular, for any $\wt p_0\in\wt D$, for some $N_0$, $\{(f^n)^{-1}(\wt p_0):n\geq N_0\}$ is compact in $D$ and condition (iii) holds for $(\wt D,\wt\cP^n)$. 
\end{lemma}
\begin{proof}
By translation and scaling, without loss of generality assume $D,\wt D\subseteq B(0;1/2)$. Below we write $\Omega_n =\carrier(\cP^n) $ and $\wt\Omega_n =\carrier(\wt\cP^n) $. Consider a decreasing sequence of bounded Jordan domains $U_j$ and $\wt U_j$ such that $\ol D\subseteq U_j$, $\ol{\wt D}\subseteq \wt U_j$ and their boundaries converge to $\partial D$ and $\partial\wt D$ under the curve metric~\eqref{eq:curve-metric-0}, respectively. Then conditions (i) and (ii) in Theorem~\ref{thm:He-Schramm-extension-0} implies that for fixed $j$, we have $\Omega_n\subseteq U_j$ and $\wt\Omega_n\subseteq \wt U_j$ for large enough $n$, and we choose a sequence $j(n)\to\infty$ sufficiently slow such that the above containment continues to hold for $j(n)$ in place of $j$ and the boundary disk distance as in~\eqref{thm:He-Schramm-extension-0--} continues to hold with $(U_j,\wt U_j)$ in place of $(D,\wt D)$. 

Fix $\wt K\subseteq \wt D$, and we may assume that $\wt K$ is connected and contains $f^n(p_0)$ for all sufficiently large $n$. Let $d = \dist(\wt K,\partial\wt D)$. We choose $r_1< \dist(p_0,\partial D)/100$ and $r_0<r_1/2$ sufficiently small such that $c_0\log(r_1/r_0)^{-1/2}<d/10$ where $c_0$ is the constant in Lemma~\ref{lem:He-Schramm} for the Euclidean setting. For sufficiently large $n$, the maximal diameter of disks in $\cP^n$ is less than $r_0/10$ and every point on $\partial\Omega_n$ has distance at most $r_0/2$ from $\partial U_{j(n)}$, while each boundary disk of $\wt\cP^n$ has distance at most $d/4$ from $\partial\wt U_{j(n)}$. Fix $z_0\in \Omega_n$ such that $\dist(z_0,\partial U_{j(n)})<r_0/2$. By Lemma~\ref{lem:He-Schramm}, for some $\rho\in[r_0,r_1]$, if we set $V_n(z_0,\rho) = \{v: D_v^n\cap\partial B(z_0;\rho)\neq\emptyset\}$, then $\sum_{v\in V_n(z_0,\rho)} \diam(\wt D_v^n)<d/10$. 

Let $S_n(z_0,\rho) = \cup_{v\in V_n(z_0,\rho)} D_v^n\cap \Omega_n$. Then $S_n(z_0;\rho)$ separates $z_0$ from $p_0$ in $\Omega_n$, and $f^n(S_n(z_0,\rho)) = \cup_{v\in V_n(z_0,\rho)} \wt D_v^n\cap \wt\Omega_n$ by our assumption. Since $\dist(z_0,\partial U_{j(n)})<r_0/2$, each connected component of $S_n(z_0,\rho) $ intersects $\partial \Omega_n$ and $f^n(S_n(z_0,\rho))\subseteq\{w:\dist(w,\partial\wt U_{j(n)})<d/2\}$. In particular, $f^n(S_n(z_0,\rho))$ is disjoint from $\wt K$. Let $\wt K_n^0$ be the connected component of $\wt\Omega_n\cap \wt K$ containing $f^n(p_0)$. Then $f^n(z_0)\notin \wt K_n^0$. By ranging $z_0$ over $\partial\Omega_n$, $\wt K_n^0\cap \partial \wt\Omega_n=\emptyset$, and the connectedness of $\wt K_n$ implies that $\wt K$ is in the interior of $\wt\Omega_n$. Consequently, $\wt K_n^0=\wt K$ and $ (f^n)^{-1}(\wt K)\subseteq \{z\in U_{j(n)}:\dist(z,\partial U_{j(n)})\geq r_0/2\}$. Since $U_{j(n)}$ is an outer approximation to the Jordan domain $D$,  the sets on the right hand side are eventually contained in a single compact subset of $D$. This finishes the proof.
\end{proof}

\begin{lemma}\label{lem:pf-He-Schramm-2}
    In the setting of Theorem~\ref{thm:He-Schramm-extension-0}, the maximal  diameter of disks in $\wt\cP^n$ tends to 0 as $n\to\infty$.
\end{lemma}

\begin{proof}
 Let $\e>0$. For a disk $D_v^n\in\cP^n$, we write $\wt D_v^n$ for its corresponding disk in $\wt\cP^n$.  By  the last sentence of condition (ii), for some $N_1>0$, the maximal  diameter of disks in $\wt\cP^n$ intersecting $\bbC\backslash\ol{\wt D}$ is less than $\e/4$ for $n>N_1$. Now let $\wt K$ be a compact set that contains $\{z\in\wt D:d(z,\partial\wt D)\geq \e/4\}$. By Lemma~\ref{lem:pf-He-Schramm-1}, for some $N_2>0$ and compact $K\subseteq D$, if $D_v^n\cap \mathrm{carrier}(\cP^n)$ is disjoint from $K$, then $\wt D_v^n\cap  \mathrm{carrier}(\wt\cP^n)$ is disjoint from $\wt K$. We may further pick $N_3>0$ such that $K$ is disjoint from boundary disks of $\cP^n$ for all $n>N_3$. For $n>N_1,N_2,N_3$, if $\wt D_v^n$ intersects $\bbC\backslash\ol{\wt D}$, then $\wt D_v^n$ has  diameter less than $\e$.
 If  $\wt D_v^n$ is a boundary disk disjoint from $\bbC\backslash\ol{\wt D}$, then $\wt D_v^n\subseteq \ol{\wt D}$. Then  $\wt D_v^n$ has diameter less than $\e$, since otherwise the center of $\wt D_v^n$ is in $\mathrm{carrier}(\wt\cP^n)\cap\wt K$ while $D_v^n$ is disjoint from $K$, which is a contradiction. If $D_v^n$ is disjoint from $K$ and not a boundary disk, then by our choice of $\wt K$, $\wt D_v^n$ also has  diameter less than $\e$.  Finally, for disks whose corresponding source disks intersect $K$, we apply the proof of ~\cite[Lemma 2.2]{HeSchramm} with $D,\wt D$ replaced by the outer domains $U_{j(n)}, \wt U_{j(n)}$ constructed in the proof of Lemma~\ref{lem:pf-He-Schramm-1}, where the parameters   in that proof can also be chosen uniformly in $n$. Hence the maximal  diameter of these target disks tends to zero.
\end{proof}

Below we call a bounded interstice a facial interstice, if the vertices for its three neighboring disks form a face of $T$.
\begin{proof}[Proof of Theorem~\ref{thm:He-Schramm-extension-0}]
    The claim that conditions (i) and (iii) holds for $( \wt\cP^n,\wt D)$ in place of $(\cP^n,D)$ is immediate from Lemma~\ref{lem:pf-He-Schramm-1} and Lemma~\ref{lem:pf-He-Schramm-2}.
In particular, for each compact $K\subseteq D$ and $\wt K\subseteq \wt D$,  there exists compact sets $K_1\subseteq D$ and $\wt K_1\subseteq \wt D$ such that for every large enough $n$, 
\begin{equation}\label{eq:thm:pf-He-Schramm-extension-0}
    f^n(K)\subseteq \wt K_1, \ \ \ (f^n)^{-1}(\wt K) \subseteq K_1. 
\end{equation}
Since we are working with bounded domains, the maximal Euclidean diameter of facial interstices in $\cP^n$ and $\wt\cP^n$ tends to 0 as $n\to\infty$ by our conditions and Lemma~\ref{lem:pf-He-Schramm-2}, and we may assume $f^n$ is M\"{o}bius on the facial interstices since the error by this change is tending to 0 uniformly.    Theorem~\ref{thm:He-Schramm-extension-0} then follows from the local arguments of~\cite[Lemmas 4.2-4.3, Section 5]{HeSchramm}, which uses the compact containment properties~\eqref{eq:thm:pf-He-Schramm-extension-0}, convergence of maximal spherical diameters of disks in $\cP^n$ and $\wt\cP^n$ and the canonical M\"{o}bius maps on triangular interstices, but does not require the entire packings to be contained in the fixed limiting domains.
\end{proof}

To prove Theorem~\ref{thm:He-Schramm-extension-1}, we start with the following lemma on equicontinuity and uniform convergence.
\begin{lemma}\label{lem:unif-conv-def}
   Let $((D^n)_{n\in\bbN},D,f^n,f)$ be as in Definition~\ref{def:unif-convergence}. Suppose $\{f^n\}_{n\in\bbN}$ is equicontinuous, $f$ is uniformly continuous, and   $f^n$ converges to $f$ uniformly on each compact $K\subseteq D$. Then $f^n$ converges to $f$ uniformly, in the sense of Definition~\ref{def:unif-convergence}.
\end{lemma}

\begin{proof}
{
Fix $\e>0$. Let $N_1,\delta>0$ be such that $|f^n(x)-f^n(y)|<\e/3$ for any $n>N_1$ and $x,y\in \ol{D^n}$ with $|x-y|<\delta$, and $|f(x)-f(y)|<\e/3$ for $x,y\in\ol D$ with $|x-y|<\delta$. Since $D$ is a bounded Jordan domain, we may pick $w_1,...,w_m\in D$ such that for any $z\in \ol D$, there exists $j$ such that $|w_j-z|<\delta/3$. By local uniform convergence, we can find $N_2$ such that $w_j\in D^n$, $|f^n(w_j)-f(w_j)|<\e/3$ for each $j$ and $n>N_2$. Now for each $n>N_1,N_2$, $z'\in \ol D^n$ and $z\in \ol D$ with $|z'-z|<\delta/3$, we pick $w_j\in D$ with $|w_j-z|<\delta/3$. Then $|z'-w_j|\leq |z-z'|+|z-w_j|<\delta$ and $|f^n(z')-f(z)|\leq |f^n(z')-f^n(w_j)|+|f^n(w_j)-f(w_j)|+|f(w_j)-f(z)|<\e$. This finishes the proof.}
\end{proof}

\begin{proof}[Proof of Theorem~\ref{thm:He-Schramm-extension-1}]
By translation and scaling, we may work on the case where $D,\wt D\subseteq \bbD$ and assume $\supp(\cP^n),\supp(\wt\cP^n)$ are contained with $B(0;2)$ by conditions (i) and (ii).
Therefore we may replace the spherical distance in Lemma~\ref{lem:He-Schramm} with Euclidean distance. In light of Lemma~\ref{lem:unif-conv-def} and Theorem~\ref{thm:He-Schramm-extension-0}, it suffices to prove the equicontinuity of $\{f^n\}$.

We write $\Omega_n$ and $\wt\Omega_n$ for the interior of $\mathrm{carrier}(\cP^n)$ and $\mathrm{carrier}(\wt\cP^n)$, respectively. It follows from condition (iv) that $\partial\Omega_n \to \partial D$ and $\partial\wt\Omega_n\to\partial\wt D$ under the curve metric~\eqref{eq:curve-metric-0}. We claim that for $x_n,y_n\in \ol\Omega_n$ with $|x_n-y_n|\to 0$, there are paths $K_n\subseteq \ol\Omega_n$ joining them with $\diam(K_n)\to 0$, and for $\wt q_n\in \wt\Omega_n$ in a fixed compact subset of $\wt D$ and $\wt\alpha_n$ is a crosscut of $\wt\Omega_n$ with $\diam(\wt \alpha_n)\to 0$, then the connected component of $\wt \Omega_n\backslash\wt\alpha_n$ that does not contain $\wt q_n$ has diameter tending to 0. One can see this by considering conformal maps sending $\bbD$ to $\Omega_n$ or $\wt\Omega_n$ and applying the uniform convergence result from Proposition~\ref{prop:boundary-perturb}.

Now assume the equicontinuity fails. Then possibly passing to a further subsequence, we may assume that there is some $\e_0>0$ and $x_n,y_n\in \Omega_n$, such that $x_n,y_n\to z_0\in\ol D$ such that \eqb\label{eq:pf-thm:He-Schramm-extension-1}|f^n(x_n)-f^n(y_n)|\geq \e_0.\eqe If $z_0\in D$, then~\eqref{eq:pf-thm:He-Schramm-extension-1} directly contradicts with the local equicontinuity supplied by the local uniform convergence in Theorem~\ref{thm:He-Schramm-extension-0} and therefore we may assume $z_0\in\partial D$. Let $K_n\subseteq\ol\Omega_n$ be the path connecting $x_n$ and $y_n$ with diameter tending to 0 as above, and fix $r_1\in(0,|z_0-p_0|/3)$, and choose $r_{0,n}\downarrow 0$ sufficiently slowly such that $K_n\subseteq B(z_0;r_{0,n}/2)$ and the maximal diameter of disks in $\cP^n$ and $\wt\cP^n$ are $o(r_{0,n})$. Let $L_n(\rho)$ be defined in the same way as Lemma~\ref{lem:He-Schramm} using $\cP^n,\wt\cP^n$ and the Euclidean metric in place of spherical metric. By our choice of $r_{0,n},r_1,\rho_n$, $\partial B(z_0;\rho_n)$ separates $p_0$ from $B(z_0;2r_{0,n}/3)$ in $\Omega_n$.   Then the Euclidean version of Lemma~\ref{lem:He-Schramm} implies that there exists $\rho_n\in [r_{0,n},r_1]$ such that $L_n(\rho_n)\leq C\log(r_1/r_{0,n})^{-1/2}\to 0$. One can further construct a crosscut or a simple closed curve $\sigma_n\subseteq \ol \Omega_n\cap B(z_0;2r_1)$ within the union of closed disks in $\cP^n$ intersecting $\partial B(z_0;\rho_n)$ that separates $K_n$ from $p_0$ in $\Omega_n$ such that $\diam(f^n(\sigma_n))\leq 2L_n(\rho_n)$. This can be done by taking $\sigma_n$ to be a connected component of $\partial B(z_0;\rho_n)\cap \Omega_n$ that separates  {$K_n$ from $p_0$}, and replacing all the intersection between this circle (segment) and bounded interstices with the segments of the boundaries of the disks that intersect both $\partial B(z_0;\rho_n)$ and the concerned interstice. If $\sigma_n$ is a closed curve, then $f^n(K_n)$ is enclosed by $f^n(\sigma_n)$ and therefore $\diam(f^n(K_n))\to 0$. If $\sigma_n$ is a crosscut, then using $p_0\notin B(z_0;2r_1)$, we see that $f^n(K_n)$ is separated from $f^n(p_0)$ in $\wt\Omega_n$ by $f^n(\sigma_n)$, which further combined with condition (iii) and the small crosscut claim in the previous paragraph gives $\diam(f^n(K_n))\to 0$. Both cases contradict~\eqref{eq:pf-thm:He-Schramm-extension-1}. Therefore $\{f^n\}$ is equicontinuous, and  the desired uniform convergence follows from Lemma~\ref{lem:unif-conv-def} and Theorem~\ref{thm:He-Schramm-extension-0}.

\end{proof}

\subsubsection{A version of He-Schramm theorem in whole plane topology}\label{subsubsec:HS-cone}


In Theorem~\ref{thm:He-Schramm-extension-0} and Theorem~\ref{thm:He-Schramm-extension-1}, the  domains $D,\wt D\subsetneq \bbC$. Now we use Theorem~\ref{thm:He-Schramm-extension-1} and Lemma~\ref{lem:He-Schramm} to extend to the setting where $D=\wt D=\bbC$ and prove Theorem~\ref{thm:He-Schramm-sphere}. 

\begin{lemma}\label{lem:pf-He-Schramm-sphere-1}
    In the setting of Theorem~\ref{thm:He-Schramm-sphere}, for every $R>0$, there exists $R_1>0$, such that for large enough $n$, $B(0;R)\subseteq f^n(B(0;R_1))$. In particular,  for   for large enough $n$,  $B(0;R)\subseteq \carrier(\wt\cP^n)$.
\end{lemma}

\begin{figure}
    \centering
    \includegraphics[scale=0.7]{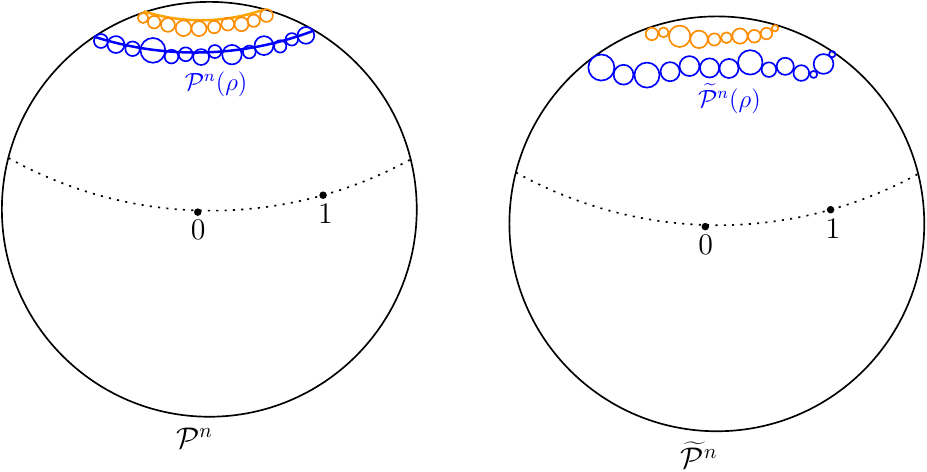}
    \caption{An illustration of the proof of Lemma~\ref{lem:pf-He-Schramm-sphere-1} using Lemma~\ref{lem:He-Schramm}. The orange disks are the boundary disks for the packing $\cP^n$ and $\wt\cP^n$, and the blue curve on the left panel represents the circle $C_{\sph}(\infty,\rho)$. We use Lemma~\ref{lem:He-Schramm} to prove that, for some appropriately chosen $\rho$, the union $\wt\cP^n(\rho)$ of the blue disks on the right panel has diameter less than $\e/10$.}
    \label{fig:packsph}
\end{figure}

\begin{proof}
    It suffices to prove that for each $\e>0$, for large enough $n$, $\{z\in\bbC:d_\sph(z,\infty)\geq\e\}\subseteq\carrier(\wt\cP^n)$.  We apply Lemma~\ref{lem:He-Schramm} for $z_0=\infty$. Let $c_0$ be the constant in Lemma~\ref{lem:He-Schramm},  $r_1 = 0.1$ and $r_0 = r_1e^{-\frac{100c_0^2}{\e^2}}$. By conditions (i) and (ii) on $\cP^n$ in Theorem~\ref{thm:He-Schramm-sphere}, there exists some $N_0$ such that for $n>N_0$, $ \{z\in\bbC:d_\sph(z,\infty)\geq r_0/4\}\subseteq\supp(\cP^n)$ and the maximal spherical diameter of disks in $\cP^n$ as well as the bounded interstices is less than $r_0/80$. Let $n>N_0$. By Lemma~\ref{lem:He-Schramm}, there exists some $\rho\in[r_0,r_1]$ such that 
    \begin{equation}\label{eq:pf-He-Schramm-sphere-1}
        \sum_{v\in\cV T^n:D_v^n\cap C_\sph(\infty,\rho)\neq\emptyset} \diam_\sph(\wt D_v^n)\leq \frac{\e}{10}. 
    \end{equation}
    See also Figure~\ref{fig:packsph} for an illustration. Let $$\cP^n(\rho) = \bigcup_{v\in\cV T^n:D_v^n\cap C_\sph(\infty,\rho)\neq\emptyset} D_v^n \ \ \text{and} \ \ \wt\cP^n(\rho) = \bigcup_{v\in\cV T^n:D_v^n\cap C_\sph(\infty,\rho)\neq\emptyset} \wt D_v^n.$$
    Then since $\rho\in [r_0,r_1]$ and  $\bbC\backslash\{z\in\bbC:d_\sph(z,\infty)<r_0/4\}\subseteq\carrier(\cP^n)$, using the diameter control on the disks and bounded interstices, $\cP^n(\rho)\neq\emptyset$, does not intersect any disk $D_v^n$ such that $v$ has graph distance at most 3 from the boundary of $T$, and separates 0 and 1 from the disks $D_v^n$ that intersects $D_v^n\cap \{z:d_\sph(z,\infty)\leq r_0/2\}\neq\emptyset$. Therefore by~\eqref{eq:pf-He-Schramm-sphere-1} and the fact that the disks in $\wt\cP^n$ are within $\bbC$, $\wt\cP^n(\rho)\subseteq\carrier (\wt\cP^n)$, and one can draw a Jordan curve $\sigma_n$ in $\wt\cP^n(\rho)$ with spherical diameter at most $\frac{\e}{10}$ that separates both 0 and 1 from infinity as well as all the disks $\wt D_v^n$ such that $D_v^n\cap \{z:d_\sph(z,\infty)\leq r_0/2\}\neq\emptyset$. Moreover, the bounded component of $\bbC\backslash\sigma_n$ contains 0 and 1 and cannot have little spherical diameter, and therefore the unbounded component of $\bbC\backslash\sigma_n$ must have spherical diameter less than $\e/4$. In particular, the union of disks $\wt D_v^n$ where $D_v^n\cap \{z\in\bbC:d_\sph(z,\infty)\geq r_0/2\}\neq\emptyset$ and their bounded interstices contains $\{z\in\bbC:d_\sph(z,\infty)\geq \e/2\}$, and none of them are boundary disks in $\wt\cP^n$. This implies that for large $n$, $\{z\in\bbC:d_\sph(z,\infty)\geq \e/2\}\subseteq f^n(\{z\in\bbC:d_\sph(z,\infty)\geq r_0/2\})$, which  is equivalent to the claim.
\end{proof}

\begin{lemma}\label{lem:pf-He-Schramm-sphere-2}
    In the setting of Theorem~\ref{thm:He-Schramm-sphere}, the maximal spherical diameter of disks as well as bounded interstices in $\wt\cP^n$ tends to 0 as $n\to\infty$.
\end{lemma}
\begin{proof}
   Let $\e>0$, and let $r_0,c_0,N_0$ be as in the proof of Lemma~\ref{lem:pf-He-Schramm-sphere-1}. Then by the second last sentence of the proof of Lemma~\ref{lem:pf-He-Schramm-sphere-1}, for $n>N_0$, if $D_v^n\in\cP^n$ is disjoint from  $\{z\in\bbC:d_\sph(z,\infty)\geq r_0/2\}$, then $\wt D_v^n$ is disjoint from $\{z\in\bbC:d_\sph(z,\infty)\geq \e/2\}$ and therefore has spherical diameter less than $\e$. Now let $r_2 = 0.01r_0e^{-\frac{100c_0^2}{\e^2}}$, and assume that $n$ is large enough such that the maximal spherical diameter of disks as well as bounded interstices in $\cP^n$ is less than $r_2/8$. Let $z_0\in\bbC$ with $d_\sph(z_0,\infty)\geq r_0/2$. By Lemma~\ref{lem:He-Schramm} applied to $(r_2,r_0/8)$ in place of $(r_0,r_1)$, we see that there exists some $\rho\in [r_2,r_0/8]$ such that 
   \begin{equation}\label{eq:lem:pf-He-Schramm-sphere-2}
       \diam_\sph(  \wt\cP^n(z_0;\rho))<\frac{\e}{10}, \ \ \text{where }\ \   \cP^n(z_0;\rho) = \bigcup_{v\in \cV T^n: D_v^n\cap C_\sph(z_0,\rho)\neq\emptyset}   D_v^n, \ \   \wt\cP^n(z_0;\rho) = \bigcup_{v\in \cV T^n: D_v^n\cap C_\sph(z_0,\rho)\neq\emptyset} \wt D_v^n.
   \end{equation}
   On the other hand, since $\rho\leq r_0/8$, $d_\sph(z_0,\infty)\geq r_0/2$, $ \{z\in\bbC:d_\sph(z,\infty)\geq r_0/4\}\subseteq\carrier(\cP^n)$ and the maximal spherical diameter of disks as well as bounded interstices in $\cP^n$ is less than $r_2/8$, we see that $\cP(z_0;\rho)$ separates the disks $D_v^n$ intersecting $\{z:d_\sph(z,z_0)\leq r_2/2\}$ from infinity, and any bounded connected component of $\bbC\backslash \cP(z_0;\rho)$ cannot contain both 0 and 1. As a consequence, 
   $\wt\cP(z_0;\rho)$ separates the corresponding disks $\wt D_v^n$ from infinity and any bounded connected component of  $\bbC\backslash \wt\cP(z_0;\rho)$ cannot contain both 0 and 1. This further implies that  the disks $\wt D_v^n$ such that $D_v^n$ intersects $\{z:d_\sph(z,z_0)\leq r_2/2\}$ has diameter less than $\e$. Since $z_0$ is arbitrary, ~\eqref{eq:lem:pf-He-Schramm-sphere-2} implies that the maximal spherical diameter of disks $\wt D_v^n$ where $D_v^n\cap \{z\in\bbC:d_\sph(z,\infty)\geq r_0/2\}\neq\emptyset$ is less than $\e$. This finishes the proof on spherical diameter of disks. To control the diameter of bounded interstices,  let $\wt\delta^n$ be the maximal spherical diameter of $\wt\cP^n$. Then for any facial interstice $\wt J^n$ of $\wt\cP^n$, we have $\diam_{\sph}(\partial \wt J^n)\leq 3\wt\delta^n \to 0$, and therefore one component of $\wh\bbC\backslash\partial\wt J^n$ has spherical diameter less than $6\wt\delta^n$.  If $\wt J^n$ is not this component, then $\partial \wt J^n\subseteq \{z:d_\sph(z,\infty)\leq 12\wt\delta^n\}$ since $\wt J^n$ is bounded. Then for large $n$, $\wt J^n$ contains both 0 and 1. This contradicts the fact that there cannot be a bounded interstice containing both 0 and 1 in $\cP^n$ by our assumption. Therefore we must have $\diam_\sph(\wt J^n)\leq 6\wt\delta^n$. If $\wt J^n$ is not a facial interstice, then $\wt J^n$ is outside the carrier of $\wt\cP^n$ and stays within a small spherical neighborhood of $\infty$ by Lemma~\ref{lem:pf-He-Schramm-sphere-1}.  This finishes the proof.
\end{proof}

\begin{proof}[Proof of Theorem~\ref{thm:He-Schramm-sphere}]
    By a symmetric argument, Lemma~\ref{lem:pf-He-Schramm-sphere-1} holds if one swaps $\cP^n$ and $\wt\cP^n$ and works with $(f^n)^{-1}$. In particular, both $f^n$ and $(f^n)^{-1}$ are uniformly bounded on compacts.  Furthermore, by the same argument in the proof of Lemma~\ref{lem:pf-He-Schramm-sphere-2}, using Lemma~\ref{lem:He-Schramm}, one can prove that for some constant $c_1,c_2>0$, for each $R>2$   and $r\in (0,1/2)$,  for every large $n$, for every $z_0\in B(0;R)$,
    $$B(f^n(z_0);e^{-c_1r^{-2}})\subseteq f^n(B(z_0;r))\subseteq  B(f^n(z_0);c_2(-\log r)^{-1/2}),$$
    where the first inclusion follows by the local separating-cycle argument in Lemma~\ref{lem:pf-He-Schramm-sphere-2} applied in the reverse direction.
    This implies that both $f^n$ and $(f^n)^{-1}$ are equicontinuous on compacts, and for any subsequence of $\bbN$, there is a further subsequence $\{k_n\}$ such that $f^{k_n}|_{B(0;R)}$ converges uniformly to some homeomorphism $f_R:B(0;R)\to D_R$. The same local argument as in~\cite[Lemmas 4.2-4.3, Section 5]{HeSchramm} implies that the map $f_R$ is conformal.  By further taking diagonal subsequences, we see that for any subsequence of $\bbN$, there is a further subsequence $\{k_n\}$ such that   $f^{k_n}$ converges some function $f$ uniformly on compacts and $f$ is conformal on $\bbC$. Since $f^n(0)=0$ and $f^n(1)=1$, $f$ must be the identity. Therefore for each subsequence of $\bbN$, there exists a further subsequence $\{n_k\}$ converging uniformly to the identity on compacts, which further implies that $f^n$ converges to the identity uniformly on compacts.   Furthermore, by Lemma~\ref{lem:pf-He-Schramm-sphere-1}, for any $\e>0$, there exists some $\delta\in (0,\e/3)$ and $N_0>0$ such that $d_\sph(f^n(z),\infty)<\e/3$ for every $z$ with $d_\sph(z,\infty)<\delta$ and $n>N_0$, which implies that $d_\sph(f^n(z),z)<\e$ for each $n>N_0$ and $z$ with $d_\sph(z,\infty)<\delta$.   This further gives the uniform convergence as in the theorem statement and therefore finishes the proof.
\end{proof}

\subsubsection{A comparsion of circle packing for submaps of disk triangulations}\label{subsubsec:HS-disk}

We are now ready to prove Theorem~\ref{thm:He-Schramm-disk}. Recall the setup in Figure~\ref{fig:pack-disk-1}.

\begin{proof}[Proof of Theorem~\ref{thm:He-Schramm-disk}]
We rotate $\cP^n$ and $\wt\cP^n$ such that the disks for $v_1^n$ is tangent to $\partial\bbD$ at $-i$ in both packings. 
  We first prove that in the circle packing $\wt\cP^n$, the diameter of the union of the disks for vertices on $\gamma^n\cup (\cV\wt T^n\backslash\cV T^n)$ tends to 0 as $n\to\infty$.    Without loss of generality assume that in $\wt\cP^n$, the disk for some vertex on $\gamma^n$ is tangent to $\partial\bbD$ at $-i$. Let $\e>0$ be small. Since we are working on circle packings in $\bbD$, we work with the version of Lemma~\ref{lem:He-Schramm} with Euclidean metric in place of spherical metric, and $c_0>0$ be the corresponding constant. Let $\delta = 0.01e^{-\frac{100c_0^2}{\e^2}},$ and assume that $n$ is large enough such that the quantities in (i) and (ii) are less than $\delta/8$. 
 Now we apply Lemma~\ref{lem:He-Schramm} with $r_0=\delta$, $r_1 = 0.01$ and $z_0 = -i$. For $v\in\cV T^n$, we write $D_v^n$ and $\wt D_v^n$ for the disk for $v$ in $\cP^n$ and $\wt \cP^n$, respectively.  Then there exists some $\rho\in[r_0,r_1]$, such that 
 \begin{equation}\label{eq:thm:He-Schramm-disk}
     \diam(\wt\cP^n(\rho))\leq\frac{\e}{10},\ \ \text{where } \cP^n(\rho) = \bigcup_{v\in T^n:D_v^n\cap\partial B(-i;\rho)\neq\emptyset}D_v^n \ \text{ and } \wt \cP^n(\rho) = \bigcup_{v\in T^n:D_v^n\cap\partial B(-i;\rho)\neq\emptyset}\wt D_v^n.
 \end{equation}
 Since $\rho\in[\delta,0.01]$, the  maximal diameter of disks in $\cP^n$ and the  diameter of the union of the disks in $\cP^n$ for vertices in $\gamma^n$ are less than $\delta/8$, we see that $\cP^n(\rho)$ is disjoint from the disks in $\cP^n$ for vertices on $\gamma^n$ so the first and last disk are for vertices on $\ell^n$. As a consequence, $\cP^n(\rho)$ separates $\cup_{v\in\gamma^n} D_v^n$ from $D_{v^n}^n$, and therefore $\wt\cP^n(\rho)$ separates $\cup_{v\in\gamma^n} \wt D_v^n$ from $\wt D_{v^n}^n$. This along with~\eqref{eq:thm:He-Schramm-disk} imply that the union of the disks $\wt D_v^n$ where $v\in \gamma^n\cup( \cV\wt T^n\backslash\cV T^n)$ can be separated from 0 in $\bbD$ by a crosscut of diameter less than $\e/4$ in $\bbD$ and therefore
 \eqb\label{eq:thm:He-Schramm-disk-1}\diam\big( \bigcup_{v\in \gamma^n\cup( \cV\wt T^n\backslash\cV T^n)} \wt D_v^n \big)<\e,\eqe
 which verifies the claim. Then by~\eqref{eq:thm:He-Schramm-disk-1} and Theorem~\ref{thm:He-Schramm-extension-0} applied to $(\cP^n,\bbD)$ and $(\cP_1^n,\bbD)$ (where we choose some arbitrary circle packing isomorphism $g^n$ with $g^n(0)=0$), the maximal diameter of disks in $\wt\cP^n$ tends to 0 as $n\to\infty$.  Further combined with~\eqref{eq:thm:He-Schramm-disk-1}, the distance~\eqref{eq:curve-metric-0} between the loop formed by joining centers of disks for vertices on $\partial T^n$ with line segments and $\partial\bbD$ tends to 0 as $n\to\infty$, which implies that the condition (iv) in Theorem~\ref{thm:He-Schramm-extension-1} holds for $\cP^n_1$.  By Theorem~\ref{thm:He-Schramm-extension-1}, for any subsequence of $\bbN$, there exists a further subsequence $\{n_k\}$ such that $f^{n_k}$ converges uniformly to a conformal map $f: \bbD\to\bbD$ in the sense of Definition~\ref{def:unif-convergence}, and by our choice of $\cP^n$ and $\wt\cP^n$, $f$ must be the identity function. This finishes the proof.
\end{proof}

\section{Convergence of the mated-CRT map under conformal embeddings}\label{sec:proof}

In this section, we prove the convergence of the mated-CRT map to LQG under circle packing and the Riemann uniformization embedding, namely Theorem~\ref{thm:pack-cone}-\ref{thm:unif-sph}. Throughout this section, we will assume $\gamma\in(0,2)$ and $\kappa=16/\gamma^2$. We will first define a cell configuration $\cH^\rmE$ for the space-filling $\SLE_\kappa$ on 0-quantum cone as in~\cite{GMS21tutte}, and the version $\wh\cH^\rmE$ with have multiple edges in the associated map collapsed, and verify that $\wh\cH^\rmE$ and $\cH^\rmE$ fall into the framework of Theorem~\ref{thm:CellSystemPack} and Theorem~\ref{thm:invariance-uniformization}, respectively. This is done in Section~\ref{subsec:mated-CRT-0-cone}, where we prove a version of Theorem~\ref{thm:pack-cone} and Theorem~\ref{thm:unif-cone} in the setting of 0-quantum cone in place of $\gamma$-quantum cone. Next in Section~\ref{subsubsec:matedCRT-cell}  we use Theorem~\ref{thm:He-Schramm-extension-1} along with the scale invariance of the quantum cone and local absolute continuity to prove a version of Theorem~\ref{thm:pack-cone} for the mated-CRT map on the quantum cell $\eta([-1,1])$ (i.e., the restriction of the map $\cG^{1/n}$ to the vertices in $[-1,1]$), namely Proposition~\ref{thm:circle-quantum-cell-0}. In Section~\ref{subsubsec:matedCRT-disk}-\ref{subsubsec:matedCRT-cone} we further use absolute continuity along with the results on circle packing maps developed in Theorem~\ref{thm:He-Schramm-sphere} and Theorem~\ref{thm:He-Schramm-disk} to complete the proof of convergence of the mated-CRT map under circle packing (Theorem~\ref{thm:pack-cone}-\ref{thm:pack-sph}). The proof of convergence of the mated-CRT map under the Riemann uniformization embedding (Theorem~\ref{thm:unif-cone}-\ref{thm:unif-sph}) is carried out in Section~\ref{subsec:unif-matedCRT-gamma-cone}.

\subsection{Convergence of the mated-CRT map on 0-quantum cone}\label{subsec:mated-CRT-0-cone}

In this subsection, we prove a version of Theorem~\ref{thm:pack-cone} and Theorem~\ref{thm:unif-cone} in the setting of 0-quantum cone in place of $\gamma$-quantum cone based on Theorem~\ref{thm:CellSystemPack} and Theorem~\ref{thm:invariance-uniformization}, namely Theorem~\ref{thm:circle-matedCRT-0-cone} and Theorem~\ref{thm:unif-matedCRT-0-cone}. In Section~\ref{subsubsec:cell-config}, we define the cell configuration $\cH^\rmE$ as in~\cite{GMS21tutte}, and a version  $\wh\cH^\rmE$ of $\cH^\rmE$ whose associated map is simple. In Section~\ref{subsubsec:cell-config-moment}, we verify the moment bound~\eqref{eq:thm-inv-circle} for $\cH^\rmE$ and  $\wh\cH^\rmE$,  {where we will additionally use a comparison between 0-quantum cone and $\gamma$-quantum cone in~\cite[Lemma 4.8]{Tutte} to prove the moment bound for the degree.} In Section~\ref{subsubsec:cell-config-pf}, we verify the   conditions in Definition~\ref{def:CellConfigConstraint} for $\cH^\rmE$ and $\wh\cH^\rmE$, and use Theorem~\ref{thm:CellSystemPack} and Theorem~\ref{thm:invariance-uniformization} along with scale invariance of the quantum cones in Lemma~\ref{lem:Rb} to prove Theorem~\ref{thm:circle-matedCRT-0-cone} and Theorem~\ref{thm:unif-matedCRT-0-cone}.

\subsubsection{Construction of the cell configuration}\label{subsubsec:cell-config}

We recap the construction of the cell configuration as in~\cite[Section 3.1]{GMS21tutte}. Recall the probability measure $\bbP_\rmE$ and $\bbP_\rmQ$ from Definition~\ref{def:P_EP_Q}. Let $(h,\eta,\fD)$ be sampled from $\bbP_\rmE$ and $\theta$ be the uniform random variable generating $\fD$ as in Definition~\ref{def:1ddyadic}.   Let 
$$\cH^{\rmE,1/n} = \{\eta([\frac{x-1}{n},\frac{x}{n}]):x\in\bbZ+\theta\},$$
and for $H,H'\in\cH^{\rmE,1/n }$, we declare that $H\sim H'$ if and only if $H$ and $H'$ intersect along a non-trivial boundary arc. If $H$ and $H'$ intersect along both their left and
right boundaries, then we declare that there are two edges connecting $H$ and $H'$. Let $\cG^{\rmE,1/n }$ be the graph obtained in this way. Let $Z = (L,R)$ be the boundary length process defined in terms of $(h,\eta)$ as in Section~\ref{subsubsec:MatedCRT}. Then following the description of $Z$ in terms of the change in $\gamma$-LQG boundary length for $\eta((-\infty,t])$, the graph $\cG^{\rmE,1/n }$ is isometric to the mated-CRT map defined for the process $Z$ via~\eqref{eq:def-mated-CRT}, and therefore is an infinite plane triangulation. We further set the conductance of each edge to be equal to 1. This defines the cell configuration $\cH^{\rmE,1/n }$ with the associated map $\cG^{\rmE,1/n }$ being an infinite plane triangulation.

To apply Theorem~\ref{thm:CellSystemPack} and work with circle packing, we will need a version of $\cH^{\rmE,1/n }$ whose associated map is a simple infinite plane triangulation. Let $\Phi_{\rm simple}$ be the operator defined right before Definition~\ref{def:pack-non-simple} and we let $\wh\cG^{\rmE,1/n } = \Phi_{\rm simple}(\cG^{\rmE,1/n })$. Now suppose for some $x_1<x_2-1$ and $x_1,x_2\in\theta+\bbZ$, $\eta([\frac{x_1-1}{n},\frac{x_1}{n}])$ and $\eta([\frac{x_2-1}n,\frac{x_2}{n}])$ intersect both along their left and right boundaries, we merge $\eta([\frac{x_1}{n},\frac{x_2-1}{n}])$ into $\eta([\frac{x_2-1}{n},\frac{x_2}{n}])$. On the level of the associated map, this is same as removing all the vertices and edges enclosed by double edges of $\cH^{\rmE,1/n }$ and further turning the double edges into single edges as in the definition of $\Phi_{\rm simple}$. Let $\wh\cH^{\rmE,1/n }$ be the cell configuration obtained in this way whose associated map is $\wh\cG^{\rmE,1/n }$, where if $\wh\cG^{\rmE,1/n }=\emptyset$, we set $\wh\cH^{\rmE,1/n }$ to be a single cell consisted of $\bbC$. We also extend the definition of $\Phi_{\rm simple}$ to cell configurations and denote $\wh\cH^{\rmE,1/n } = \Phi_{\rm simple}(\cH^{\rmE,1/n})$. In Proposition~\ref{prop:cell-diam2/area} below, we will show that the latter event a.s.\ does not happen. We let $\cH^{\rmQ,1/n }$ and $\wh\cH^{\rmQ,1/n }$ be defined in the identical way using samples $(h,\eta,\fD)$ from the measure $\bbP_\rmQ$. Note that using the translation invariance of Brownian motion, $\cG^{\rmQ,1/n }$ has the same law as the mated-CRT map $\cG^{1/n}$. We write $\cH^\rmE,\wh\cH^\rmE,\cH^\rmQ,\wh\cH^\rmQ$ for $\cH^{\rmE,1},\wh\cH^{\rmE,1},\cH^{\rmQ,1},\wh\cH^{\rmQ,1}$, respectively. Observe that by Lemma~\ref{lem:Rb}, for the constant $\fR_n$ defined in~\eqref{eq:Rb} with $b=n$, $\cH^{\rmE,1/n }$ has the same distribution as $\fR_n^{-1}\cH^\rmE$.

We will also constantly use the following observation when working with double edges.
\begin{lemma}\label{lem:double-edge-topology}
  Let $(h,\eta)$ be sampled from $\bbP_\rmE$ or $\bbP_\rmQ$ with boundary length process $(L,R)$.  Let $t_1<t_2<t_3<t_4$. Suppose 
    \begin{equation}\label{eq:def-mated-CRT-AA}
    \max\Big\{\inf_{t\in[t_1,t_2]}L_t,  \inf_{t\in[t_3,t_4]}L_t\Big\}\leq \inf_{t\in[t_2,t_3]}L_t \ \ \text{and} \ \ \max\Big\{\inf_{t\in[t_1,t_2]}R_t,  \inf_{t\in[t_3,t_4]}R_t\Big\}\leq \inf_{t\in[t_2,t_3]}R_t.
\end{equation}
Then $\eta([t_2,t_3])$ is disjoint from the unbounded component of $\bbC\backslash \eta([t_1,t_2]\cup [t_3,t_4])$ and $\diam(\eta([t_1,t_4]))\leq \diam(\eta([t_1,t_2]))+\diam(\eta([t_3,t_4]))$.
\end{lemma}

\begin{proof}
    Let $s_\rmL$ be the last time in $[t_2,t_3]$ when $L_s = \inf_{[t_2,t_3]}L_t$. 
    Then the condition $\inf_{t\in[t_1,t_2]}L_t\leq \inf_{t\in[t_2,t_3]}L_t$ implies that the left boundary arc of $\eta([t_2,t_3])$ from $\eta(t_2)$ to $\eta(s_\rmL)$ is contained within the left boundary arc of $\eta([t_1,t_2])$, while the condition $\inf_{t\in[t_3,t_4]}L_t\leq \inf_{t\in[t_2,t_3]}L_t$ implies that the left boundary arc of  $\eta([t_2,t_3])$ from $\eta(s_\rmL)$ to $\eta(t_3)$ is contained within the left boundary arc of $\eta([t_3,t_4])$. By a symmetric argument, the right boundary arc of $\eta([t_2,t_3])$ is also contained within the union of the right boundary arcs of $\eta([t_1,t_2])$ and $\eta([t_3,t_4])$. This gives the first claim. Furthermore, the condition~\eqref{eq:def-mated-CRT-AA} implies that $\eta([t_1,t_2])\cap \eta([t_3,t_4])\neq\emptyset$ and therefore $\diam(\eta([t_1,t_4]) = \diam(\eta([t_1,t_2]\cup [t_3,t_4]))\leq \diam(\eta([t_1,t_2]))+\diam(\eta([t_3,t_4]))$.
\end{proof}

\subsubsection{Cell configuration moment bounds}\label{subsubsec:cell-config-moment}

We start with checking the moment bound~\eqref{eq:thm-inv-circle} for $\cH^\rmE$ and $\wh\cH^\rmE$. Much of the proof idea is similar to~\cite[Theorem 4.1]{GMS21tutte}, except for two differences. First, ~\cite[Theorem 4.1]{GMS21tutte} is mostly on the cell $H_0$ in $\cH^\rmE_{ }$ containing 0, and we need to take care of the collapsing of the double edges in the context of circle packings. Second, ~\cite[Theorem 4.1]{GMS21tutte} only showed moment bounds for $\deg(H_0)$ up to $4/\gamma^2$-th moment, while we need at least fourth moment bound in~\eqref{eq:thm-inv-circle}. To circumvent this,  {we are going to use a comparison lemma between 0-quantum cone and $\gamma$-quantum cone developed in ~\cite[Lemma 4.8]{Tutte}, as the degree moment bound in the latter setting can be derived via Brownian motion estimates.}

In~\cite[Theorem 4.1]{GMS21tutte}, it has been proved that for the cell $H_0$ in $\cH^\rmE$ containing 0, for any $p>0$,   \begin{equation}\label{eq:cell-diam2/area-E}
        \bbE_\rmE\big[\Big(\frac{\diam(  H_0)^2}{\area(  H_0)}\Big)^p\big]<\infty.
    \end{equation}   
    We start with the following analog for $\wh\cH^\rmE$.
\begin{proposition}
\label{prop:cell-diam2/area}
     Let $\wh H_0$ be the cell in $\wh\cH^{\rmE}$ containing 0. Then under the measure $\bbP_\rmE$, we have the following:
     \begin{enumerate}[(i)]
         \item  
     there exists some constant $c>0$, such that with probability at least $1-c^{-1}2^{-cm}$, for each $\wh\theta\in[0,1]$ and $x_1,x_2\in \bbZ+\wh\theta$ such that $x_1\leq\wh\theta$, $x_2\geq\wh\theta$ and both inequalities in~\eqref{eq:def-mated-CRT} for $n=1$ holds, $x_1,x_2\in [-2^m,2^m]$;
     \item $\wh\cG^\rmE$ is a.s.\ not empty, and almost surely, all the cells in $\wh\cH^\rmE$ are compact;
     \item  for any $p>1$, 
    \begin{equation}\label{eq:cell-diam2/area}
        \bbE_\rmE\big[\Big(\frac{\diam(\wh H_0)^2}{\area(\wh H_0)}\Big)^p\big]<\infty.
    \end{equation}    
     \end{enumerate}
    The same holds for $\bbP_\rmQ$ instead of $\bbP_\rmE$, and also $\bbP_\rmQ$ conditioned on $\theta=0$.
\end{proposition}

\begin{proof}
We start with  the measure $\bbP_\rmE$. For $b>0$,   let $\fR_b,h^b,\eta^b$ be as in~\eqref{eq:Rb} and~\eqref{eq:hb-etab}. Let $q\in(0,2(2+\gamma)^{-2})$. Let $A_n$ be the event where any cell of the form $\eta([x,x+1])$ intersecting $B(0;{2\fR_{2^n}/3})$ has diameter at most $2^{-nq}\fR_{2^n}$.  By Lemma~\ref{lem:diam-cell-max} applied to $(h^b,\eta^b)$ for $b=2^{n}$,  for some $\alpha>0$, \eqb\label{eq:pf-cell-diam2/area-1}\bbP_\rmE[A_n^c]\lesssim 2^{-\alpha n},\eqe and therefore $1-\bbP_\rmE\big[\cap_{k=n}^\infty A_k\big]\lesssim 2^{-\alpha n}$.

Define the event $B_n = \{2^{-nq}\mathfrak{R}_{2^n}< \frac{1}{4} \fR_{2^{n-1}}\}$. Lemma~\ref{lem:Rb} and~\eqref{eq:pf-cell-diam2/area-1} imply that for some constant $C_1,c_1>0$, $\bbP_\rmE[A_n^c]+\bbP_\rmE[B_n^c] \leq C_12^{-c_1n}$, and by a union bound, if we set $F_n = \cap_{k=n}^\infty (A_k\cap B_k)$, we have \eqb\label{eq:pf-AAA-nnew}\bbP_\rmE[F_n^c]\lesssim 2^{-c_1n}.\eqe Fix $\wh\theta\in [0,1]$, and let
$\mathrm{Double}_{\wh\theta}$ be the union of $\{(\wh\theta,\wh\theta)\}$ and the collection of all pairs $(a,b)$ such that $a\leq\wh\theta \leq b$, $a,b\in \bbZ+\wh\theta$, $b-a\neq 1$ and both inequalities in~\eqref{eq:def-mated-CRT} hold for $(a,b)=(x_1,x_2)$ and $n=1$. In other words, in the mated-CRT map $\cG^\rmE$ or $\cG^\rmQ$ with vertex set $\wh\theta+\bbZ$ defined via $Z$ using~\eqref{eq:def-mated-CRT}, for each pair of vertices $(a,b)\in \mathrm{Double}_{\wh\theta}\backslash\{(\wh\theta,\wh\theta)\}$, there are two edges connecting them, and all these double edges either have $\wh\theta$ as an endpoint or encloses $\wh\theta$. Let $$\mathcal{K}_{\wh\theta} = \eta([\wh\theta-1,\wh\theta]) \cup \bigcup_{(a,b)\in\mathrm{Double}_{\wh\theta}}\eta([a-1,b]).$$
We claim that for all large enough $n$, under the event $F_n$,  \eqb\label{eq:pf-AAAnew} \wh H_0 \subseteq \cup_{\wh\theta\in[0,1]} \mathcal{K}_{\wh\theta}\subseteq B(0;\fR_{2^n}/2).\eqe
Since $\theta\in[0,1]$, it is clear that under the event $A_n$ we have $\cup_{\wh\theta\in[0,1]}\eta([\wh\theta-1,\wh\theta])\subseteq B(0;\fR_{2^n}/2)$. Now assume on the contrary that 
~\eqref{eq:pf-AAAnew} fails. Then for some $\wh\theta$ and $(a,b)\in \mathrm{Double}_{\wh\theta}$, $\eta([a-1,b])\nsubseteq B(0;\fR_{2^n}/2)$. Let $N$ be the smallest $j>n$ such that  $\eta([a-1,b])\subseteq B(0;\fR_{2^j}/2)$. Then by minimality and $a-1\leq 0\leq b$, $\diam(\eta([a-1,b]))\geq  \fR_{2^{N-1}}/2$ and therefore by Lemma~\ref{lem:double-edge-topology}, $\max\{\diam(\eta[a-1,a]),\diam(\eta[b-1,b])\}\geq \fR_{2^{N-1}}/4>2^{-Nq}\mathfrak{R}_{2^N}$ by the event $B_N$. On the other hand, by the event $A_N$,  $\max\{\diam(\eta[a-1,a]),\diam(\eta[b-1,b])\}\leq  2^{-Nq}\mathfrak{R}_{2^N}$. This is a contradiction. Therefore~\eqref{eq:pf-AAAnew} holds. By~\eqref{eq:Rb-gammaQ} and Lemma~\ref{lem:SLE-contain-1}, for some $c'>0$, we have $\bbP_\rmE\big[B(0;{\fR_{2^n}/2})\nsubseteq \eta([-2^{2n},2^{2n}])\big]\lesssim 2^{-c'n}$. This together with~\eqref{eq:pf-AAA-nnew} and~\eqref{eq:pf-AAAnew} gives Assertion (i) and the claim $\wh\cG^\rmE\neq\emptyset$.

Now for any other vertex $x\in \bbZ+\theta$ of $\cG^\rmE$, consider a finite path in $\cG^\rmE$ from $\theta$ to $x$. Then by planarity, any double edge in $\cG^\rmE$ that encloses $x$ but not $\theta$ must be incident to a vertex on this path, and therefore combined with Assertion (i) there are finitely many double edges that encloses $x$. This implies compactness of cells in $\wh\cH^\rmE$.

To prove Assertion (iii), as shown in~\cite[Lemma 4.3]{GMS21tutte}, there exists constants $c_4>0$ such that 
\begin{equation}\label{eq:pf-double-edge-0-cone-42}
    \bbP_\rmE\big[\wh H_0\subseteq B(0;{\fR_{2^{-n}}}) \big] \leq \bbP_\rmE\big[H_0\subseteq B(0;{\fR_{2^{-n}}}) \big]\leq \bbP_\rmE\big[\mu_h\big(B(0;{\fR_{2^{-n}}})\big)\geq 1 \big] \lesssim  2^{-c_4n},
\end{equation}
where the last inequality comes from the scaling in Lemma~\ref{lem:Rb} that $  \bbP_\rmE\big[\mu_h\big(B(0;{\fR_{2^{-n}}})\big)\geq 1 \big] = \bbP_\rmE[\mu_{h-\frac{n}{\gamma}\log 2}(\bbD)\geq 1] = \bbP_\rmE[\mu_h(\bbD)\geq 2^n]$ and~\cite[Theorem 2.11]{RV144} where $\bbP_\rmE[\mu_h(\bbD)^\beta]<\infty$ for $\beta\in (0,\frac{4}{\gamma^2})$.

By our construction, $\wh H_0$ has the  form $\eta([x,y])$ for some $x\leq 0\leq y$.  Now by~\cite[Lemma 3.6]{GHM20KPZ} along with the reversibility of $\eta$, if we let $\tau_k^+$ (resp.\ $\tau_k^-$) be the first time after 0 (resp.\ last time before 0) when $\eta$ hits $\partial B(0;2^k)$, then for some constant $c_5>0$ and any $k\in\bbZ,C>0$,
\begin{equation}
    \bbP_\rmE\big[\area(\eta([\tau_k^-,0]))\geq 2^{2k}C^{-1}, \area(\eta([0,\tau_k^+]))\geq 2^{2k}C^{-1} |h\big] \geq 1-Ce^{-c_5C^{1/2}}.
\end{equation}
By a union bound, conditioned on $h$, for any $n$, with probability at least $1-C(2+\log_2(\fR_{2^n}/\fR_{2^{-n}}))e^{-c_5C^{1/2}}$,  $\area(\eta([\tau_k^-,0]))\geq 2^{2k}C^{-1}$ and $ \area(\eta([0,\tau_k^+]))\geq 2^{2k}C^{-1}$ for any $k\in [\log_2 \fR_{2^{-n}}-1, \log_2 \fR_{2^n} +1]\cap\bbZ$.  In particular, under this event, if $\wh H_0\not\subseteq B(0;{\fR_{2^{-n}}})$ and $\wh H_0\subseteq  B(0;{\fR_{2^n}})$, then one can check that $$\frac{\diam(\wh H_0)^2}{\area(\wh H_0)}\leq 100C.$$
Therefore together with~\eqref{eq:pf-AAA-nnew},~\eqref{eq:pf-AAAnew} and~\eqref{eq:pf-double-edge-0-cone-42}, for some constant $c_6>0$ and all $C$ large enough, 
\begin{equation}\label{eq:pf-double-edge-0-cone-43}
    \bbP_\rmE\big[\frac{\diam(\wh H_0)^2}{\area(\wh H_0)}\geq 100C\big] \leq \bbP\big[ \log_2(\fR_{2^n}/\fR_{2^{-n}})\geq e^{C^{1/4}} \big] + Ce^{-C^{1/4}} + c_6^{-1}2^{-c_6n}.
\end{equation}
Finally, if we choose $n$ such that $c_6^{-1}2^{c_6n}<C^{p+2}<c_6^{-1}2^{c_6(n+1)}$, by~\eqref{eq:Rb-bound} and~\eqref{eq:pf-double-edge-0-cone-43}, we conclude that for some constant $c>0$ and all large enough $C$,
\begin{equation}
    \bbP_\rmE\big[\frac{\diam(\wh H_0)^2}{\area(\wh H_0)}\geq 100C\big] \leq c(e^{-C^{1/5}} + C^{-p-2})\leq 2cC^{-p-2}
\end{equation}
and therefore $ \bbE_\rmE\big[\Big(\frac{\diam(\wh H_0)^2}{\area(\wh H_0)}\Big)^p\big]<\infty.$ The proof for $\bbP_\rmQ$ is identical, where one uses the version of Lemma~\ref{lem:Rb} with $\alpha=\gamma$ and  {the fact that $\bbE_\rmQ[\mu_h(\bbD)^\beta]<\infty$ for $\beta\in (0,1\wedge(\frac{4}{\gamma^2}-1))$ from ~\cite[Lemma A.4]{Tutte}} to replace~\eqref{eq:pf-double-edge-0-cone-42}.
\end{proof}

One consequence is the following analog of Lemma~\ref{lem:diam-cell-max} for $\wh\cH^{\rmE,1/n }$ and $\wh\cH^{\rmQ,1/n }$.
\begin{lemma}\label{lem:diam-cell-max-1}
  Let $q\in(0,2(2+\gamma)^{-2})$.  There exists some constant $c,\alpha>0$ such that 
    \eqb\label{eq:lem:diam-cell-max-1}\bbP_\rmE\big[\diam(\wh H)\leq 2^qn^{-q}, \ \ \text{for all }\,\wh H\in\wh\cH^{\rmE,1/n }_{}(B(0;2/3))\big]\geq 1-cn^{-\alpha}.\eqe
Furthermore, for each compact $K\subseteq\bbC$, the maximal diameter of cells in  $\cH^{\rmE,1/n }(K)$ and $\wh\cH^{\rmE,1/n }(K)$ tends to 0 in probability as $n\to\infty$.   The same holds for $\bbP_\rmQ$ instead of $\bbP_\rmE$.
    \end{lemma}
\begin{proof}
    By Lemma~\ref{lem:diam-cell-max}, \eqb\label{eq:pf-lem:diam-cell-max-1}\bbP_\rmE\big[\diam(H)\leq n^{-q}/4, \ \text{ for all }\,H\in \cH^{\rmE,1/n }_{}( B(0;3/4))\big]\geq 1-C_1n^{-\alpha_1}\eqe for some $C_1,\alpha_1>0$. For each $\wh H\in \wh\cH^{\rmE,1/n }_{\ }(B(0;2/3) )$,  we claim that if $B(0;2/3) \nsubseteq \wh H$, then there   exists   $H\in \cH^{\rmE,1/n }( B(0;3/4))$ such that $\diam(H)\geq \min\{\diam(\wh H)/2, 0.01\}$. We can express $\wh H = \eta([a,b])$. By Lemma~\ref{lem:double-edge-topology}, the outer boundary of $\wh H$ is contained within $\eta([a-\frac1n,a])\cup\eta([b-\frac{1}{n},b])$. Consider  the cell $\eta([a-\frac1n,a])$ or $\eta([b-\frac{1}{n},b])$ with larger diameter. If this cell intersects $B(0;3/4)$, $\diam(\wh H)$ is controlled by two times the diameter of this cell and  the claim holds . Otherwise the other cell must contain a path crossing $B(0;3/4)\backslash B(0;2/3)$ and the claim is also true.
    By the scale invariance in Lemma~\ref{lem:Rb},~\eqref{eq:pf-AAA-nnew} and~\eqref{eq:pf-AAAnew}, the   event that $B(0;2/3)$ is contained within a single cell in $\wh\cH$ has probability at most $C_2n^{-\alpha_2}$ for some $C_2,\alpha_2>0$. Then~\eqref{eq:lem:diam-cell-max-1} follows from~\eqref{eq:pf-lem:diam-cell-max-1} and our claim. Now we scale by $\fR_j^{-1}$ as in Lemma~\ref{lem:Rb}, and $\fR_j^{-1}\wh\cH^{\rmE,1/n }_{}$ has the same law as $\wh\cH^{\rmE,1/n j}_{ }$. The second claim then follows by applying~\eqref{eq:lem:diam-cell-max-1} and sending $j\to\infty$ sufficiently slower than $n\to\infty$.  The proof for $\bbP_\rmQ$ is similar.
\end{proof}

Next we work with the moment bound for degrees of the cell $H_0$ and $\wh H_0$. We will prove the exponential tail in  the setting of $\bbP_\rmQ$ in Lemma~\ref{lem:prop:0-cone-deg-3} and use the following comparison between $\bbP_\rmQ$ and $\bbP_\rmE$ to transfer to the setting of $\bbP_\rmE$.
 {Let $D(\bbR,\bbR^2)$ be the space of functions from $\bbR$ to $\bbR^2$. Then following from ~\cite[Lemma 4.8]{Tutte}, for any function $G_1:D(\bbR,\bbR^2)\times \Omega_\fD\to \{0,1\}$ such that the event
    \begin{equation}
          A_1:=\{\exists s\in I(0):  G_1(Z(\cdot+s) - Z(s), \fD-s) =  1\}
    \end{equation}
    is in $\sigma(Z|_{I(k)}, \fD)$, one has
    \begin{equation}\label{eq:0-gamma-comparison-00}
        \bbP_\rmE [A_1]\leq \bbP_\rmQ[ A_1]^{1/2}2^{C_0k} + 2^{-k/C_0}
    \end{equation}
    where $C_0$ is some constant depending only on $\gamma$. Then combining~\eqref{eq:0-gamma-comparison-00} with the same argument as ~\cite[Corollary 4.3]{Tutte}, we have the following. 
    \begin{proposition}\label{prop:0-cone-gamma-compare}
    Consider functions $G, G_1,G_2,...:D(\bbR,\bbR^2)\times \Omega_\fD\to\{0,1\}$.  For $k\in\bbN$, define
    \begin{equation*}
    \begin{split}
    & A:=\{\exists s\in I(0):  G(Z(\cdot+s) - Z(s), \fD-s) =  1\}
    \\& A_k:=\{\exists s\in I(0):  G_k(Z(\cdot+s) - Z(s), \fD-s) =  1\}. 
    \end{split}
    \end{equation*}
    Suppose for any $k\in\bbN$, $A_k\in \sigma(Z|_{I(k)}, \fD)$ (i.e., $A_k$ is measurable with respect to $(Z|_{I(k)},\fD)$), $A\in\sigma(Z,\fD)$, and $A_k\subseteq A_{k+1}\subseteq A$. Also assume there is some constant $c>0$ such that  {$\bbP_\rmE[\mathds{1}_{A_k}\neq \mathds{1}_{A}]<c^{-1}2^{{-}ck}$}. Then there exists some constant $\e>0$ depending only on $c$ {and $\gamma$} such that
    \begin{equation}\label{eq:0-cone-gamma-compare}
        \bbP_\rmE[A]\leq \e^{-1}\bbP_\rmQ[A]^\e.
    \end{equation}
    \end{proposition}}

    \begin{proposition}\label{prop:0-cone-deg-1}
     Let $H_0$ be the cell in $\cH^\rmE_{ }$ containing 0 and $ \wh H_0$ be the cell in $\wh\cH^\rmE_{ }$ containing 0. Then for any $p>0$, 
    \begin{equation}\label{eq:prop:0-cone-deg-1}
    \bbE_\rmE [\deg(  H_0)^p]<\infty, \ \ \     \bbE_\rmE [\deg( \wh H_0)^p]<\infty.
    \end{equation}
\end{proposition}

Let $M>400$.   Recall the notation $\Lambda:=\mathrm{Double}_\theta$ defined below~\eqref{eq:pf-AAA-nnew}, which is a.s.\ finite by Proposition~\ref{prop:cell-diam2/area}. Let $\Lambda_k\subseteq \Lambda$ be the set of $(a,b)$ where $[a-1,b]\subseteq I(k)$. Let $N_k(a,b) = \#\{x\in\theta+\bbZ:[x-1,x]\subseteq I(k), x\sim a \text{ or } x\sim b\}$, $N(a,b) = \#\{x\in\bbZ+\theta:x\sim a \text{ or } x\sim b\}$, where $\sim$ is the adjacency in the original graph $\cG^\rmE$ or $\cG^\rmQ$. Set $E_k = \{\exists (a,b)\in \Lambda_k: N_k(a,b)>M/4-4\}$, $E = \{\exists (a,b)\in \Lambda: N(a,b)>M/4-4\}$. Since $(\theta,\theta)\in\Lambda$, by definition $E^c\subseteq \{\deg (H_0)<M,\deg(\wh H_0)<M\}$. Let $G_k = \mathds{1}_{E_k}$ and $G = \mathds{1}_E$  be the indicator function for $E_k$ and $E$, respectively,   and let $(A_k,A)$ be defined in terms of $(G_k,G)$ as in Proposition~\ref{prop:0-cone-gamma-compare}.

 \begin{lemma}\label{lem:prop:0-cone-deg-1}
     Let $A,A_k$ be defined as above. Then $A_k=E_k$ and $A=E$. Furthermore, $A_k\in \sigma(Z|_{I(k)},\fD)$, $A\in\sigma(Z,\fD)$, and  $A_k\subseteq A_{k+1}\subseteq A$.
 \end{lemma}
 \begin{proof}
 The first claim is immediate from the definition since the definition of $E_k$ and $E$ is invariant if we add random constants to $Z$, and the fact that the events $E_k$ and $E$ are invariant under translation $(Z,\cD)\mapsto (Z(\cdot+s)-Z(s),\cD-s)$ when $s\in I(0)$. The other claims are then straightforward from the definition of $E_k$ and $E$.
 \end{proof}

\begin{lemma}\label{lem:prop:0-cone-deg-2}
  There is some constant $c>0$ such that $\bbP_\rmE[\mathds{1}_{A_k}\neq \mathds{1}_{A}]\leq c^{-1}2^{{-}ck}$.
\end{lemma}

 \begin{proof}
Let $A_n,B_n,F_n$ be the events considered in the proof of Proposition~\ref{prop:cell-diam2/area}.  By~\eqref{eq:pf-AAAnew}, under the event $F_n$, further using the event $A_n$, we have $$\{\eta([x-1,x]):x\in\bbZ+\theta, x\sim a \text{ or } x\sim b \text{ for some } (a,b)\in\Lambda\}\subseteq B(0;3\fR_{2^n}/4),$$
and therefore $\Lambda_{2n}=\Lambda$ and $N_{2n}(a,b)=N(a,b)$ for each $(a,b)\in \Lambda$ under the event $F_n\cap \{B(0;\fR_{2^n})\subseteq \eta(I(2n))\}$.
Combining~\eqref{eq:Rb-gammaQ} along with~\eqref{eq:SLE-contain-ball-1}, we have $\bbP_\rmE[B(0;\fR_{2^n})\nsubseteq \eta([-2^{1.5n}, 2^{1.5n}])]\lesssim 2^{-cn}$, and using the definition of $\fD$, we have $\bbP_\rmE\big[[-2^{1.5n},2^{1.5n}]\nsubseteq I(2n)\big]\lesssim 2^{-cn}$. This gives $\bbP_\rmE[B(0;\fR_{2^n})\nsubseteq \eta(I(2n))]\lesssim 2^{-cn}$, and together with~\eqref{eq:pf-AAAnew} gives the claim.
 \end{proof}

\begin{lemma}\label{lem:prop:0-cone-deg-3}
    There exists $c>0$ such that $\bbP_\rmQ[A]\leq c^{-1}2^{-cM}$ for $M\geq400$. 
\end{lemma}

\begin{proof}
   By the first claim of Proposition~\ref{prop:cell-diam2/area} for $\bbP_\rmQ$, it follows that for some constant $c_1>0$, $\bbP_\rmQ[\Lambda\nsubseteq [-2^n,2^n]]\lesssim 2^{-c_1n}$. By~\cite[Lemma 2.5]{GMS19har} and translation invariance of Brownian motion, for some constant $c_2>0$, in the map $\cG^\rmQ$, $\bbP_\rmQ[\deg(x)\geq M/16]\lesssim 2^{-c_2M}$ for each vertex $x\in \theta+\bbZ$. We set $n = \lceil \frac{c_2M}{2} \rceil$. By a union bound, for some constant $C>0$ and $c = \min\{c_1c_2/3,c_2/3\}$, we have
    \begin{equation}
        \begin{split}
            \bbP_\rmQ[A]&\leq \bbP_\rmQ\big[\Lambda\nsubseteq [-2^n,2^n]\big]+\bbP_\rmQ\big[\exists x\in (\theta+\bbZ)\cap [-2^n,2^n] \text{ with } \deg(x)\geq M/16\big]\\
            &\leq C(2^{-c_1n} + 2^{n+2}2^{-c_2M})\lesssim 2^{-cM}.
        \end{split}
    \end{equation} 
    This completes the proof.
\end{proof}

\begin{proof}[Proof of Proposition~\ref{prop:0-cone-deg-1}]
    By combining Lemma~\ref{lem:prop:0-cone-deg-1}, Lemma~\ref{lem:prop:0-cone-deg-2}, Lemma~\ref{lem:prop:0-cone-deg-3} with Proposition~\ref{prop:0-cone-gamma-compare}, it follows that for some constant $\e>0$, for all large enough $M>0$, $\bbP_\rmE[\deg(H_0)\geq M]\leq \e^{-1}2^{-\e M}$ and $\bbP_\rmE[\deg(\wh H_0)\geq M]\leq \e^{-1}2^{-\e M}$. This gives the claim.
\end{proof}

\subsubsection{Convergence of the mated-CRT map for 0-quantum cone}\label{subsubsec:cell-config-pf}

By~\eqref{eq:cell-diam2/area-E},~\eqref{eq:cell-diam2/area} and ~\eqref{eq:prop:0-cone-deg-1}, the moment bound~\eqref{eq:thm-inv-circle} holds for both $\cH^\rmE$ and $\wh\cH^\rmE$. Now we check the   conditions in Definition~\ref{def:CellConfigConstraint}. Let $\wh\eta$ be the space-filling curve $\eta$ modulo increasing time reparametrization. Then $\eta$ is measurable with respect to $(h,\wh\eta)$.  Below we write $g$ for the function sending $(h,\wh\eta,\theta)$ to the cell configuration $\wh\cH^\rmE$ as in its definition.  Fix $z\in\bbC$, $a>0$ and let $\tau_z$ be the first time when $\eta$ hits $z$. Let $h^{a,z}(\cdot) = h(a\cdot+z)+Q\log a$, $\eta^{a,z}(t) = a^{-1}(\eta(t+\tau_z)-z)$ and $\theta^{a,z} = [\theta-\tau_z]$ where $[x]$ is the fractional part of $x$. By the LQG coordinate change formula, $\eta^{a,z}$ is parameterized via the LQG area measure with respect to $h^{a,z}$, and from our construction \eqb\label{eq:SCALE-cH}  g(h^{a,z},a^{-1}(\wh\eta-z),\theta^{a,z}) = a^{-1}(g(h,\wh\eta,\theta)-z).\eqe

\begin{lemma}\label{lem:ergodicityAaAb}
    The cell configuration $\wh\cH^\rmE$ is translation invariant modulo scaling as in Condition 1.\ of Definition~\ref{def:CellConfigConstraint}.
\end{lemma}

\begin{proof}
    Let $\mathfrak{D}_\bbC$ be a uniform dyadic system defined in~\cite[Section 2.1]{GMSinvariance} independent of $(h,\eta,\theta)$. For $m\in\bbN$, let $S_m$ be the largest square $S$ of $\mathfrak{D}_\bbC$ containing 0 such that $\mu_h(S)\leq m$. Conditioned on $(h,\mathfrak{D}_\bbC)$, let $z_m$ be sampled uniformly from the Lebesgue measure on $S_m$ independent of $\theta$. Then by~\cite[Lemma A.4]{GMS21tutte}, for some random constant $C_m$ depending on $(h,\mathfrak{D}_\bbC,z_m)$, $h_m:=h(C_m^{-1}\cdot+z_m)-Q\log C_m\overset{d}{=}h$. Set $\wh \eta_m = C_m(\wh\eta-z_m)$ and $\theta_m=[\theta-\tau_{z_m}]$. Then conditioned on $(h_m,z_m,\mathfrak{D}_\bbC)$, $\wh\eta_m$ has the law as whole-plane space-filling SLE modulo time parameterization, and conditioned on $(h_m,z_m,\mathfrak{D}_\bbC,\wh\eta_m)$, $\theta_m$ is uniform in $(0,1)$ since $\theta$ is uniform in $(0,1)$ and independent of $(h,\eta)$. Therefore $(h_m,\wh\eta_m,\theta_m)\overset{d}{=}(h,\wh\eta,\theta)$, and by~\eqref{eq:SCALE-cH}, Condition 1.\ of Definition~\ref{def:CellConfigConstraint} holds (with in fact equality in law) for the (interior of the) square $S_m$ and $C_m$.
\end{proof}

\begin{lemma}\label{lem:ergodicityAaA}
    The cell configuration $\wh\cH^\rmE$ is ergodic modulo scaling as in Condition 2.\ of Definition~\ref{def:CellConfigConstraint}.
\end{lemma}

\begin{proof}
    Suppose $F$ is a function that $F(C(\wh\cH^\rmE-z))=F(\wh\cH^\rmE)$ for every $C>0$ and $z\in\bbC$. Let $G(h,\wh\eta,\theta) = F(g(h,\wh\eta,\theta))$. Then~\eqref{eq:SCALE-cH} implies that for any deterministic $a>0,z\in\bbC$, $G(h,\wh\eta,\theta) = G(h(a\cdot+z)+Q\log a, a^{-1}(\wh\eta-z), [\theta-\tau_z])$. The proof is then identical to the corresponding proof of~\cite[Proposition 3.1]{GMS21tutte}.
\end{proof}


 \begin{lemma}\label{lem:config-pack}
     The cell configuration $\wh\cH^\rmE$ satisfies all the conditions in Definition~\ref{def:CellConfigConstraint}.
 \end{lemma}

 \begin{proof}
   By Lemma~\ref{lem:ergodicityAaAb} and Lemma~\ref{lem:ergodicityAaA}, Condition 1.\ and 2.\ in Definition~\ref{def:CellConfigConstraint} hold for $\wh\cH^\rmE$. For the line connectivity property in Condition 3., our construction of $\wh\cH^\rmE$ induces a natural map $\pi:\cH^\rmE\to\wh\cH^\rmE$ where $\pi(H)$ is the cell in $\wh\cH^\rmE$ containing $H$. In particular, from the merging process, we have $\pi(H)=\pi(H')$ or $\pi(H)\sim\pi(H')$ if $H\sim H'$ and $\wh\cH^\rmE(\ell)=\pi(\cH^\rmE(\ell))$ for line segments $\ell$, and Condition 3.\ then follows from the line connectivity for $\cH^\rmE$ as in~\cite[Proposition 3.1]{GMS21tutte}.  Condition 4.\ is a consequence of the fact that the law of $(h,\eta)$ is  invariant under rotations about the origin.   
  \end{proof}

\begin{lemma}\label{lem:config-unif}
    The cell configuration $\cH^\rmE$ satisfies all the conditions in Definition~\ref{def:CellConfigConstraint}.
\end{lemma}

\begin{proof}
    This is exactly proved in~\cite[Proposition 3.1]{GMS21tutte}.
\end{proof}

We are now ready to prove the convergence of mated-CRT map on 0-quantum cone under circle packing and Riemann uniformization. We will constantly use the following convention. Let $\wt\cG^{1/n}$ be some variant of the mated-CRT map, and let $x\in\bbR$   be viewed as a vertex on the mated-CRT map. For a  circle packing $\wt\cP^n$ for $\wt\cG^{1/n}$, we write $o_n(x)$ for the center of the disk for the vertex $x$ in $\wt\cP^n$. We will also write $x$ for the vertex on the Riemann surface $M(\wt\cG^{1/n})$ associated with $\wt\cG^{1/n}$.

\begin{theorem}\label{thm:circle-matedCRT-0-cone}
Let $(h,\eta,\theta)$ be sampled from the probability measure $\bbP_\rmE$, and $\cG^{\rmE,1/n }$  be defined as in the beginning of Section~\ref{subsubsec:cell-config}. Almost surely, the map $\cG^{\rmE,1/n }$  is circle packing parabolic. Furthermore, there exists a circle packing $\cP^{\rmE,n}$ for $\cG^{\rmE,1/n }$ in the sense of Definition~\ref{def:pack-non-simple} such that the following is true: 
\begin{enumerate}
    \item For each $R>0$, the maximal diameter of disks intersecting $B(0;R)$ tends to 0 in probability as $n\to\infty$;
    \item We have
      \begin{equation}
          \sup\big\{|\eta(x)-o_n(x)|: x\in\frac{1}{n}\bbZ+\frac{1}{n}\theta, \eta(x)\in B(0;R)  \big\} \to 0 \text{ in probability as } n\to\infty;
      \end{equation}
\end{enumerate}
\end{theorem}

\begin{proof}
    By Lemma~\ref{lem:config-pack}, Proposition~\ref{prop:cell-diam2/area} and Proposition~\ref{prop:0-cone-deg-1}, $\wh\cH^\rmE$ satisfies all the conditions in Theorem~\ref{thm:CellSystemPack}. Furthermore, since the a priori embedding $\wh\cH^\rmE$ preserves the orientation of $\wh\cG^\rmE$, it must be the the circle packing for $\wh\cG^\rmE$ rather than its reflection that approximate $\wh\cH^\rmE$.
    Furthermore, by~\cite[Lemma 2.9]{GMSinvariance}, for each $\e>0$, almost surely, there exists $R>0$ such that for every $r>R$, all the cells in $\wh\cH^\rmE$ intersecting $B(0;r)$ have diameter less than $\e r$. Furthermore, for any $x\in \bbZ+\theta$, let $\wh H(x)$ be the cell in $\wh\cH^\rmE$ containing $\eta([x-1,x])$. Then in the extended circle packing,  $|\eta(x)-o(x)|\leq \diam(\wh H(x))+|c(\wh H(x))-o(\wh H(x))|+\diam (D_{\wh H(x)})$ where  $D_{\wh H(x)}$ is the disk for $\wh H(x)$ in the circle packing.
  Then by the no-macroscopic-disk conclusion as in Theorem~\ref{thm:CellSystemPack}, the conclusions in Theorem~\ref{thm:CellSystemPack} continue to hold if work with $\cH^\rmE$ instead of $\wh\cH^\rmE$ and consider the extended circle packing for $\cG^\rmE$ as in Definition~\ref{def:pack-non-simple}. Theorem~\ref{thm:circle-matedCRT-0-cone} then follow from the scale invariance using $\fR_n$ as in Lemma~\ref{lem:Rb}.
\end{proof}

\begin{theorem}\label{thm:unif-matedCRT-0-cone}
 Let $(h,\eta,\theta)$ and $\cG^{\rmE,1/n }$ be as in Theorem~\ref{thm:circle-matedCRT-0-cone}. For $n\geq1$, the Riemann surface  $M(\cG^{\rmE,1/n })$ associated with  $\cG^{\rmE,1/n }$ is a.s.\  parabolic. Moreover,  there exists a conformal map $\varphi^{\rmE,n}: M(\cG^{\rmE,1/n })\to\bbC$  such that the following is true. For each $x\in \frac{1}{n}(\theta+{\bbZ})$, we also  write     $x$ for its vertex on the Riemann surface $M(\cG^{\rmE,1/n })$. Then for any $r>0$, as $n\to\infty$, 
  \begin{enumerate}
      \item For each $R>0$, the maximal diameter of edges  $\varphi^{\rmE,n}(\fre)$  intersecting $B(0;R)$ tends to 0 in probability;
      \item We have $\varphi^{\rmE,n}(\frac{1}{n}\theta)=0$, and
      \begin{equation}
          \sup\big\{|\eta(x)-\varphi^{\rmE,n}(x)|: x\in\frac{1}{n}(\theta+\bbZ), \eta(x)\in B(0;R)  \big\} \to 0 \text{ in probability as } n\to\infty;
      \end{equation}
  \end{enumerate}
\end{theorem}
\begin{proof}
    By~\eqref{eq:cell-diam2/area-E} and ~\eqref{eq:prop:0-cone-deg-1}, the moment bound~\eqref{eq:thm-inv-circle} holds for both $\cH^\rmE$, and further by Lemma~\ref{lem:config-unif}, $\cH^\rmE$ satisfies all the conditions in Theorem~\ref{thm:invariance-uniformization}. Now we claim that for large $r$,  the maximal diameter of edges $\varphi^{\rmE}(\fre)$ intersecting $B(0;r)$ has order $o(r)$. Suppose $\fre$ has an endpoint $x\in \bbZ+\theta$ and let $H_x=\eta([x-1,x])$. For large enough $R$, if $|\eta(x)|\geq R$, using the no-macroscopic-cell result from ~\cite[Lemma 2.9]{GMSinvariance}, the cell for the other endpoint of $\fre$ has to be contained within $B(0;2|\eta(x)|)$, and by~\eqref{eq:thm:invariance-uniform-1}, $\diam(\varphi^{\rmE}(\fre)) = o(|\eta(x)|)$ and has to be contained within $B(0;2|\eta(x)|)\backslash B(0,|\eta(x)|/2)$, which verifies our claim.
    Theorem~\ref{thm:unif-matedCRT-0-cone} then follows from  Theorem~\ref{thm:invariance-uniformization} and the scaling in Lemma~\ref{lem:Rb}.
\end{proof}

\subsection{Convergence of  mated-CRT map under circle packing}\label{subsec:matedCRT-gamma-cone}

In this subsection we use the results in Section~\ref{sec:CP-conformal}    along with local absolute continuity between quantum surfaces to transfer the convergence of circle packing in Theorem~\ref{thm:circle-matedCRT-0-cone} to  the convergence of circle packing for mated CRT map in various topologies. Throughout this section, we will constantly write $(h^\rmQ,\eta)\sim \bbP_\rmQ$  for the $\gamma$-quantum cone $(\bbC,h^\rmQ,0,\infty)$ and the space-filling $\SLE_\kappa$ sampled from the probability measure $\bbP_\rmQ$, and $(h^\rmE,\eta)\sim \bbP_\rmE$  for the $0$-quantum cone $(\bbC,h^\rmE,0,\infty)$ and the space-filling $\SLE_\kappa$ sampled from the probability measure $\bbP_\rmE$.

One major difficulty is that the non-locality of the operator $\Phi_{\rm simple}$ and we need localized versions to use absolute continuity to do the transfer. In Section~\ref{subsubsec:matedCRT-cell}, for $a<t<b$, we define a disk triangulation $\cG^{\rmE,1/n,\theta}_{[a,b],t}$ as in Figure~\ref{fig:quantum-cell}, which is measurable with respect to $(\eta|_{[a,b]},\theta)$ and with high probability, $\Phi_{\rm simple}(\cG^{\rmE,1/n,\theta}_{[a,b],t})$ is a submap of $\Phi_{\rm simple}(\cG^{\rmE,1/n})$. This allows us to cut a piece $\eta([a,b])$ off the 0-quantum cone in Theorem~\ref{thm:circle-matedCRT-0-cone} and use Theorem~\ref{thm:He-Schramm-extension-1} to compare the circle packing for $\cG^{\rmE,1/n,\theta}_{[a,b],t}$ and the circle packing $\cP^{\rmE,n}$ and prove the convergence of circle packing for the mated-CRT map restricted to this piece (Lemma~\ref{lem:circle-pack-cell-1}). In Proposition~\ref{thm:circle-quantum-cell-0}, we further use local absolute continuity and standard techniques to transfer to the setting of a piece $\eta([a,b])$ cut off from the $\gamma$-quantum cone. In Section~\ref{subsubsec:matedCRT-disk}, we prove Theorem~\ref{thm:pack-disk}, where we first use absolute continuity between $\eta^\disk([\e,1-\e])$ and $\eta^\rmQ([\e,1-\e])$ to transfer Proposition~\ref{thm:circle-quantum-cell-0} to the setting of $\eta^\disk([\e,1-\e])$ and further send $\e\to 0$ and apply Theorem~\ref{thm:He-Schramm-disk}. In Section~\ref{subsubsec:matedCRT-sph}, we use   similar methods along with Theorem~\ref{thm:He-Schramm-sphere} to prove Theorem~\ref{thm:pack-sph}. Finally, in Section~\ref{subsubsec:matedCRT-cone}, we prove Theorem~\ref{thm:pack-cone} by applying the scale invariance of quantum cones  in Lemma~\ref{lem:Rb} to zoom in near 0 in the setting of Proposition~\ref{thm:circle-quantum-cell-0} and further applying Theorem~\ref{thm:He-Schramm-sphere}.
 
\subsubsection{Convergence of circle packing for mated-CRT cells}\label{subsubsec:matedCRT-cell}

Our starting point is to prove a version of convergence the mated-CRT map $\cG^{1/n}$  restricted to the time $[a,b]$ under circle packing. Recall the operation of collapsing double edges as discussed in Definition~\ref{def:pack-non-simple}. This operation is non-local. To transfer from the convergence in the infinite volume setting as in Theorem~\ref{thm:circle-matedCRT-0-cone} to the finite volume setting, we define a version of cell configuration which is locally measurable with respect to $\eta$ as follows.

Let  $h$ be some variant of GFF and $\eta$ be some variant of space-filling $\SLE_\kappa$ parameterized with respect to the $\gamma$-LQG mass with respect to $h$. Let $a<t<b$ which are possibly random  such that $\eta(t)\notin\partial\eta([a,b])$ almost surely,  and let $\theta\in[0,1)$. We define  operators $\Phi_{[a,b],t}^{\theta,n}$, $\frt_{[a,b],t}^{\theta,n,-}$ and $\frt_{[a,b],t}^{\theta,n,+}$  as follows.  Let $a_n(\theta)$   be the smallest  number in $(a,b]\cap \frac{1}{n}(\theta+\bbZ)$, and let $b_n(\theta) = a_n(\theta)+\frac{1}{n}(\lfloor n(b-a) \rfloor-1)$.  Let $t^-\in  \frac{1}{n}(\theta+\bbZ)$ be the largest time in $[a_n(\theta),t]$ such that $\eta([t^-,t^-+\frac{1}{n}])$ intersects both the left and right boundaries of $\eta([a_n(\theta),b_n(\theta)])$ along a non-trivial boundary arc,   $t^+\in  \frac{1}{n}(\theta+\bbZ)$ be the smallest time in $[t,b_n(\theta)]$ such that $\eta([t^+-\frac{1}{n},t^+])$ intersects both the left and right boundaries of $\eta([a_n(\theta),b_n(\theta)])$ along a non-trivial boundary arc.  We let $\Phi_{[a,b],t}^{\theta,n}(h,\eta)$ be the collection of cells $\{\eta([x-\frac{1}{n},x]):x\in \frac{1}{n}(\bbZ+\theta)\cap (t^-,t^+]\}$. Using that $\eta$ is non-self-crossing, the map induced by the cells in $\Phi_{[a,b],t}^{\theta,n}(h,\eta)$ is a disk triangulation. We  set $\frt_{[a,b],t}^{\theta,n,-}(h,\eta)=t^-$,   and  $\frt_{[a,b],t}^{\theta,n,+}(h,\eta)=t^+$. Let $\wh\Phi_{[a,b],t}^{\theta,n}(h,\eta)$ be the version of $\Phi_{[a,b],t}^{\theta,n}(h,\eta)$ with the double edges collapsed, i.e., if two cells $\eta([x_1-\frac{1}{n},x_1])$ and $\eta([x_2-\frac{1}{n},x_2])$ where $x_1<x_2-\frac{1}{n}$ intersect along their both  left and right boundaries intersect along a non-trivial boundary arc, then we merge $\eta([x_1,x_2-\frac{1}{n}])$ into $\eta([x_2-\frac{1}{n},x_2])$.  See Figure~\ref{fig:quantum-cell} for an illustration.

  
 Under the event that $\eta(t)\notin\partial\eta([a,b])$, let $D_{[a,b],t}$ be the connected component of the interior of $\eta([a,b])$ containing $\eta(t)$, which is a Jordan domain. Let $\psi_{[a,b],t}:D_{[a,b],t}\to\bbD$ be a conformal map sending $\eta(t)$ to 0.

  \begin{figure}
    \centering
    \includegraphics[scale=0.77]{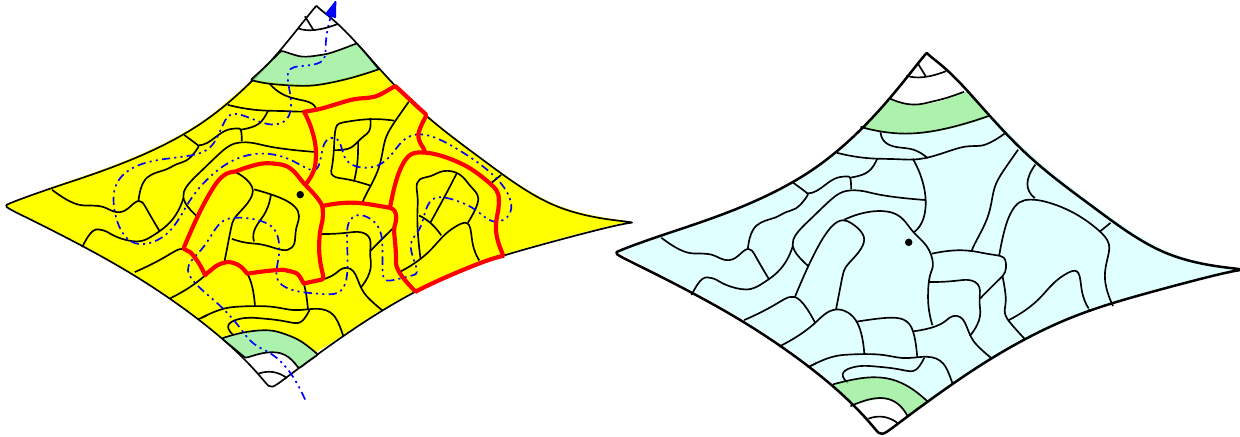}
    \caption{\textbf{Left}: An illustration of the definition $\Phi_{[a,b],t}^{\theta,n},\frt_{[a,b],t}^{\theta,n,-}$ and $\frt_{[a,b],t}^{\theta,n,+}$  where $\kappa\geq 8$. The cells and the blue trace represent $\eta([a_n(\theta),b_n(\theta)])$. The black dot in the middle represents $\eta(t)$. The cells surrounded by the red curves are either part of or surrounded by an double edge. The space-filling curve $\eta$ traverses the cells in the order depicted. The cells $\eta([\frt_{[a,b],t}^{\theta,n,-},\frt_{[a,b],t}^{\theta,n,-}+\frac{1}{n}])$ and $\eta([\frt_{[a,b],t}^{\theta,n,+}-\frac{1}{n},\frt_{[a,b],t}^{\theta,n,+}])$, which are the last cell before $t$ and the first cell after $t$ having non-empty intersection with both left and right boundary,  are colored in green. The set  $\Phi_{[a,b],t}^{\theta,n}(h,\eta)$ consists of thw two green cells and the yellow cells in their middle visited by $\eta$. The other cells in  $\Phi_{[a,b],t}^{\theta,n}(h,\eta)$ intersect only the left or right boundary of  $\eta([a_n(\theta),b_n(\theta)])$. \textbf{Right}: The cells in $\wh \Phi_{[a,b],t}^{\theta,n}$, which are colored in green and blue and obtained by collapsing the double edges involving the yellow and green cells on the left panel.}
    \label{fig:quantum-cell}
\end{figure}

Our goal is to prove the following.

\begin{proposition}\label{thm:circle-quantum-cell-0}
     Let  $a<t<b$ be deterministic, and $(h^\rmQ,\eta)\sim\bbP_\rmQ$. Let  $\psi_{[a,b],t}$ be defined as above, and let $\cH^{\rmQ,1/n,0}_{[a,b],t} =\Phi_{[a,b],t}^{0,n}(h^\rmQ,\eta)$  (with $\theta=0$). Let $\cG^{\rmQ,1/n,0}_{[a,b],t}$ be the submap of $\cG^{1/n}$ induced by $\cH^{\rmQ,1/n,0}_{[a,b],t}$. Let  $t_{a}^{n,-}=\frt_{[a,b],t}^{0,n,-}(h^\rmQ,\eta)$ and $t^{n,+}_b=\frt_{[a,b],t}^{0,n,+}(h^\rmQ,\eta)$.   There is a circle packing $\cP^{\rmQ,1/n,0}_{[a,b],t}$ for $\cG^{\rmQ,1/n,0}_{[a,b],t}$ in $\bbD$ in the sense of Definition~\ref{def:pack-non-simple} such that the following is true. 
     The maximal diameter of disks in $\cP^{\rmQ,n,0}_{[a,b],t}$ tends to 0 in probability, and up to a rotation, 
      \begin{equation}\label{eq:thm:circle-quantum-cell}
         \max_{x\in  (t_a^{n,-},t_b^{n,+})\cap \frac{1}{n}\bbZ }\dist(\psi_{[a,b],t}(\eta(x)), o_n(x))\to 0 \ \ \text{in probability.}
     \end{equation} 
\end{proposition}

Our proof follows a similar transfer as in ~\cite[Proof of Proposition 2.10]{GMS19har}.
Let $U = B(0;1/8)$ and $C_U=\bbE_\rmE[\mu_h(U)]<\infty$. Define the weighted probability measure $\bbP_*(dh,d\eta,dz_*) = C_U^{-1}\mathds{1}_U(z_*)\mu_h(dz_*)\bbP_\rmE(dh,d \eta)$ where $\eta$ is parameterized by $\mu_h$. In other words, we weight the law of $(h^\rmE,\eta)$ under $\bbP_\rmE$ by $\mu_{h^\rmE}(U)$ and sample a point $z_*\in U$ from the probability measure proportional to $\mu_{h^\rmE}|_U$. Let $\tau_*$ be the time when $\eta$ hits $z_*$. We first prove that the independent random variable $\theta$ can be removed if we recenter at $z_*$. Let $\cH^{*,1/n}$, $\wh\cH^{*,1/n}$, $\cG^{*,1/n}$, $\wh\cG^{*,1/n}$ be the corresponding versions of $\cH^{\rmE,1/n}$, $\wh\cH^{\rmE,1/n}$, $\cG^{\rmE,1/n}$, $\wh\cG^{\rmE,1/n}$ defined using $h$ and $[\tau_*]$ in place of $h^\rmE$ and $\theta$,  where $[x]$ is the fractional part of $x\in\bbR$.

\begin{lemma}\label{lem:0-cone-phase-0}
   Let $\beta_n\subseteq [0,1)$ be  deterministic, $(h,\eta,z_*)$ be sampled from $\bbP_*$ and $\tau_*$ be the time when $\eta$ hits $z_*$. Then the conclusions of Theorem~\ref{thm:circle-matedCRT-0-cone} continue to hold for some circle packing $\cP^{*,n}$ for $\cG^{*,1/n}$ if we replace $(\bbP_\rmE,\cG^{\rmE,1/n},\frac{1}{n}\theta)$ with $(\bbP_*,\cG^{*,1/n},\frac{1}{n}[n\tau_*+\beta_n])$, i.e., the map $\cG^{*,1/n}$ is defined in the identical way as $\cG^{\rmE,1/n}$ using $(h,\eta)$ and vertex set $\frac{1}{n}\bbZ+\frac{1}{n}[n\tau_*+\beta_n]$. Furthermore, for each $R>0$, the maximal diameter of cells in $\wh\cH^{*,1/n}$ intersecting $B(0;R)$ tends to 0 in probability as $n\to\infty$.
\end{lemma}

\begin{proof}
The circle packing parabolicity of $\cG^{*,1/n}$ follows since for every fixed $n$, the law of $\cG^{*,1/n}$ is absolutely continuous with respect to that of $\cG^{\rmE,1/n}$.
    Let $E_{n,R,\e}(h,\eta,\theta)$ be the event that there is a circle packing for the $\cG^{1/n}$ map defined via $(h,\eta)$ with vertex set $\frac{1}{n}(\bbZ+\theta)$ such that the maximum of the quantities considered in part 1.\ and part 2.\ of Theorem~\ref{thm:circle-matedCRT-0-cone} is no greater than $\e$. Since $\eta$ is parameterized by LQG area and $z_*$ is sampled according to LQG area measure, for each fixed $N\in\bbN$, we have
    \eqb\begin{split}
        \bbP_*[E_{n,R,\e}^c&(h,\eta,[n\tau_*+\beta_n])\cap \{|\tau_*|\leq N\}] = C_U^{-1}\int_{-N}^N\bbP_\rmE\big[\{\eta(s)\in U\}\cap E_{n,R,\e}^c(h,\eta,[ns+\beta_n])\big]\,ds \\ &\leq C_U^{-1}\int_{-N}^N \bbP_\rmE[E_{n,R,\e}^c(h,\eta,[ns+\beta_n])]\,ds \leq 3C_U^{-1}N\bbP_\rmE[E_{n,R,\e}^c(h,\eta,\theta)].
    \end{split}
    \eqe
   Here we used   $\bbP_\rmE[E_{n,R,\e}^c(h,\eta,\theta)] = \int_0^1\bbP_\rmE[E_{n,R,\e}^c(h,\eta,s)]\,ds$. By sending $N\to\infty$ sufficiently slower than $n\to\infty$, we have $ \bbP_*[E_{n,R,\e}^c(h,\eta,[n\tau_*+\beta_n])]\to 0$ as $n\to\infty$. The  claim that Theorem~\ref{thm:circle-matedCRT-0-cone} in this setting  then follows since $R,\e$ can be arbitrary.  The last claim follows from Lemma~\ref{lem:diam-cell-max-1} and an identical argument.
\end{proof}

\begin{lemma}\label{lem:circle-pack-cell}
Let $(h,\eta,z_*)\sim\bbP_*$ and $\beta_n\in[0,1)$ be deterministic. Let $k\in\bbN$, $\theta=[n\tau_*+\beta_n]$ and $(a,t,b) = (\tau_*-1/k,\tau_*,\tau_*+1/k)$. Let $t^-=\frt_{[a,b],t}^{\theta,n,-}(h,\eta)$ and $t^+=\frt_{[a,b],t}^{\theta,n,+}(h,\eta)$. The following are true.
\begin{enumerate}[(i)]
    \item Under the event that neither $\eta([t^-+\frac{1}{n},t])$ nor $\eta([t,t^+-\frac{1}{n}])$ is contained inside a single cell $\wh H\in\wh\cH^{*,1/n}$, the cells in $\wh\Phi_{[a,b],t}^{\theta,n}(h ,\eta)$ other than the first and last cell visited by $\eta$ are also cells in  $\wh\cH^{*,1/n}$.
   In this case, the map $\wh\cG^{*,1/n}_{[a,b],t}$ induced by $\wh\Phi_{[a,b],t}^{\theta,n}(h,\eta)$  is a submap of $\wh\cG^{*,1/n}$, where edges of $\wh\cG^{*,1/n}$ with both endpoints in $\wh\cG^{*,1/n}_{[a,b],t}$ are also edges of $\wh\cG^{*,1/n}_{[a,b],t}$.
    \item Let $t_a^-$ and $t_b^+$ be such that $\eta([t_a^-,t_b^+]) = \ol{D_{[a,b],t}}$. The times $(t^-,t^+)$
    converge to $(t_a^-,t_b^+)$ almost surely as $n\to\infty$.  
\item 
With probability tending to 1 as $n\to\infty$, the map $\wh\cG^{*,1/n}_{[a,b],t}$ is a simple disk triangulation and a submap of $\wh\cG^{*,1/n}$. Consider the embedding $\wh\gamma_n$ of $\partial\wh\cG^{*,1/n}_{[a,b],t}$ where we connect the centers of the corresponding cells of $\wh\Phi_{[a,b],t}^{\theta,n}(h,\eta)$ by line segments. Then   $\wh\gamma_n$ converges to $\partial D_{[a,b],t}$ under the curve metric~\eqref{eq:curve-metric-0} as $n\to\infty$ in probability.
\end{enumerate}
The same conclusions also hold for $(h,\eta)\sim \bbP_\rmQ$, $\theta=0$ and $a<t<b$ are deterministic. 
\end{lemma}

\begin{proof}
  Since $\eta$ is non-self-crossing, for any two cells $H_1$ and $H_2$ that produces a double edge in $\cH^{*,1/n}$, $\eta$ must first visit $H_1$ or $H_2$, then the content between them, and finally the other cell.  Suppose $\eta(I_1)$ and $\eta(I_2)$ form a double edge in the map $\cH^{*,1/n}$, such that $I_1\subseteq[t^-+\frac1n,t^+-\frac1n]$. If $I_2\cap(a_n(\theta),b_n(\theta))=\emptyset$, then $\eta(I_1)$ must intersect both the left and right boundaries of $\eta([a_n(\theta),b_n(\theta)])$ along a nontrivial arc. If $I_2\subseteq [b_n(\theta),\infty)$, then $I_1\cap [t,\infty)\neq\emptyset$, otherwise $\eta([t,b_n(\theta)])$ will be contained in a single cell of $\wh\cH^{*,1/n}$; the further implies that $I_1\subseteq [t^+-\frac{1}{n},\infty)$, which is again a contradiction.  We rule out the case $I_2\subseteq (-\infty,a_n(\theta)]$ similarly. Now observe that $\eta(I_1)$ is separated from $\eta([a_n(\theta),b_n(\theta)])\backslash\eta([t^-,t^+])$ by $\eta([t^-,t^-+\frac{1}{n}])$  and $\eta([t^+-\frac1n,t^+])$, which further implies that $I_2\subseteq [t^-,t^+]$. Therefore if a cell $\wh H\in \wh\cH^{*,1/n}$ intersects the interior of $\eta([t^-+\frac{1}{n},t^+-\frac{1}{n}])$, it must either be contained in  $\eta([t^-,t^+])$ or contain $\eta([t^-+\frac{1}{n},t^+-\frac{1}{n}])$, while the latter case is ruled out by our assumption. Furthermore, if the vertex $t^-+\frac{1}{n}$ is swallowed within a double edge $u^-<t^-+\frac{1}{n}<u^+$, then $\eta([t^-,u^+])$ is in a single cell $\wh H\in \wh\cH^{*,1/n}$, so the assumption that  we must have $u^+< t^+$ by our assumption that $\eta([t^-+\frac{1}{n},t])$ is not in a single cell of $\wh H$. We also cannot have $u^+<t^+$ since $\eta([u^+-\frac1n,u^+])\subseteq  \eta([a_n(\theta),b_n(\theta)])$  and only intersects at most one side of $\partial\eta([a_n(\theta),b_n(\theta)])$ along a nontrivial boundary arc, which implies that $\eta([u^--\frac1n,u^-])\cup\eta([u^+-\frac1n,u^+])$ cannot enclose $\eta([t^-,t^-+\frac1n])$. Therefore $t^-+\frac1n$ is retained as we collapse the double edges, and by a symmetric argument the vertex $t^+$ also survives in the double edge collapsing.  This gives claim (i), where the claim on edges in the last sentence follows from the 
  fact that the operation of collapsing   double edges retains the adjacency between surviving vertices.
  
  To prove claim (ii), observe that by the definition of $t^+$, we have $t^+-\frac{1}{n}\leq t_b^+$. We may further assume $t_b^+\geq t^++\frac1n$ since the other case does not affect our conclusion on convergence of $t^+$ to $t^+_b$. 
  By our definition of $t^+$,  $\eta([t^+-\frac1n,t^+])$ intersects both the left and right boundaries of $\eta([t_a^-,t_b^+])$. Since $t_b^+>t^+\geq t$, $\partial\eta([t^+-\frac1n,t^+])$ either contains $\eta(t)$ or separates $\eta(t)$ from $\eta(t_b^+)$ in the domain $\ol {D_{[a,b],t}}$ since $\eta$ is non-self-crossing. 
  By uniform continuity of $\eta$ on $[a,b]$, we have $\diam(\eta([t^+-\frac1n,t^+]))\to 0$ almost surely as $n\to\infty$. Since $D_{[a,b],t}$ is a Jordan domain with $\eta(t)$ in the interior and $\eta$ is non-self-crossing, this readily implies that $\diam(\eta([t^+-\frac1n,t_b^+]))\to 0$ almost surely as $n\to\infty$, which further implies that $t^+\to t_b^+$ almost surely since the LQG measure a.s.\ has no atoms. The claim that $t^-\to t_a^-$ almost surely follows similarly.
  
  For claim (iii), the first part follows from claim (i) and the last claim of Lemma~\ref{lem:0-cone-phase-0}. Now note that under the embedding via $\wh\Phi_{[a,b],t}^{\theta,n}(h,\eta)$, the cells for the vertices on the counterclockwise boundary arc of $\partial\wh\cG^{*,1/n}_{[a,b],t}$ from {the vertex $t^-+\frac1n$ to the vertex $t^+$} intersect $\partial D_{[a,b],t}$ and are visited by $\eta$ in the same order, and the same applies for clockwise boundary arc from  the vertex $t^-+\frac1n$ to the vertex $t^+$. {By claim (ii), the diameter of  $\eta([t^+-\frac1n,t_b^+])$ and the diameter of  $\eta([t_a^-,t^-+\frac1n])$  tend to almost surely.} This gives the second part.
\end{proof}

Let $k>0$,  $h$ be some variant of the 0-quantum cone or $\gamma$-quantum cone, $\eta$ be space-filling $\SLE_\kappa$ independent from $h$ and parameterized by $\gamma$-LQG mass with respect to $h$ such that $\eta(0)=0$. For deterministic $z\in\bbC$, let $\tau_z$ be the time when $\eta$ hits $z$. Let $(a,b,t) = (\tau_z-\frac1k,\tau_z+\frac1k,\tau_z)$,  $\cH^{h,1/n}_{k,z} = \Phi_{[a,b],t}^{\theta,n}(h,\eta)$, $\wh\cH^{h,1/n}_{k,z} = \wh\Phi_{[a,b],t}^{\theta,n}(h,\eta)$, and let $\cG^{h,1/n}_{k,z}$ and $\wh\cG^{h,1/n}_{k,z}$ be the map induced by the cells in $\cH^{h,1/n}_{k,z}$ and $\wh \cH^{h,1/n}_{k,z}$. Let $D_{k,z}$ be the  connected component of the interior of $\eta([\tau_z-\frac1k,\tau_z+\frac1k])$ containing $z$, and $\psi_{k,z}:D_{k,z}\to\bbD$ be a conformal map sending $z$ to 0. Let $t^{n,-}_k=\frt_{[a,b],t}^{\theta,n,-}(h,\eta)$ and $t^{n,+}_k=\frt_{[a,b],t}^{\theta,n,+}(h,\eta)$. Let $F^{h,n,\theta}_{k,z,\e}$ be the following event: there exists a circle packing   $\cP^{h,1/n}_{k,z}$ for $\cG^{h,1/n}_{k,z}$ in $\bbD$ in the sense of Definition~\ref{def:pack-non-simple}, such that the maximal diameter of disks in $\cP^{h,1/n}_{k,z}$ is less than $\e$, and 
$$ \max_{x\in (t_k^{n,-},t_k^{n,+})\cap \frac{1}{n}(\theta+\bbZ)}\dist(\psi_{k,z}(\eta(x)), o_n(x))<\e .$$

\begin{lemma}\label{lem:circle-pack-cell-1}
   Fix $k>0$.    Under the probability measure $\bbP_*$, for $\theta=[n\tau_*+\beta_n]$, we have $\bbP_*[F^{h,n,\theta}_{k,z_*,\e}]\to 1$ as $n\to\infty$.
\end{lemma}

\begin{proof}
    Let   $\wh H_*^n\in \wh\cH^{*,1/n}$ be a cell containing $z_*$. Consider a circle packing $\cP^{h,n}_{k,z_*}$ for $\wh\cG^{h,1/n}_{k,z_*}$ in $\bbD$ centered at the vertex $\wh H_*^n$. Now by Lemma~\ref{lem:circle-pack-cell}, under $\bbP_*$, with probability tending to 1 as $n\to\infty$, the map $\wh\cG^{h,1/n}_{k,z_*}$ is a submap of $\wh\cG^{*,1/n}$ where edges of $\wh\cG^{*,1/n}$ with both endpoints in $\wh\cG^{h,1/n}_{k,z_*}$ are also edges of $\wh\cG^{h,1/n}_{k,z_*}$. Therefore we can apply Theorem~\ref{thm:He-Schramm-extension-1} to compare the circle packing $\cP^{h,n}_{k,z_*}$ with the circle embedding $\wt\cP^{h,n}_{k,z_*}$, which is the restriction of the circle packing $\cP^{*,n}$ in Lemma~\ref{lem:0-cone-phase-0}   to the vertices in $\wh\cG^{h,1/n}_{k,z}$ and shifted such that the disk for $\wh H_*^n$ is centered at $z_*$. We let $\cP^n = \wt\cP^{h,n}_{k,z_*}$, $D = D_{k,z_*}$, $\wt\cP^n = \cP^{h ,n}_{k,z_*}$ and $\wt D= \bbD$. Theorem~\ref{thm:He-Schramm-extension-0} implies that the maximal diameter of disks in $\wt\cP^n$ tends to 0 in probability. Then Theorem~\ref{thm:circle-matedCRT-0-cone} and part (iii) of Lemma~\ref{lem:circle-pack-cell} implies that all the conditions in Theorem~\ref{thm:He-Schramm-extension-1} hold for $\cP^n$ and $\wt\cP^n$, and therefore the claim further follows from  Theorem~\ref{thm:He-Schramm-extension-1} and Theorem~\ref{thm:circle-matedCRT-0-cone} where we extend the circle packing $\cP^{h ,n}_{k,z_*}$ from vertices in $\wh\cG^{h,1/n}_{k,z_*}$ to all vertices in $\cG^{h,1/n}_{k,z_*}$ and use the fact that all cells in $\wh \cH^{h,1/n}_{k,z_*}$ and all disks in $\wt\cP^n$ are microscopic as in the proof of Theorem~\ref{thm:circle-matedCRT-0-cone}.
\end{proof}

To prove Proposition~\ref{thm:circle-quantum-cell-0}, to simplify the notations, we will mainly work with $t=0$ and $a=-1,b=1$. For this purpose, we shall assume $\beta_n=0$. The case where $t=0$ and $a,b$ are general follows by the same argument and scaling, while the general $t$ case follows from the translation invariance of $\gamma$-quantum cone as in~\cite[Lemma 8.3]{DMS14} and further setting $\beta_n = [-nt]$. 

\begin{proof}[Proof of Proposition~\ref{thm:circle-quantum-cell-0}]
 Let $(h,\eta,z_*)\sim \bbP_*$. Let $h_*(\cdot)=h(\cdot+z_*)$ and $\eta_*(\cdot)=\eta(s+\tau_*)-z_*$. Then following~\cite[Lemma A.10]{DMS14}, one can show that $h_*|_{B(0;1/2)}$ has the same law as $h^\rmQ|_{B(0;1/2)}$ under $\bbP_\rmQ$, and $\eta_*$ is an independent space-filling $\SLE_\kappa$ curve parameterized by $\mu_{h_*}$. Let $E_k^* = \{\eta([\tau_*-1/k,\tau_*+1/k])\}\subseteq B(z_*;1/4)$ and $E_k^\rmQ = \{\eta([-1/k,1/k])\subseteq B(0;1/4)\}$. Therefore by Lemma~\ref{lem:circle-pack-cell-1}, \eqb\label{eq:pf:thm:circle-quantum-cell-0-122}\bbP_\rmQ[E_k^\rmQ\cap \big(F_{k,0,\e}^{h^\rmQ,n,0}\big)^c] = \bbP_*[E_k^*\cap \big(F_{k,z_*,\e}^{h,n,[n\tau_*]}\big)^c]\to 0 \ \ \text{as }n\to\infty.\eqe
  Now  $\bbP_\rmQ[\eta([-\frac{1}{k},\frac{1}{k}])\subseteq B(0;1/4) ]\to 1$ as $k\to\infty$, and therefore as we replace $n$ with $kn$ in~\eqref{eq:pf:thm:circle-quantum-cell-0-122} and further send $k\to\infty$, the scale invariance in Lemma~\ref{lem:Rb} readily implies that $\bbP_\rmQ[F^{h^\rmQ,n,0}_{1,0,\e}]\to 1$ as $n\to\infty$. This finishes the proof.
\end{proof}

\subsubsection{Convergence of circle packing for mated CRT map on the disk}\label{subsubsec:matedCRT-disk}

Here we prove Theorem~\ref{thm:pack-disk} based on Proposition~\ref{thm:circle-quantum-cell-0}, Theorem~\ref{thm:He-Schramm-disk} and local absolute continuity. Let $(\bbD,h^\disk,1)$ be a quantum disk of unit area and boundary length, and $\eta^\disk:[0,1]\to\ol\bbD$ be a  counterclockwise space-filling  $\SLE_\kappa$ loop such that $\eta^\disk(0)=\eta^\disk(1)=1$ and parameterized via $\gamma$-LQG area with respect to $h^\disk$. We choose the embedding of $h^\disk$ such that $\eta^\disk(1/2)=0$. We continue to write $(h^\rmQ,\eta^\rmQ)$ for the space-filling $\SLE_\kappa$ decorated quantum cone sampled from $\bbP_\rmQ$.

\begin{lemma}\label{lem:QD-AC}
    Fix $\e\in(0,0.01)$. Then the law of the curve decorated quantum surface $(\eta^{\mathrm{disk}}([\e,1-\e]),h^\disk,\eta^{\mathrm{disk}}|_{[\e,1-\e]})/\sim_\gamma$ is absolutely continuous with respect to that of $(\eta^\rmQ([\e,1-\e]),h^\rmQ, \eta^\rmQ|_{[\e,1-\e]})/\sim_\gamma$.
\end{lemma}

\begin{proof}
    Let  $(L_t,R_t)_{t\in\bbR}$ be the boundary length process for $(h^\rmQ,\eta^\rmQ)$. As in~\cite[Lemma 2.17]{AG19} and in the last paragraph of the proof of~\cite[Theorem 1.1]{AG19}, there exists some function $F$ such that $(\eta^\rmQ([\e,1-\e]),h^\rmQ,\eta^\rmQ|_{[\e,1-\e]})/\sim_\gamma = F((L ,R )|_{[\e,1-\e]})$ and $(\eta^{\mathrm{disk}}([\e,1-\e]),h^\disk,\eta^{\mathrm{disk}}|_{[\e,1-\e]})/\sim_\gamma = F((L^{\disk},R^\disk)|_{[\e,1-\e]})$. The claim then follows from the absolute continuity between the law of the processes $(L^{\disk},R^\disk)|_{[\e,1-\e]}$ and $(L ,R )|_{[\e,1-\e]}$ for fixed $\e$.
\end{proof}

Let $\cH^{\disk,1/n} = \{\eta^\disk([x-\frac1n,x]):x\in(0,1]\cap\frac1n\bbZ\}$, and let $\wh\cH^{\disk,1/n}$ be the version of  $\cH^{\disk,1/n}$ with multiple edges collapsed, i.e., for $x_1<x_2-\frac1n$, if $\eta^\disk([x_1-\frac1n,x_1])$ and $\eta^\disk([x_2-\frac1n,x_2])$ intersect along both of their left and right boundaries over a non-trivial boundary arc, then we merge $\eta^\disk([x_1,x_2-\frac1n])$ into $\eta^\disk([x_2-\frac1n,x_2])$. Recall the notations $\frt_{[a,b],t}^{\theta,n,-}$, $\frt_{[a,b],t}^{\theta,n,+}$ and $\Phi_{[a,b],t}^{\theta,n}$ in Section~\ref{subsubsec:matedCRT-cell}.
For $\e\in(0,0.01)$, let $\cH^{\disk,1/n}_\e = \Phi_{[\e,1-\e],1/2}^{0,n}(h^\disk,\eta^\disk)\subseteq \cH^{\disk,1/n}$ and $\cG^{\disk,1/n}_\e$ be the map associated with $\cH^{\disk,1/n}_\e$. Let $s_\e^{n,-} = \frt_{[\e,1-\e],1/2}^{0,n,-}(h^\disk,\eta^\disk)$, $s_\e^{n,+} = \frt_{[\e,1-\e],1/2}^{0,n,+}(h^\disk,\eta^\disk)$.
Let $\wh \cG^{\disk,1/n}_\e = \Phi_{\rm simple}(\cG^{\disk,1/n}_\e)$ and $\wh \cG^{\disk,1/n}=\Phi_{\rm simple}(\cG^{\disk,1/n})$. Let $D_\e^\disk$ be the   connected component of the interior of $\eta^\disk([\e,1-\e])$ containing 0.

\begin{lemma}\label{lem:QD-AC-1}
We have the following.
\begin{enumerate}[(i)]
    \item Under the event where neither $\eta^\disk([s^{n,-}_\e+\frac{1}{n},1/2])$ nor $\eta^\disk([1/2,s_\e^{n,+}-\frac{1}{n}])$ is contained in a single cell in $\wh\cH^{\disk,1/n}$, $\wh \cG^{\disk,1/n}_\e$ is a submap of $\wh \cG^{\disk,1/n}$, where edges of  $\wh \cG^{\disk,1/n}$ with both endpoints in $\wh \cG^{\disk,1/n}_\e$ are also edges of $\wh \cG^{\disk,1/n}_\e$.
    \item  As we send $\e\to 0$ sufficiently slower than $n\to\infty$, $(s_\e^{n,-},s_\e^{n,+})$ converges to $(0,1)$ in probability. 
    \item  The maximal diameter of cells in $\wh\cH^{\disk,1/n}$ tends to 0 in probability as $n\to\infty$. In particular, with probability  tending to 1 as we send $\e\to 0$ sufficiently slower than $n\to\infty$, $\wh \cG^{\disk,1/n}_\e$ is a submap of $\wh \cG^{\disk,1/n}$.
\end{enumerate}
\end{lemma}

\begin{proof}
     The first   claim follows from the identical proof of part (i) of Lemma~\ref{lem:circle-pack-cell}. For the second claim, suppose $\ol D_\e^\disk=\eta^\disk([t^-_\e,t_\e^+])$, using the same argument as in the proof of part (ii) of Lemma~\ref{lem:circle-pack-cell}, as we send $n\to\infty$, $(s_\e^{n,-},s_\e^{n,+})$ converges in probability to $(t^-_\e,t_\e^+)$. Claim (ii) then follows since again by continuity of $\eta^\disk$, $(t^-_\e,t_\e^+)$ converges to $(0,1)$ in probability as we further send $\e\to0$.

    To prove claim (iii), let $$\Delta_n = \max\{|x_1-x_2|: x_1,x_2\in (0,1]\cap \frac{1}{n}\bbZ: x_1,x_2 \text{ form a double edge in } \cG^{\disk,1/n}\}.$$ We prove that $\Delta_n\to 0$ in probability as $n\to\infty$. Let $\delta_1>0$, and we choose $\e_1>0$ sufficiently small such that with probability at least $1-\delta_1/4$, $R_t^\disk>0.9$ on $[0,\e_1]$ while $R_t^\disk<0.1$ on $[1-\e_1,1]$. Under this event, for $n>10\e_1^{-1}$, any $x_1,x_2\notin [\e_1/2,1-\e_1/2]$ forming double edges must satisfy $|x_1-x_2|\leq \e_1$. Next pick $\e_2<\e_1/8$ sufficiently small, such that with probability at least $1-\delta_1/4$, $L_t^\disk<L_s^\disk$ for any $t\in [0,\e_2]\cup[1-\e_2,1]$ and $s\in [\e_1/4,1-\e_1/4]$. Under this event, for $n>10^4\e_2^{-1}$, there are no $x_1,x_2$ such that one of them is in $[\e_1/2,1-\e_1/2]$ and the other is not in $[\e_2/2,1-\e_2/2]$ and they form a double edge. Finally, note that in the setting of $\wh\cH^{\rmQ,1/n}$ and $(h^\rmQ,\eta)\sim\bbP_\rmQ$, if $x_1,x_2$ form a double edge and $\eta([x_2-\frac1n,x_2])\subseteq \wh H$ with $H\in \wh\cH^{\rmQ,1/n}$, then $|x_2-x_1|-\frac{1}{n}\leq \mu_h(\wh H)$. Therefore
    given $\e_1,\e_2$,  using the absolute continuity from Lemma~\ref{lem:QD-AC} and Lemma~\ref{lem:diam-cell-max-1}, the maximal possible gap $|x_1-x_2|$ over all choices $x_1,x_2\in [\e_2/2,1-\e_2/2]$ which forms double edges tends to 0 in probability. This proves $\Delta_n\to 0$ in probability, which further implies claim (iii) by continuity of $\eta^\disk$.
\end{proof}

\begin{proof}[Proof of Theorem~\ref{thm:pack-disk}]
  We work in the setting of the proof of Lemma~\ref{lem:QD-AC-1}. Let $\psi_\e^\disk:D_\e^\disk\to\bbD$ be a conformal map fixing 0 such that $(\psi_\e^\disk)'(0)>0$. Note that by the continuity of $\eta^\disk$ near its endpoints, $\diam(\bbD\backslash\eta^\disk([\e,1-\e]))\to 0$ almost surely and $\bbD\backslash B(1;2\diam(\bbD\backslash\eta^\disk([\e,1-\e])))\subseteq D_\e^\disk$. Proposition~\ref{prop:boundary-perturb} then gives that $\sup_{z\in\ol{D_\e^\disk}}|\psi_\e^\disk(z)-z|\to 0$ almost surely as $\e\to 0$.  By Proposition~\ref{thm:circle-quantum-cell-0} and the absolute continuity from Lemma~\ref{lem:QD-AC}, there is a circle packing $\cP^{\disk,n}_\e$ for $\cG^{\disk,1/n}_\e$ in $\bbD$ such that  as we send $\e\to0$ sufficiently slower than $n\to\infty$,   the maximal diameter of disks in $\cP^{\disk}_\e$ tends to 0 in probability, and up to a rotation, 
  \eqb\label{eq:pf-thm:pack-disk-1}\max\{|\psi_\e^\disk(\eta^\disk(x)) - o_n(x)|: x\in (s_\e^{n,-},s_\e^{n,+})\cap \frac{1}{n}\bbZ\}\eqe
  tends to 0 in probability. We may further center this packing at the vertex in $\wh\cG^{\disk,1/n}_\e$ with closest distance to 1/2.  On the other hand, let $y_1^{n,\e}$ (resp.\ $y_2^{n,\e}$)   be the last time $s$ before $s_\e^{n,+}$ (resp.\ first time $s$ after $s_\e^{n,-}$)  in $  \frac{1}{n}\bbZ\cap(0,1]$ such that the cell containing $\eta^\disk([s-\frac1n,s])$  shares a nontrivial boundary arc with the connect component of $\partial D_\e^\disk\cap\partial\bbD$ containing $-1$. Note that if a cell $\eta^\disk([s-\frac1n,s])$ intersects $\partial\bbD$ along a nontrivial boundary arc, then $\eta^\disk([s-\frac1n,s])$ cannot be enclosed by any two other cells  $\eta^\disk([s'-\frac1n,s'])$ and   $\eta^\disk([s''-\frac1n,s''])$ as in Lemma~\ref{lem:double-edge-topology} and the vertex $s\in\cV\cG^{\disk,1/n}$ will survive as we collapse all the double edges. Now let  $\ell^{n,\e}$ (resp.\ $\gamma^{n,\e}$) be the clockwise (resp.\ counterclockwise) arc of $\partial\wh\cG^{\disk,1/n}_\e$ from $y_1^{n,\e}$ to $y_2^{n,\e}$. Then the cells in $\wh\cH^{\disk,1/n}$ for vertices on $\gamma^{n,\e}\backslash\{y_1^{n,\e},y_2^{n,\e}\}$ intersect $\bbD\backslash D_\e^\disk$, and it further follows from the no-macroscopic-cell result from Lemma~\ref{lem:QD-AC-1} and~\eqref{eq:pf-thm:pack-disk-1} that the diameter of the union of the disks in $\cP^{\disk,n}_\e$  for vertices on $\gamma^{n,\e}$ tends to 0 in probability as $n\to\infty$ as $\e\to0$ sufficiently slower than $n\to\infty$.
  Consider a circle packing $\cP^{\disk,n}$ of $\wh\cG^{\disk,1/n}$ in $\bbD$ centered at the vertex closest to $1/2$ and extend it to all vertices of $\cG^{\disk,1/n}$ as in Definition~\ref{def:pack-non-simple}. Since with high probability $\wh \cG^{\disk,1/n}_\e$ is a submap of $\wh \cG^{\disk,1/n}$, we can use Theorem~\ref{thm:He-Schramm-disk} to compare the circle packing $\cP^{\disk,n}_\e$ with the circle packing $\wh\cP^{\disk,n}$, and from Theorem~\ref{thm:He-Schramm-disk}, the maximal diameter of disks in $\wh\cP^{\disk,n}$ tends to 0, and up to a rotation, any homeomorphism between the carrier of $\cP^{\disk,n}_\e$ and the carrier of the restriction $\cP^{\disk,n}$ to $\wh \cG^{\disk,1/n}_\e$ that corresponds the disks is approximately the identity. 
  Assertion 1.\ and 2.\ then follow by~\eqref{eq:pf-thm:pack-disk-1} and Theorem~\ref{thm:He-Schramm-disk} where we send $\e\to 0$ sufficiently slower than $n\to\infty$, and further extend to all vertices in $\cG^{\disk,1/n}$ using the no-macroscopic-cell result from Lemma~\ref{lem:QD-AC-1}, no-macroscopic-disks from Theorem~\ref{thm:He-Schramm-disk} as well as continuity of $\eta^\disk$ as in the proof of Theorem~\ref{thm:circle-matedCRT-0-cone}.
  Assertion 4.\ is a direct consequence of Assertion 2.\, and Assertion 3.\ follows from Assertion 1.\ and 2.\ and the fact that $\eta^\disk$ is parameterized by $\gamma$-LQG mass with respect to $h^\disk$.
\end{proof}

\subsubsection{Convergence of circle packing for mated CRT map on the sphere}\label{subsubsec:matedCRT-sph}
We now prove Theorem~\ref{thm:pack-sph} following the same strategy as the proof of Theorem~\ref{thm:pack-disk}, where we use local absolute continuity along with Theorem~\ref{thm:He-Schramm-sphere} and Proposition~\ref{thm:circle-quantum-cell-0}. 

Let $((\wh\bbC,h^\mathrm{sp}, 0,\infty),\eta^{\rm sp})$ be the unit area quantum sphere and the space-filling $\SLE_\kappa$ as in Theorem~\ref{thm:pack-sph}, $(L^\sph_t,R^\sph_t)|_{t\in[0,1]}$ be the left and right boundary length of $\eta^\sph([0,t])$.   Let $(h^\rmQ,\eta^\rmQ)$ be sampled from $\bbP_\rmQ$ as before. We start with the following local absolute continuity.

\begin{lemma}\label{lem:QS-AC}
    Let $\e\in(0,0.1)$. The law of the curve-decorated quantum surface $(\eta^{\rm sp}([\e,1-\e]),h^{\rm sp},\eta^{\rm sp}|_{[\e,1-\e]},0)/\sim_\gamma$  is absolutely continuous with respect to that of $(\eta^{\rmQ}([\e,1-\e]),h^{\rmQ},\eta^{\rm Q}|_{[\e,1-\e]})/\sim_\gamma$.
\end{lemma}

\begin{proof}
    The claim follows from the same argument as in Lemma~\ref{lem:QD-AC}, where we use the local absolute continuity between the Brownian paths as described in~\cite[Theorem 3.1]{MS-mot-finite}.
\end{proof}

Let $n>200$, and $\cH^{\sph,1/n} = \{\eta^\sph([x-\frac1n,x]):x\in (0,1]\cap\frac{1}{n}\bbZ\}$. Suppose $x_1,x_2\in (0,1]\cap\frac{1}{n}\bbZ$, $x_1<x_2$ such that there are two edges between $x_1$ and $x_2$ on $\cG^{\sph,1/n}$. We work on the event $E^{\sph,n}$ where for all such possible $(x_1,x_2)$, either $x_1<0.05$ and $x_2>0.95$ or $|x_1-x_2|<0.05$. In the former case, we merge $\eta^\sph([x_2,1])\cup \eta^{\sph}([0,x_1-\frac{1}{n}])$ with $\eta^\sph([x_1-\frac{1}{n},x_1])$, and in the latter case, we merge $\eta^\sph([x_1,x_2-\frac{1}{n}])$ into $\eta^\sph([x_2-\frac1n,x_2])$. Let $\wh\cH^{\sph,n}$ be the resulting collection of cells after we do this modification. Under the $E^{\sph,n}$, $\wh\cG^{\sph,1/n}=\Phi_{\rm simple}(\cG^{\sph,1/n})$ is not empty, and $\wh\cH^{\sph,n}$ can be viewed as an embedding of the simple sphere triangulation $\wh\cG^{\disk,1/n}$.
\begin{lemma}\label{lem:QS-AC-1}
As $n\to\infty$, $\bbP[E^{\sph,n}]\to1$, and the maximal spherical diameter of cells in    $\wh\cH^{\sph,n}$ tends to 0 in probability.
\end{lemma}

\begin{proof}
   Let $\e_1,\e_2>0$ be small. Let $F_{\e_1,\e_2}$ be the event that  {$\eta^\sph([\e_1/2,1-\e_1/2])$} is disjoint from $\eta^\sph([0,2\e_2])\cup\eta^\sph([1-2\e_2,1])$. By the continuity of $\eta^\sph$,  $\bbP[F_{\e_1,\e_2}]\to 1$ if we send $\e_1\to 0$ sufficiently slower than $\e_2\to 0$. Below we assume $n>40\max\{\e_1^{-1},\e_2^{-1}\}$.
    
    Suppose $x_1,x_2\in (0,1]\cap\frac{1}{n}\bbZ$, $x_1<x_2$ such that there are two edges between $x_1$ and $x_2$ on $\cG^{\sph,1/n}$. There are now three possibilities: (a) $x_1\in [\e_1,1-\e_1]$ and $x_2\notin[\e_2,1-\e_2]$, or $x_1\notin [\e_2,1-\e_2]$ and $x_2\in [\e_1,1-\e_1]$; (b) $x_1,x_2\in [\e_2,1-\e_2]$ and (c) $x_1,x_2\notin [\e_1,1-\e_1]$. Indeed under the event $F_{\e_1,\e_2}$, case (a) is not possible. For case (b), if we send $\e_2\to 0$ sufficiently slower than $n\to\infty$, using Lemma~\ref{lem:diam-cell-max-1} and the absolute continuity from Lemma~\ref{lem:QS-AC} with  {$\e=\e_2/2$}, we see that the maximal gap $|x_1-x_2|$ over all possible choices of $(x_1,x_2)$ in this case tends to 0 in probability and therefore the maximal spherical diameter of $\eta^\sph([x_1,x_2])$ tends to 0 by continuity of $\eta^\sph$. For case (c), either $|x_1-x_2|\leq \e_1$ or $x_1\leq \e_1$ and $x_2\geq 1-\e_1$. The lemma then follows by combining these cases and using the continuity of $\eta^\sph$. 
\end{proof}

For $\e\in(0,0.01)$, let $\cH^{\sph,1/n}_\e = \Phi_{[\e,1-\e],1/2}^{0,n}(h^\sph,\eta^\sph)\subseteq \cH^{\sph,1/n}$ and $\wh\cH^{\sph,1/n}_\e = \wh\Phi_{[\e,1-\e],1/2}^{0,n}(h^\sph,\eta^\sph)$. Let $\wt s_\e^{n,-} = \frt_{[\e,1-\e],1/2}^{0,n,-}(h^\sph,\eta^\sph)$ and  $\wt s_\e^{n,+} = \frt_{[\e,1-\e],1/2}^{0,n,+}(h^\sph,\eta^\sph)$.  Let $\cG^{\sph,1/n}_\e$ be the  map formed by the cells in $\cH^{\sph,1/n}_\e$, and let $\wh\cG^{\sph,1/n}_\e = \Phi_{\rm simple}(\cG^{\sph,1/n}_\e)$ be the map formed by the cells in $\wh\cH^{\sph,1/n}_\e$.
\begin{lemma}\label{lem:QS-AC-2}
    With probability tending to 1 as $n\to\infty$, $\wh\cG^{\sph,1/n}_\e$ is a submap of $\wh\cG^{\sph,1/n}$, where edges of $\wh\cG^{\sph,1/n}$ with both endpoints in $\wh\cG^{\sph,1/n}_\e$ are also edges of $\wh\cG^{\sph,1/n}_\e$. Furthermore, as we send $\e\to 0$ sufficiently slower than $n\to\infty$, $(\wt s_\e^{n,-},\wt s_\e^{n,+})$ converges to $(0,1)$ in probability.
\end{lemma}

\begin{proof}
    For the first claim, by Lemma~\ref{lem:QS-AC-1}, with probability tending to 1 as $n\to\infty$, $\eta^\sph([0,\e])$, $\eta^\sph([\e,1/2])$, $\eta^\sph([1/2,1-\e])$ and $\eta^\sph([1-\e,1])$ are all not contained in a single cell of $\wh\cH^{\sph,n}$, and in this case, following identical arguments as in the proof of part (i) of Lemma~\ref{lem:circle-pack-cell}, $\wh\cG^{\sph,1/n}_\e$ can be viewed as a submap of  $\wh\cG^{\sph,1/n}$. The second claim follows similarly by the continuity of $\eta^\sph$ as in the proof of Lemma~\ref{lem:circle-pack-cell} and Lemma~\ref{lem:QD-AC-1}. 
\end{proof}

\begin{proof}[Proof of Theorem~\ref{thm:pack-sph}]
    We continue with the setting of Lemma~\ref{lem:QS-AC-2}. Let $D_\e^\sph$ be the simply connected component of $\eta^\sph([\e,1-\e])$ containing 0, and $\psi_\e:D_\e^\sph\to\bbD$ be a conformal map fixing 0. Note that by continuity of $\eta^\sph$, the spherical diameter of $\bbC\backslash \eta^\sph([\e,1-\e])$ tends to 0 almost surely as $\e\to 0$ and therefore the spherical diameter of $\bbC\backslash D_\e^\sph$ also a.s.\  tends to 0.  By Proposition~\ref{thm:circle-quantum-cell-0} and the absolute continuity from Lemma~\ref{lem:QS-AC}, there is a circle packing $\cP^{\sph,n}_\e$ for $\wh\cG^{\sph,1/n}_\e$ in $\bbD$ centered at the vertex of closest distance to $1/2$, and as we send $\e\to0$ sufficiently slower than $n\to\infty$,   the maximal diameter of disks in $\cP^{\sph,n}_\e$ tends to 0 in probability, and up to a rotation, 
  \eqb\label{eq:pf-thm:pack-sph-1}\max\{|\psi_\e^\sph(\eta^\sph(x)) - o_n(x)|: x\in (\wt s_\e^{n,-},\wt s_\e^{n,+})\cap \frac{1}{n}\bbZ\}\eqe
  tends to 0 in probability. Furthermore, using the endpoint continuity of $\eta^\sph$, $\psi_\e^\sph(1)$ converges to 0 in probability as we send $\e\to0$. Now let $r_\e\in\bbC$ such that in the circle packing $r_\e \cP^{\sph,n}_\e$, the disk for the vertex $\wh x_{1/4}\in \frac{1}{n}\bbZ\cap(0,1]$ in the graph $\wh \cG_\e^{\sph,1/n}$ of closet distance to $1/4$ is centered at 1. Then we may further choose a sequence $\e(n)\to 0$ sufficiently slowly, such that as $n\to\infty$ and $\e = \e(n)$, the maximal spherical diameter of disks as well as bounded interstices in $r_\e \cP^{\sph,n}_\e$ tends to 0 in probability,   $|r_\e| \to\infty$ in probability, and the quantity in~\eqref{eq:pf-thm:pack-sph-1} multiplied with $|r_\e|$ still converges to 0 in probability. Moreover, by our choice of $r_\e$ and~\eqref{eq:pf-thm:pack-sph-1}, by a possible small perturbation and Proposition~\ref{prop:unif-sph}, the map $r_\e\psi_\e^\sph$ converges uniformly to the identity in spherical topology.  Let $\cP^{\sph,n}$ be a circle packing for $\wh\cG^{\sph,1/n}$ in $\wh\bbC$, such that the disks for the vertices in $\wh\cG^{\sph,1/n}$ of closest distance to $(1,1/2,1/4)$ contain $(\infty,0,1)$ in their interiors. Note that by Lemma~\ref{lem:QS-AC-1}, with probability tending to 1 as $n\to\infty$, $\cP^{\sph,n}$ is well-defined. By Lemma~\ref{lem:QS-AC-2}, $\wh\cG^{\sph,1/n}_\e$ can be viewed as a submap of  $\wh\cG^{\sph,1/n}$ with high probability, and therefore we can compare the circle packing $r_\e \cP^{\sph,n}_\e$ with $\cP^{\sph,n}$. 
  Now we compare $r_\e\cP_\e^{\sph,n}$ with the restriction of $\cP^{\sph,n}$ to vertices in $\wh\cG^{\sph,1/n}_\e$, where the support of the latter is within $\bbC$ with probability tending to 1 as $n\to\infty$ by Lemma~\ref{lem:QS-AC-1}.  {We may further choose a compatible homeomorphism $f^n$ between the two carriers that corresponds the disks  such that $f^n(0)=0$ and $f^n(1)=1$.}  By Theorem~\ref{thm:He-Schramm-sphere}, the maximal spherical diameter of disks in $\cP^{\sph,n}$ tends to 0 in probability,  any homeomorphism sending disks in $r_\e \cP^{\sph,n}_\e$ to the corresponding disk in  $\cP^{\sph,n}$ is approximately identity, and the disks in $\cP^{\sph,n}$ for vertices not in $\cG_\e^{\sph,1/n}$ are contained within an $o_\e(1)$-spherical neighborhood of $\infty$.  Assertion 1.\ and 2.\ then follow by applying~\eqref{eq:pf-thm:pack-sph-1} and Theorem~\ref{thm:He-Schramm-sphere} where we send $\e\to 0$ sufficiently slower than $n\to\infty$ and further extend to all all the vertices in $\cG^{\sph,1/n}$ via continuity of $\eta^\sph$ and the no macroscopic disks or cells result as in Theorem~\ref{thm:He-Schramm-sphere} and Lemma~\ref{lem:QS-AC-1}. Assertion 4.\ is a direct consequence of Assertion 2.\, and Assertion 3.\ follows from Assertion 1.\ and 2.\ and the fact that $\eta^\sph$ is parameterized by $\gamma$-LQG mass with respect to $h^\sph$.
\end{proof}

\subsubsection{Convergence of circle packing for mated CRT map on $\gamma$-quantum cone}\label{subsubsec:matedCRT-cone}

Now we prove Theorem~\ref{thm:pack-cone} based on Theorem~\ref{thm:He-Schramm-sphere}, Proposition~\ref{thm:circle-quantum-cell-0} and the scale invariance of the quantum cone in Lemma~\ref{lem:Rb}.

\begin{lemma}\label{lem:parabolicMOT}
    For $n\geq1$, the map $\wh\cG^{1/n} = \Phi_{\rm simple}(\cG^{1/n})$ is a.s.\  not empty and is a.s.\ circle packing parabolic.
\end{lemma}
\begin{proof}
The claim that $\wh\cG^{1/n}$ is non-empty is a consequence of Proposition~\ref{prop:cell-diam2/area} for $\bbP_\rmQ$ conditioned on $\theta=0$.  By~\cite[Theorem 1.4]{GMS19har}, the simple random walk on $\cH^{\rmQ,n}$ is a.s.\ recurrent. Since the associated map for $\wh\cH^{\rmQ,n}$ can be viewed as a subgraph of the associated map for $\cH^{\rmQ,n}$, we see that the simple random walk on $\wh\cH^{\rmQ,n}$ is also a.s.\ recurrent. The claim then follows from~\cite[Theorem 1.1]{HeSchramm-hy-pa}.
\end{proof}

\begin{proof}[Proof of Theorem~\ref{thm:pack-cone}]
   The claim that $\cG^{1/n}$ is circle packing parabolic is a consequence of Lemma~\ref{lem:parabolicMOT}. Let $(h,\eta)$ be the space-filling $\SLE_\kappa$ decorated $\gamma$-quantum cone as in the statement. Let $j\in\bbN$ and let $(\fR_j,h^j,\eta^j)$ be defined in~\eqref{eq:Rb} and~\eqref{eq:hb-etab} with $\alpha=\gamma$ and $b=j$. Fix $a>2$. Let $D_j$ be the  connected component of the interior of $\eta^j([-1,1])$ containing 0, and $\psi_j:D_j\to \bbD$ be a conformal map fixing 0  {with $\psi_j'(0)>0$}. Using $\eta^j\overset{d}{=}\eta$ for each $j$, we see that \eqb\label{eq:pf:thm:pack-cone-0}\psi_j(\eta^j(x)) = (\psi_j'(0)+\delta^j_x)\eta^j(x),\eqe where  {$\sup_{x\in[-a/j,a/j]}|\delta^j_x|\to 0$} in probability as $j\to\infty$.  

    Assume that $\ol D_j=\eta^j([t^{j,-},t^{j,+}])$.  Set $s^{n,j,-} = \frt_{[-1,1],0}^{0,jn,-}(h^j,\eta^j)$ and $s^{n,j,+} = \frt_{[-1,1],0}^{0,jn,+}(h^j,\eta^j)$. By Lemma~\ref{lem:circle-pack-cell} and the fact that $(h^j,\eta^j)\overset{d}{=}(h,\eta)$, the distance between $(s^{n,j,-},s^{n,j,+})$ and $(t^{j,-},t^{j,+})$ converges to 0 in probability as $n\to\infty$ with $j$ fixed. Let $\cG_j^{1/n}$ be the submap of $\cG^{1/n}$ consisted of vertices in $(js^{n,j,-},js^{n,j,+}]\cap\frac{1}{n}\bbZ$, which is isomorphic to the map $\cG_{[a,b],t}^{\rmQ,1/n,0}$ under   $x\mapsto x/j$ in Proposition~\ref{thm:circle-quantum-cell-0} defined for $(h^j,\eta^j)$ and $(n,[a,b],t)$ there replaced with $(jn,[-1,1],0)$. By Proposition~\ref{thm:circle-quantum-cell-0}, there exists a circle packing $\cP_j^{n}$ for $\cG_j^{1/n}$ in $\bbD$, such that the maximal diameter of disks in $\cP_j^{n}$ tends to 0 in probability, and up to a rotation, 
    \begin{equation}\label{eq:pf:thm:pack-cone-1}
         \max_{x\in \big(js^{n,j,-},js^{n,j,+}\big]\cap \frac{1}{n}\bbZ }\dist(\psi_{j}(\eta^j(\frac{x}{j})), o_n(x))\to 0 \ \ \text{in probability.}
     \end{equation} 
    
   Let $x_*^{0,n},x_*^{1,n}\in \frac{1}{n}\bbZ$ with closest distance to 0 and 1 such that the vertex $x_*^n$ remains in $\Phi_{\rm simple}(\cG^{1/n})$ as we collapse all the double edges, respectively. Lemma~\ref{lem:diam-cell-max-1} implies that $(x_*^{0,n},x_*^{1,n})\to (0,1)$ in probability as $n\to\infty$.  Now we consider the circle packing $ \psi_j'(0)^{-1}\fR_j\cP_j^{n}$ for $ \cG_j^{1/n}$,  and a circle packing $\cP^n$ for $ \cG^{1/n}$ such that the disks for the vertices $x_*^{0,n}$ and $x_*^{1,n}$ are centered at 0 and $\eta(1)$.  As we send $j\to\infty$ sufficiently slower than $n\to\infty$, by~\eqref{eq:pf:thm:pack-cone-0},~\eqref{eq:pf:thm:pack-cone-1} and $\eta^j(x)=\fR_j^{-1}\eta(jx)$, we have the following:
     \begin{enumerate}[(i)]
         \item The maximal diameter of disks in $\psi_j'(0)^{-1}\fR_j\cP_j^{n}$ tends to 0 in probability, and $ \psi_j'(0)^{-1}\fR_j\to\infty$ as $j\to\infty$;
         \item With probability tending to 1, $[-a,a]\subseteq (js^{n,j,-},js^{n,j,+}]$. Furthermore, in the circle packing $\psi_j'(0)^{-1}\fR_j\cP_j^{n}$ for $\cG_j^{1/n}$,
         \begin{equation}\label{eq:pf:thm:pack-cone-2}
         \max_{x\in [-a,a]\cap \frac{1}{n}\bbZ }\dist(\eta(x), o_n(x))\to 0 \ \ \text{in probability.}
     \end{equation} 
     \item The distance between the disk for the vertex 0 and $0$ as well as the distance between the disk for the vertex 1 and $\eta(1)$ tend to 0 in probability.
     \end{enumerate}
     Let $\alpha_{n,j}$ and $\beta_{n,j}$ be such that in the circle packing $\alpha_{n,j}  \psi_j'(0)^{-1}\fR_j\cP_j^{n}+\beta_{n,j}$, the disks for the vertices $(x_*^{0,n}, x_*^{1,n})$ are centered at 0 and $\eta(1)$. Then by property (iii) above, $(\alpha_{n,j},\beta_{n,j})\to (1,0)$ in probability as we send $j\to\infty$ sufficiently slower than $n\to\infty$. Therefore if we let $j = j(n)$ be depend on $n$ such that $j\to\infty$ sufficiently slow, and further multiply both embeddings by $\eta(1)^{-1}$, the conditions in Theorem~\ref{thm:He-Schramm-sphere} hold for $\alpha_{n,j}  \psi_j'(0)^{-1}\fR_j\cP_j^{n}+\beta_{n,j}$ and $\cP^{n}$. Therefore by Theorem~\ref{thm:He-Schramm-sphere}, the circle packing map between $\psi_j'(0)^{-1}\fR_j\cP_j^{n}$ and $\cP^{1/n}$ (restricted to vertices in $\Phi_{\rm simple}(\cG_j^{1/n})$) converges uniformly on compacts, and for each $r>0$, the maximal diameter of disks in $\cP^{n}$ intersecting $B(0;r)$ tends to 0 in probability. Assertions 1.\ , 2.\ and 4.\ then follow by~\eqref{eq:pf:thm:pack-cone-2} by further varying $a$. By Assertions 1.\ and 2., if $B(0;2r)\subseteq\eta([-M,M])$, then the disks for the vertices $|x|\leq M$ whose cell $\eta([x-\frac1n,x])$ intersects $\partial\eta([-M,M])$ are disjoint from $B(0;3r/2)$ with probability tending to 1 as $n\to\infty$. This implies that with probability tending to 1 as $n\to\infty$ then $M\to\infty$,  the support of the restriction of $\cP^n$ to vertices in $\frac{1}{n}\bbZ\cap [-M,M]$ contains $B(0;r)$. Further note that for each deterministic $r>0$, almost surely $\mu_h(\partial B(0;r))=0$.
     Assertion 3.\ is then a consequence of Assertion 2.\ since $\eta$ is parameterized by $\gamma$-LQG mass with respect to $h$.
\end{proof}

\subsection{Convergence of  mated-CRT map under Riemann uniformization embedding}\label{subsec:unif-matedCRT-gamma-cone}

In this section we prove the convergence of the mated-CRT map under the Riemann uniformization embedding, namely Theorem~\ref{thm:unif-cone}-\ref{thm:unif-sph}. Most of the ideas are parallel to those in Section~\ref{subsec:matedCRT-gamma-cone}. The proof will be easier than the circle packing case as we no longer need to collapse the double edges, and we will only list the key steps here. We will use the convention that if $T$ is a submap of $T'$, then we view $M(T)\subseteq M(T')$.

\subsubsection{Convergence of mated-CRT cells under Riemann uniformization embedding}\label{subsubsec:matedCRT-cell-1}

Recall the  operators $\Phi_{[a,b],t}^{\theta,n}$, $\frt_{[a,b],t}^{\theta,n,-}$ and $\frt_{[a,b],t}^{\theta,n,+}$ defined in Section~\ref{subsubsec:matedCRT-cell}. We start with the following analog of Proposition~\ref{thm:circle-quantum-cell-0}.

\begin{proposition}\label{thm:circle-quantum-cell-0-unif}
    Let $a<t<b$ be deterministic, $(h^\rmQ,\eta),\psi_{[a,b],t},\cG^{\rmQ,1/n,0}_{[a,b],t},t_a^{n,-}$ and $t_b^{n,+}$ be the same as in Proposition~\ref{thm:circle-quantum-cell-0}, and let $M(\cG^{\rmQ,1/n,0}_{[a,b],t})$ be the Riemann surface associated with $\cG^{\rmQ,1/n,0}_{[a,b],t}$. There exists a conformal map $\varphi^{\rmQ,1/n,0}_{[a,b],t}:M(\cG^{\rmQ,1/n,0}_{[a,b],t})\to\bbD$ such that the following is true. The maximal diameter of edges $\varphi^{\rmQ,1/n,0}_{[a,b],t}(\fre)$ tends to 0 in probability, and  up to a rotation, 
      \begin{equation}\label{eq:thm:unif-quantum-cell}
         \max_{x\in  (t_a^{n,-},t_b^{n,+})\cap \frac{1}{n}\bbZ }\dist(\psi_{[a,b],t}(\eta(x)), \varphi^{\rmQ,1/n,0}_{[a,b],t}(x))\to 0 \ \ \text{in probability.}
     \end{equation} 
\end{proposition}

To prove Proposition~\ref{thm:circle-quantum-cell-0-unif}, we first prove a  version of~\eqref{eq:thm:unif-quantum-cell} in the setting of 0-quantum cone, and then transfer to the setting of $\gamma$-quantum cone. We work in the setting of the paragraph before Lemma~\ref{lem:circle-pack-cell-1}. Let $G^{h,n,\theta}_{k,z,\e}$ be the following event: 
there exists a conformal map  $\varphi^{h,1/n}_{k,z}:M(\cG^{h,1/n}_{k,z})\to\bbD$, such that the maximal diameter of edges in $\varphi^{h,1/n}_{k,z}(\fre)$ is less than $\e$, and 
$$ \max_{x\in (t_k^{n,-},t_k^{n,+})\cap \frac{1}{n}(\theta+\bbZ)}\dist(\psi_{k,z}(\eta(x)), \varphi^{h,1/n}_{k,z}(x))<\e .$$

\begin{lemma}\label{lem:unif-cell-1}
Let $\bbP_*$ and $(h,\eta,z_*,\tau_*)$ be as in Lemma~\ref{lem:0-cone-phase-0}.  Under the probability measure $\bbP_*$, we have $\bbP_*[G^{h,n,[n\tau_*]}_{k,z_*,\e}]\to 1$ as $n\to\infty$.
\end{lemma}

\begin{proof}
By the identical argument as in the proof of Lemma~\ref{lem:0-cone-phase-0}, Theorem~\ref{thm:unif-matedCRT-0-cone} remains true if we replace $\theta$ with $[n\tau_*]$ in the sense of Lemma~\ref{lem:0-cone-phase-0}, and let
    $\varphi^{*,n}:M(\cG^{*,1/n})\to\bbC$ be the corresponding conformal map.  We view $M(\cG^{h,1/n}_{k,z})$ as a subset of $M(\cG^{*,1/n})$, and consider a conformal map $\varphi^{h,n}_{k,z}:M(\cG^{h,1/n}_{k,z})\to\bbD$ such that $\varphi^{h,n}_{k,z}(\tau_*)=0$. By a small linear perturbation with error $o_n(1)$, we may further assume that $\varphi^{*,n}$ embeds the vertex  $\tau_*$ at $z_*$.  Then the conformal map  $\varphi^{h,n}_{k,z}\circ (\varphi^{*,n}-z_*)^{-1}$ fixes 0. The claim then follows by applying Proposition~\ref{prop:boundary-perturb} to $\varphi^{h,n}_{k,z}\circ (\varphi^{*,n}-z_*)^{-1}$ and the domain $D_{k,z_*}-z_*$.
\end{proof}

\begin{proof}[Proof of Proposition~\ref{thm:circle-quantum-cell-0-unif}]
    The proof is identical to that of Proposition~\ref{thm:circle-quantum-cell-0} using Lemma~\ref{lem:unif-cell-1}.
\end{proof}

\subsubsection{Convergence of Riemann uniformization for mated CRT map on the disk}

\begin{proof}[Proof of Theorem~\ref{thm:unif-disk}]
    Let $\cG^{\disk,1/n}_\e$, $D_\e^\disk$, $s_\e^{n,-}$ and $s_\e^{n,+}$ be the same as in the paragraph before Lemma~\ref{lem:QD-AC-1}.  Let $\psi_\e^\disk:D_\e^\disk\to\bbD$ be a conformal map fixing 0. By Proposition~\ref{thm:circle-quantum-cell-0-unif} and the absolute continuity from Lemma~\ref{lem:QD-AC}, there is a conformal map $\varphi_\e^{\disk,n}:M(\cG^{\disk,1/n}_\e)\to\bbD$ such that as we send $\e\to0$ sufficiently slower than $n\to\infty$,   the maximal diameter of edges in $\varphi_\e^{\disk,n}(\fre)$ tends to 0 in probability, and up to a rotation, 
     \eqb\label{eq:pf-thm:unif-disk-1}\max\{|\psi_\e^\disk(\eta^\disk(x)) - \varphi_\e^{\disk,n}(x)|: x\in (s_\e^{n,-},s_\e^{n,+})\cap \frac{1}{n}\bbZ\}\eqe
     tends to 0 in probability. We may further normalize $(\psi_\e^\disk,\psi^\disk)$, such that the vertex with closest distance to $1/2$ is embedded at $\eta^\disk(1/2)$ and the convergence~\eqref{eq:pf-thm:unif-disk-1} continues to hold.  Now let $y_1^{n,\e}, y_2^{n,\e}\in \frac{1}{n}\bbZ\cap (0,1]$ be the vertices as in the proof of Theorem~\ref{thm:pack-disk}, and $\ell^{n,\e}$ (resp.\ $\gamma^{n,\e}$) be the clockwise (resp.\ counterclockwise) arc of  $\partial\cG^{\disk,1/n}_\e$ from $y_1^{n,\e}$ to $y_2^{n,\e}$. As we send $\e\to 0$ sufficiently slower than $n\to\infty$, the diameter of the boundary arc $\gamma^{n,\e}$ under the uniformization map $\varphi_\e^{\disk,n}$ tends to 0 in probability. 
     The assertions can therefore be proved via~\eqref{eq:pf-thm:unif-disk-1}, where we Proposition~\ref{prop:unif-disk} to the conformal map $\varphi^{\disk,n}\circ(\varphi_\e^{\disk,n})^{-1}$ (where we may possibly apply further a M\"{o}bius map sending $\eta^\disk(1/2)$ to 0 and rotate) and extend to all the vertices of $\cG^{\disk,1/n}$ by endpoint continuity of $\eta^\disk$.
\end{proof}

\subsubsection{Convergence of Riemann uniformization for mated CRT map on the sphere}

\begin{proof}[Proof of Theorem~\ref{thm:unif-sph}]
    We work with the setting of Lemma~\ref{lem:QS-AC-2}.  Let $D_\e^\sph$ be the simply connected component of $\eta^\sph([\e,1-\e])$ containing 0, and $\psi_\e:D_\e^\sph\to\bbD$ be a conformal map fixing 0. By Proposition~\ref{thm:circle-quantum-cell-0-unif} and the absolute continuity from Lemma~\ref{lem:QS-AC}, there is a conformal map $\varphi_\e^{\sph,n}: M(\cG^{\sph,1/n}_\e)\to\bbD$  such that the image of the vertex of closest distance to 1/2 is at 0, and as we send $\e\to0$ sufficiently slower than $n\to\infty$,   the maximal diameter of edges in $\varphi_\e^{\sph,n}(\fre)$ tends to 0 in probability, and up to a rotation, 
  \eqb\label{eq:pf-thm:unif-sph-1}\max\{|\psi_\e^\sph(\eta^\sph(x)) - \varphi_\e^{\sph,n}(x)|: x\in (\wt s_\e^{n,-},\wt s_\e^{n,+})\cap \frac{1}{n}\bbZ\}\eqe
  tends to 0 in probability.  Now let $r_\e\in\bbC$ such that the image of the vertex $x_{1/4}\in \frac{1}{n}\bbZ\cap(0,1]$ of closet distance to $1/4$ under the map $r_\e \varphi_\e^{\sph,n} $ is 1.  We may choose a sequence $\e(n)\to 0$ sufficiently slower than $n\to\infty$, such that as $n\to\infty$ and $\e=\e(n)$, the maximal spherical diameter of edges in $r_\e \varphi_\e^{\sph,n}(\fre)$ tends to 0 in probability,   $|r_\e| \to\infty$ in probability, and the quantity in~\eqref{eq:pf-thm:unif-sph-1} multiplied with $|r_\e|$ still converges to 0 in probability. Moreover, by our choice of $r_\e$ and~\eqref{eq:pf-thm:unif-sph-1}, by Proposition~\ref{prop:unif-sph}, the map $r_\e\psi_\e^\sph$ converges uniformly to the identity in spherical topology. Let $\varphi^{\sph,n}:M(\cG^{\sph,1/n})\to\wh\bbC$ be the conformal map sending the vertices of closest distance to $(1,1/2,1/4)$ to $(\infty,0,1)$. By Proposition~\ref{prop:unif-sph},  $\varphi^{\sph,n}\circ (\varphi_\e^{\sph,n})^{-1}(r_\e^{-1}\cdot)$ converges uniformly to the identity  in probability in spherical topology, and for any $R>0$, with probability tending to 1 as $n\to\infty$, $B(0;R)\subseteq \varphi^{\sph,n}(M(\cG^{\sph,1/n}_\e))$.   This along with~\eqref{eq:pf-thm:unif-sph-1} (with the version where we multiply by $|r_\e|$), where we send $\e\to0$ sufficiently slower than $n\to\infty$ and further extend to all the vertices of $\cG^{\sph,1/n}$ by endpoint continuity of $\eta^\sph$, give Assertions 1.\ and 2.. Assertion 3.\ and 4.\ follow from Assertions 1.\ and 2.\ .  
\end{proof}

\subsubsection{Convergence of Riemann uniformization for mated CRT map on the $\gamma$-quantum cone}


\begin{proof}[Proof of Theorem~\ref{thm:unif-cone}]
We consider the same setup as in the proof of Theorem~\ref{thm:pack-cone}. By Proposition~\ref{thm:circle-quantum-cell-0-unif} applied to $(h^j,\eta^j)$ with $jn$ in place of $n$ and using $(h^j,\eta^j)\overset{d}= (h,\eta)$, there exists a conformal map $\varphi_j^{n}:M(\cG_j^{1/n})\to \bbD$, such that  as  $n\to\infty \text{ with $j$ fixed or $j\to\infty$ with $n$ fixed,}$ the maximal diameter of edges $\varphi_j^{n}(\fre)$ tends to 0 in probability, and up to a rotation, 
\begin{equation}\label{eq:pf:thm:unif-cone-1}
         \max_{x\in \big(js^{n,j,-},js^{n,j,+}\big]\cap \frac{1}{n}\bbZ }\dist(\psi_{j}(\eta^j(\frac{x}{j})),  \varphi_j^{n}(x))\to 0 \ \ \text{in probability.}
     \end{equation} 
    By a M\"{o}bius map, we may further assume that $ \varphi_j^{n}$ embeds the vertex 0 at 0 while the above convergence continues to hold.

We first prove that $M(\cG^{1/n})$ is a.s.\ parabolic.  By the scaling in Lemma~\ref{lem:Rb}, it suffices to prove the parabolicity for $n=1$. We temporarily fix $n=1$ and send $j\to\infty$ first. Since the distance between $(s^{1,j,-},s^{1,j,+})$ and $(t^{j,-},t^{j,+})$ converges to 0 in probability, by $\eta^j\overset{d}= \eta$ we see that $(js^{n,j,-},js^{n,j,+})\to(-\infty,+\infty)$ in probability.  Assume on the contrary that the event $E$ that there exists some conformal map $\varphi:M(\cG^1)\to\bbD$ has positive probability $p_0>0$. Assume that $\varphi$ embeds the vertex 0 at 0. We condition on $E$. Since $\varphi^{-1}(\ol{B(0;1/2)})$ is compact  and $(js^{1,j,-},js^{1,j,+})\to(-\infty,+\infty)$ in probability, by~\eqref{eq:pf:thm:unif-cone-1} we see that $\diam(\varphi_j^{1}\circ \varphi^{-1}(\ol{B(0;1/2)}) )\to 0$   in probability. One can see this by first applying ~\eqref{eq:pf:thm:unif-cone-1} to control the $\varphi_j^{1}$-images of the vertices and edges of the triangles on $M(\cG^{1/n})$ whose $\varphi$-image meets $\ol{B(0;1/2)}$, and further applying the maximum principle for the interior.  This implies that conditioned on $E$, $(\varphi_j^{1}\circ \varphi^{-1})'(0)\to 0$ and $(\varphi\circ (\varphi^1_j)^{-1})'(0)\to\infty$ in probability. On the other hand, under the event $E$, $\varphi\circ (\varphi_j^1)^{-1}$ is a conformal map  from $\bbD$ to a subset of $\bbD$. This contradicts with the Koebe quarter theorem. This proves that  $M(\cG^{1/n})$ is a.s.\ parabolic.

 Below we will send $j\to\infty$ sufficiently slower than $n\to\infty$.   Consider the conformal map $\phi_{n,j}:=\psi_j'(0)^{-1}\fR_j\varphi_j^{n}$, and let $\varphi^n:M(\cG^{1/n})\to \bbC$ be a conformal map, and $\varphi^n(0)=0,\varphi^n(1)=\eta(1)$.  
 By~\eqref{eq:pf:thm:pack-cone-0},~\eqref{eq:pf:thm:unif-cone-1} and $\eta^j(x)=\fR_j^{-1}\eta(jx)$, if we let $j$ depend on $n$ and tends to $\infty$ sufficiently slow, then we have the following:
     \begin{enumerate}[(i)]
         \item The maximal diameter of edges in $\phi_{n,j}(\fre)$ tend to 0 in probability, and $ \psi_j'(0)^{-1}\fR_j\to\infty$ as $j\to\infty$;
         \item With probability tending to 1, $[-a,a]\subseteq (js^{n,j,-},js^{n,j,+}]$. Furthermore,  
         \begin{equation}\label{eq:pf:thm:unif-cone-2}
         \max_{x\in [-a,a]\cap \frac{1}{n}\bbZ }\dist(\eta(x), \phi_{n,j}(x))\to 0 \ \ \text{in probability.}
     \end{equation} 
     \item The distance $| \phi_{n,j}(0)-0|$ and $| \phi_{n,j}(1)-\eta(1)|$ tend to 0 in probability.
     \end{enumerate}
     Let $\alpha_{n,j}$ and $\beta_{n,j}$ be such that the conformal map $\wt\phi_{n,j}:=  \alpha_{n,j} \phi_{n,j} +\beta_{n,j}$ sends the vertices 0 and 1 to 0 and $\eta(1)$. By property (iii) above, $(\alpha_{n,j},\beta_{n,j})\to (1,0)$ in probability. Then $\varphi^{n}\circ(\wt\phi_{n,j})^{-1}$ is a conformal map whose domain contains $B(0;\psi_j'(0)^{-1}\fR_j/2)$ with probability tending to 1 as $n\to\infty$, and fixes 0 and $\eta(1)$. Therefore by Proposition~\ref{prop:unif-sph}, $\varphi^{n}\circ(\wt\phi_{n,j})^{-1}$ converges uniformly on compacts to the identity in probability. This along with~\eqref{eq:pf:thm:unif-cone-2} gives Assertions 1.\ and 2., which further imply that for each $r>0$, under the event $B(0;2r)\subseteq\eta([-M,M])$, with probability tending to 1 as $n\to\infty$, the surface formed by the submap of $\cG^{1/n}$ restricted to vertices in $[-M,M]$ under the embedding $\varphi^n$ contains $B(0;r)$ as in the proof of Theorem~\ref{thm:pack-cone}. Assertions 3.\ and 4.\ then follows. 
\end{proof}

\bibliographystyle{alpha}
\bibliography{circle}

\end{document}